\documentclass[12pt]{amsart}

\usepackage{amssymb}
\usepackage{amsmath, graphicx, rotating}
\usepackage{color}
\usepackage{soul}

\usepackage[T1]{fontenc}
\usepackage{lmodern}
\usepackage[english]{babel}

\usepackage{ upgreek }
\usepackage{stmaryrd}
\SetSymbolFont{stmry}{bold}{U}{stmry}{m}{n}
\usepackage{amsthm}
\usepackage{float}

\usepackage{ bbm }
\usepackage{ stmaryrd }
\usepackage{ mathrsfs }
\usepackage{ frcursive }
\usepackage{ comment }

\usepackage{pgf, tikz}
\usetikzlibrary{shapes}
\usepackage{varioref}
\usepackage{enumitem}

\definecolor{rouge}{rgb}{0.7,0.00,0.00}
\definecolor{vert}{rgb}{0.00,0.5,0.00}
\definecolor{bleu}{rgb}{0.00,0.00,0.8}

\newtheorem{theorem}{Theorem}[section]
\newtheorem{lemma}[theorem]{Lemma}

\newtheorem{corollary}[theorem]{Corollary}
\newtheorem{proposition}[theorem]{Proposition}
\newtheorem{hypothesis}{Hypothesis M\kern-0.1mm}
\labelformat{hypothesis}{\textbf{M\kern-0.1mm#1}}
\newtheorem{hypothesisH}[theorem]{Hypothesis}
\theoremstyle{definition}
\newtheorem{example}[theorem]{Example}
\newtheorem{remark}[theorem]{Remark}

\usepackage[colorlinks=true, linkcolor=blue, citecolor=blue]{hyperref}

\numberwithin{equation}{section}

\newcommand*{\abs}[1]{\left\lvert#1\right\rvert}
\newcommand*{\norm}[1]{\left\lVert#1\right\rVert}
\newcommand*{\pent}[1]{\left\lfloor#1\right\rfloor}
\newcommand*{\sachant}[2]{\left.#1 \,\middle|\,#2\right.}
\newcommand*{\scal}[2]{\left\langle {#1}, {#2} \right\rangle}

\def\bb#1{\mathbb{#1}}
\def\bs#1{\boldsymbol{#1}}
\def\geq{\geqslant}
\def\leq{\leqslant}

\newcommand\ee{\varepsilon}
\DeclareMathOperator{\GL}{GL}
\DeclareMathOperator{\dd}{d\!}
\DeclareMathOperator{\e}{e}
\DeclareMathOperator{\sh}{sh}
\DeclareMathOperator{\ch}{ch}
\DeclareMathOperator{\Var}{Var}

\DeclareMathOperator{\supp}{supp}

\begin{document}
\title[Conditioned limit theorems for Markov walks]{Limit theorems for Markov walks conditioned to stay positive under a spectral gap assumption}

\author{Ion Grama}
\curraddr[Grama, I.]{Universit\'{e} Bretagne Sud, LMBA UMR CNRS 6205,
Vannes, France}
\email{ion.grama@univ-ubs.fr}
\author{Ronan Lauvergnat}
\curraddr[Lauvergnat, R.]{Universit\'{e} Bretagne Sud, LMBA UMR CNRS 6205,
Van\-nes, France}
\email{ronan.lauvergnat@univ-ubs.fr}
\author{\'Emile Le Page}
\curraddr[Le Page, \'E.]{Universit\'{e} Bretagne Sud, LMBA UMR CNRS 6205,
Vannes, France}
\email{emilepierre.lepage@gmail.com}

\begin{abstract}
Consider a Markov chain $(X_n)_{n\geq 0}$ with values in the state space $\bb X$. 
Let $f$ be a real function on $\bb X$ and set $S_0=0,$ $S_n = f(X_1)+\cdots + f(X_n),$ $n\geq 1$. 
Let $\bb P_x$ be the probability measure generated by the Markov chain starting at $X_0=x$.
For a starting point $y \in \bb R$ denote by $\tau_y$ the first moment  when the Markov walk $(y+S_n)_{n\geq 1}$ becomes non-positive. 
Under the condition that $S_n$ has zero drift, we find the asymptotics of the probability $\bb P_x \left( \tau_y >n  \right)$ 
and of the conditional law $\bb P_x \left(  \sachant{y+S_n\leq \cdot\sqrt{n}}{ \tau_y >n }  \right)$ as $n\to +\infty.$ 
\end{abstract}

\date{\today }
\subjclass[2010]{Primary 60G50, 60F05, 60J50. Secondary 60J05, 60J70, 60G40}
\keywords{Markov chain, random walk, exit time, harmonic function, limit theorem}

\maketitle

\section{Introduction}

Assume that on the probability space $\left(\Omega, \mathscr F, \bb P\right)$ we are given a sequence of random variables 
$\left(X_n\right)_{n\geq 1}$ with values in a measurable space $\bb X.$ Let $f$ be a real function on $\bb X.$ 
Suppose that the random walk $S_n = f(X_1)+\cdots + f(X_n),$ $n\geq 1$ has zero drift. 
For a starting point $y\in \bb R$ denote by $\tau_y$ the time at which $\left(y+S_n\right)_{n\geq 1}$ first passes into the interval $(-\infty,0].$   
We are interested in determining the asymptotic behaviour of the probability $\bb P (\tau_y>n)$ and of the 
conditional law of $ \frac{y+S_n}{\sqrt{n}}$ given the event $\left\{ \tau_y >n \right\}=\left\{ S_1>0,\dots,S_n>0 \right\}$
as $n\to +\infty.$

The case when $f$ is the identity function and $\left(X_n\right)_{n\geq 1}$ 
are i.i.d.\  in $\bb X = \bb R$ has been extensively studied in the literature.  
We refer to 
Spitzer \cite{spitzer_principles_2013}, 
Iglehart \cite{iglehart_functional_1974,iglehart_random_1974},   
Bolthausen \cite{bolthausen_functional_1976},
Doney \cite{doney_asymptotic_1989}, 
Bertoin and Doney \cite{bertoin_conditioning_1994},
Borovkov \cite{borovkov_asymptotic_2004,borovkov_asymptotic_2004-1},
Caravenna \cite{caravenna_local_2005},
Vatutin and Wachtel \cite{vatutin_local_2008} to cite only a few. 
Recent progress has been made for random walks with independent increments in $\bb X = \bb R^d$, 
see 
Eichelbacher and K\"onig \cite{eichelsbacher_ordered_2008},  
Denisov and Wachtel \cite{denisov_random_2015, denisov_conditional_2010} and
Duraj \cite{duraj_random_2014}. 
However, to the best of our knowledge, the case of the Markov chains has been treated only in some special cases.
Upper and lower bounds for $\bb P (\tau_y>n)$ have been obtained in Varapoulos \cite{varopoulos_potential_1999}, \cite{varopoulos_potential_2000} for Markov chains with bounded jumps and in Dembo, Ding and Gao \cite{dembo_persistence_2013} for integrated random walks based on independent increments. 
An approximation of $\bb P\left( \tau_y >n \right)$ by the survival probability of the Brownian motion for Markov walk under moment conditions is given in Varopoulos \cite{varopoulos_potential_2001}.
Exact asymptotics are obtained in Presman \cite{presman_1967,presman_1969} in the case of sums of random variables defined on a finite Markov chain under the additional assumption that the distributions have an absolute continuous component and in Denisov and Wachtel  \cite{denisov_exit_2015} for integrated random walks. 
The case of products of i.i.d.\  random matrices 
which reduces to the study of a particular Markov chain defined on a merely compact state space was 
considered in  \cite{grama_conditioned_2016}
and the case of affine walks in $\bb R$ has been treated in  \cite{GLLP_affine_2016}.
 
In this paper we give the asymptotics of the probability of the exit time $\tau_y$ and of the law of $y+S_n$ conditioned to stay positive
for a Markov chain under the assumption that its transition operator has a spectral gap. In particular our results cover the case of Markov chains with compact state spaces, the affine random walks in $\bb R$ (see \cite{GLLP_affine_2016}) and $\bb R^d$ (see Gao, Guivarc'h and Le Page \cite{gao_stable_2015}). Our results apply also to the case of sums of i.i.d.\  random variables.

To present briefly the main results of the paper denote by $\bb P_x$ and $\bb E_x$ the probability 
and the corresponding expectation generated by the trajectories of a Markov chain 
$\left(X_n\right)_{n\geq 1}$ with the initial state $X_0=x \in \bb X.$ 
Let $\mathbf Q$ be the transition operator of the Markov chain $\left(X_n, y+S_n\right)_{n\geq 1}$ and let $\mathbf Q_{+}$ be the restriction of $\mathbf Q$ on $\bb X \times \bb R^{+}_{*}$. We show that under appropriate assumptions, there exists a $\mathbf Q_+$-harmonic function, say $V$, which is positive on a domain $\mathscr D_+(V) \subseteq \bb X \times \bb R$ and $0$ on its complement such that, for any $(x,y) \in \mathscr D_+(V)$,
\begin{equation}
	\label{intro-001}
	\bb P_x \left( \tau_y > n \right) \underset{n\to +\infty}{\sim} \frac{2V(x,y)}{\sqrt{2\pi n}\sigma}
\end{equation}
and  
\[
\bb P_x \left( \sachant{\frac{y+S_n}{ \sigma \sqrt{n}} \leq t }{\tau_y >n} \right) \underset{n\to+\infty}{\longrightarrow} \mathbf \Phi^+(t),
\]
where $\mathbf \Phi^+(t) = 1-\e^{-\frac{t^2}{2}}$ is the Rayleigh distribution function and $\sigma$ is a positive real.
On the complement of $\mathscr D_+(V)$ we find that
\begin{equation}
	\bb P_x \left( \tau_y > n \right) \leq c_x e^{-c n},
\label{intro-002}
\end{equation}
where $c_x$ depends on $x$ and $c$ is a constant. Moreover, we obtain uniform versions of \eqref{intro-001} and \eqref{intro-002}. 
We give an example of a Markov chain for which the bound  \eqref{intro-002} is attained. 
This is different from the case of sums of i.i.d.\ random variables where on the complement of $\mathscr D_+(V)$ it holds
$\bb P_x \left( \tau_y > n \right)=0.$ 
For details we refer to Section \ref{sec-not-res}.

The asymptotics of the probability of the exit time  $\bb P (\tau_y>n)$ for walks in $\bb R$  are usually obtained by the Wiener-Hopf factorization 
(see Feller \cite{feller_introduction_1971}). 
Eichelbacher and K\"onig \cite{eichelsbacher_ordered_2008} and
Denisov and Wachtel \cite{denisov_random_2015}
have developed an alternative approach for obtaining the asymptotics of $\bb P (\tau_y>n)$ 
for random walks with independent increments in $\bb R^d.$ 
To study the case of Markov chains we mainly rely upon the developments made in \cite{eichelsbacher_ordered_2008}, \cite{denisov_random_2015} 
for the independent case 
and in the work of authors \cite{grama_conditioned_2016} and \cite{GLLP_affine_2016} 
for two particular cases of Markov chains.
We also make use  of the strong approximation result for Markov chains 
obtained in \cite{ion_grama_rate_2014} with explicit constants
depending on the properties of the transition operator of the Markov chain and  on its initial state.

To carry out the approach developed in \cite{denisov_random_2015} from the independent case 
to the case of a Markov chain it is necessary to refine it substantially by taking into account the 
dependence on its initial state $x \in \bb X$, which is one of the major difficulties
of this paper.
We assume that the transition operator of the Markov chain satisfies a spectral gap condition on some associated Banach space, 
which implies that  for any function $g$ in this space we have, for any $x \in \bb X,$
\begin{equation}
	\label{ideedecexp}
	\bb E_x \left( \abs{ g \left( X_n \right)} \right) = c + e^{-cn} N(x),
\end{equation}
where $N(x)$ is a function carrying the dependence on the initial state $x$
(see Section \ref{sec-not-res} for details).
 The relation \eqref{ideedecexp} ensures that the dependence on the initial state decreases exponentially fast. In the present paper it is supposed essentially that 
the property \eqref{ideedecexp} is extended to other functions than those in the Banach space
(see Hypothesis \ref{Momdec} in the text section).
In Section \ref{Applications} we show that the conditions we impose are verified 
for a stochastic recursion in $\bb R^d$ and
for Markov chains with compact state space.

The paper is organized as follows.
In Section \ref{Mart Approx} we approximate the walk by an appropriate martingale and state some properties on this martingale and on associated exit times. In Section \ref{CMWI} we prove that the martingale killed at a special exit time has a uniformly bounded expectation. This result implies in particular that the sequence $( \bb E_x ((y+S_n) \mathbbm 1_{\{ \tau_y >n \}} ) )_{n\geq 0}$ is bounded. Using the results of Sections \ref{Mart Approx} and \ref{CMWI}, we establish in Section \ref{Sec Harm Func} the existence of a $\mathbf Q_+$-harmonic function and prove in Section \ref{PosHaFun} that this function is non identically zero. In Section \ref{AsExTi}, we determine the asymptotic of the probability $\bb P_x ( \tau_y > n )$ and in Section \ref{AsCondMarkWalk} we prove that the conditional law of $(y+S_n)/(\sigma\sqrt{n})$ given the event $\{\tau_y > n\}$ converges to the Rayleigh distribution.
 
We end this section by agreeing on some basic notations. For the rest of the paper the symbol $c$ denotes a positive constant depending on the all previously introduced constants. Sometimes, to stress the dependence of the constants on some parameters 
 $\alpha,\beta,\dots$ we shall use the notations $ c_{\alpha}, c_{\alpha,\beta},\dots$. All these constants are likely to change their values every occurrence. For any real numbers $u$ and $v$, denote by $u \wedge v=\min(u,v)$ the minimum between $u$ and $v$. The indicator of an event $A$ is denoted by $\mathbbm 1_A$. For any bounded measurable function $f$ on $\bb X$, random variable $X$ in $\bb X$ and event $A$, the integral $\int_{\bb X}  f(x) \bb P (X \in \dd x, A)$ means the expectation $\bb E\left( f(X); A\right)=\bb E \left(f(X) \mathbbm 1_A\right)$.

\section{Notations and results}
\label{sec-not-res}

On the probability space $\left(\Omega, \mathscr F, \bb P\right)$ consider a Markov chain $(X_n)_{n\geq 0}$ taking values in the measurable state space $(\bb X, \mathscr{X})$. For any given $x \in \bb X,$ denote by $\mathbf P (x,\cdot)$ its transition probability, to which we associate the transition operator
\[
\mathbf Pg (x) = \int_{\bb X} g(x') \mathbf P (x,\dd x'),
\]
for any of complex bounded measurable function  $g$  on $\bb X$. 
Denote by $\bb P_x$ and $\bb E_x$ the probability and the corresponding expectation generated by the finite dimensional distributions of the Markov chain $(X_n)_{n\geq 0}$ starting at $X_0=x.$ 
We remark that 
$\mathbf{P} g \left( x\right) = \bb E_x \left(g\left( X_1 \right)\right)$ and $\mathbf{P}^n g \left( x\right) = \bb E_x \left(g\left( X_n \right)\right)$  
for any $g$ complex bounded measurable, $x \in \bb X$ and $n\geq 1.$

Let $f$ be a real valued function defined on the state space $\bb X$ and let $\mathscr{B}$ be a Banach space of complex valued functions on $\bb X$ endowed with the norm $\norm{\cdot}_{\mathscr{B}}$. Let $\norm{\cdot}_{\mathscr{B} \to \mathscr{B}}$ be the operator norm on $\mathscr{B}$ and let $\mathscr{B}'=\mathscr{L}\left( \mathscr{B},\bb C \right)$ be the topological dual of $\mathscr{B}$ endowed with the norm $\norm{\varphi}_{\mathscr{B}'}=\sup_{h \in \mathscr{B}} \frac{\abs{\varphi(h)}}{\norm{h}_{\mathscr{B}}}$, for any $\varphi \in\mathscr{B}'$.  Denote by $e$ the unit function of $\bb X$: $e(x) = 1$, for any $x\in \bb X$ and by $\bs \delta_x$ the Dirac measure at $x\in \bb X$: $\bs \delta_x(g) = g(x)$, for any $g \in \mathscr{B}$.

Following \cite{ion_grama_rate_2014}, we assume the following hypotheses.

\begin{hypothesis}[Banach space]\ 
\label{BASP}
\begin{enumerate}[ref=\arabic*, leftmargin=*, label=\arabic*.]
	\item \label{BASP001} The unit function $e$ belongs to $\mathscr{B}$.
	\item \label{BASP002} For any $x\in \bb X$, the Dirac measure $\bs \delta_x$ belongs to $\mathscr{B}'$.
	\item \label{BASP003} The Banach space $\mathscr{B}$ is included in $L^1\left( \mathbf{P}(x,\cdot) \right)$, for any $x\in \bb X$.
	\item \label{BASP004} There exists a constant $\ee_0 \in (0,1)$ such that for any $g\in \mathscr{B}$, the function $\e^{itf}g$ is in $\mathscr{B}$ for any $t$ satisfying $\abs{t} \leq \ee_0$.
\end{enumerate}
\end{hypothesis}

Under the point \ref{BASP003} of \ref{BASP}, $\mathbf{P}g(x)$ exists for any $g\in \mathscr{B}$ and $x\in \bb X$.

\begin{hypothesis}[Spectral gap]\ 
\label{SPGA}
\begin{enumerate}[ref=\arabic*, leftmargin=*, label=\arabic*.]
	\item \label{SPGA001} The map $g \mapsto \mathbf{P}g$ is a bounded operator on $\mathscr{B}$.
	\item \label{SPGA002} There exist constants $C_Q >0$ and $\kappa \in (0,1)$ such that
	\[
	\mathbf{P} = \Pi+Q,
	\]
	where $\Pi$ is a one-dimensional projector and $Q$ is an operator on $\mathscr{B}$ satisfying $\Pi Q=Q\Pi=0$ and for any $n\geq 1$,
	\[
	\norm{Q^n}_{\mathscr{B}\to\mathscr{B}} \leq C_Q \kappa^n.
	\]
\end{enumerate}
\end{hypothesis}

Since $\Pi$ is a one-dimensional projector and $e$ is an eigenvector of $\mathbf{P}$, there exists a linear form $\bs \nu \in \mathscr{B}'$, such that for any $g \in \mathscr{B}$,
\begin{equation}
\Pi g = \bs \nu(g) e.
\label{linear form}
\end{equation}
When Hypotheses \ref{BASP} and \ref{SPGA} hold, we set $\mathbf{P}_tg := \mathbf{P}\left( \e^{itf}g \right)$ for any $g \in \mathscr{B}$ and $t \in [-\ee_0,\ee_0]$. In particular $\mathbf{P}_0=\mathbf{P}$.

\begin{hypothesis}[Perturbed transition operator]\ 
\label{PETO}
\begin{enumerate}[ref=\arabic*, leftmargin=*, label= \arabic*.]
	\item \label{PETO001} For any $\abs{t} \leq \ee_0$ the map $g \mapsto \mathbf{P}_tg$ is a bounded operator on $\mathscr{B}$.
	\item \label{PETO002} There exists a constant $C_{\mathbf{P}} >0$ such that, for any $n\geq 1$ and $\abs{t} \leq \ee_0$,
	\[
	\norm{\mathbf{P}_t^n}_{\mathscr{B} \to \mathscr{B}} \leq C_{\mathbf{P}}.
	\]
\end{enumerate}
\end{hypothesis}

To control the dependence on $x$ of the Markov chain $\left( X_n \right)_{n\geq0}$, the following hypothesis is a little more demanding than the one of \cite{ion_grama_rate_2014}.

\begin{hypothesis}[Local integrability]\ 
\label{Momdec}
The Banach space $\mathscr{B}$ contains a sequence of real non-negative functions $N, N_1, N_2, \dots $  such that:
\begin{enumerate}[ref=\arabic*, leftmargin=*, label=\arabic*.]
\item \label{Momdec001} There exist $\alpha > 2$ and $\gamma > 0$ such that, for any $x\in \bb X$,
	\[
	\max \left\{ \abs{f(x)}^{1+\gamma}, \norm{\bs \delta_x}_{\mathscr{B}'}, \bb E_x^{1/\alpha} \left( N\left( X_n \right)^{\alpha} \right) \right\} \leq c \left( 1+N(x) \right)
	\]
and
	\[
	N(x) \mathbbm 1_{\{ N(x) > l\}} \leq N_l(x), \quad \text{for any} \quad l\geq 1.
	\]
\item \label{Momdec002} There exists $c > 0$ such that, for any $l\geq 1$,
	\[
	\norm{N_l}_{\mathscr{B}} \leq  c.
	\]
\item \label{Momdec003} There exists $\beta>0$ and $c > 0$ such that, for any $l\geq 1$,
	\[
	 \abs{\bs \nu \left( N_l \right)} \leq \frac{c}{l^{1+\beta}}.
	\]
\end{enumerate}
\end{hypothesis}

Under Hypotheses \ref{BASP}, \ref{SPGA} and \ref{Momdec}, we have, for any $x\in \bb X$ and $n\geq 0$,
\begin{align}
	\bb E_x \left( N(X_n) \right) &= \bs \nu(N) + Q^n N(x) \nonumber\\
	&\leq \abs{\bs \nu(N)} + \norm{Q^n}_{\mathscr{B}\to\mathscr{B}} \norm{N}_{\mathscr{B}} \norm{\bs \delta_x}_{\mathscr{B}'} \nonumber\\
	 &\leq c + \e^{-cn} N(x)
	\label{decexpN}
\end{align}
and, in the same way, for any $x\in \bb X$, $l\geq 1$ and $n \geq 0$,
\begin{equation}
	\label{decexpNl}
	\bb E_x \left( N_l \left( X_n \right) \right) \leq \frac{c}{l^{1+\beta}} + \e^{-c n} \left( 1+N(x) \right).
\end{equation}

Remark that in Hypothesis \ref{Momdec} the function $f$ need not belong to the Banach space $\mathscr B$.

A consequence of Hypotheses \ref{BASP}-\ref{Momdec} is the following proposition (cf.\ \cite{ion_grama_rate_2014}). For any $x \in \bb X,$ set $\mu_\alpha(x)=\sup_{n\geq 1}\bb E_x^{1/\alpha} \left( \abs{f\left( X_n \right)}^{\alpha} \right)$.

\begin{proposition}
\label{MomAs}
Assume that the Markov chain $\left(X_n\right)_{n\geq 0}$ and the function $f$ satisfy Hypotheses \ref{BASP}-\ref{Momdec}.
\begin{enumerate}[ref=\arabic*, leftmargin=*, label=\arabic*.]
	\item \label{MomAs001} There exists a constant $\mu$ such that, for any $x\in \bb X$ and $n \geq 0$,
	\[
	\abs{ \bb E_x \left( f(X_n) \right) - \mu } \leq \e^{-c  n} \left( 1+\mu_{\alpha}(x)^{1+\gamma}  + \norm{\bs \delta_x}_{\mathscr{B}'} \right).
	\]
	\item \label{MomAs002} There exists a constant $\sigma \geq 0$ such that, for any $x\in \bb X$ and $n \geq 1$,
	\[
	\underset{m\geq 0}{\sup} \abs{ \Var_{x} \left( \sum_{k={m+1}}^{m+n} f(X_k) \right) - n \sigma^2 } \leq c \left( 1+1+\mu_{\alpha}(x)^{2+2\gamma}  + \norm{\bs \delta_x}_{\mathscr{B}'} \right),
	\]
	where $\Var_x$ is the variance under $\bb P_x$.
\end{enumerate}
\end{proposition}

We do not assume the existence of the stationary probability measure. If a stationary probability measure $\bs \nu'$ satisfying $\bs \nu' \left( N^2 \right) < +\infty$ exists then, under Hypotheses \ref{BASP}-\ref{Momdec}, we have that $\bs \nu' = \bs \nu$ is necessarily unique and it holds (see \cite{ion_grama_rate_2014})
\begin{equation}
	\label{mu-sigma001}
	\bs \nu(f) = \mu \quad\text{and}\quad \sigma^2 = \int_{\bb R^d} f^2(x) \bs \nu(\dd x) + 2\sum_{n=1}^{+\infty} \int_{\bb R^d} f(x) \mathbf{P}^n f(x) \bs \nu(\dd x).
\end{equation}

\begin{hypothesis}[Centring and non-degeneracy]\ 
\label{CECO}
We suppose that the constants $\mu$ and $\sigma$ defined in Proposition \ref{MomAs} satisfy $\mu=0$ and $\sigma > 0$.
\end{hypothesis}

Using this assumption and the point \ref{Momdec001} of \ref{Momdec} we have $\mu_{\alpha}(x) \leq c\left(1+N(x)^{\frac{1}{1+\gamma}} \right)$ and therefore, for any $x\in \bb X$ and $n \geq 0$,
\begin{equation}
	\label{bound_EfXn}
	\abs{ \bb E_x \left( f(X_n) \right) } \leq \e^{-c  n} \left( 1+ N(x) \right).
\end{equation}
	
Let $y \in \bb R$ be a starting point and $(y+S_n)_{n\geq 0}$ be the Markov walk defined by $S_n := \sum_{k=1}^n f\left( X_k \right)$, $n\geq 1$ with $S_0=0$. Denote by $\tau_y$ the first moment when $y+S_n$ becomes non-positive:
\[
\tau_y := \inf \left\{ k\geq 1, \, y+S_k \leq 0 \right\}.
\]
It is shown in Corollary \ref{Exitfinit} that for any $y \in \bb R$ and $x\in \bb X$, the stopping time $\tau_y$ is $\bb P_x$-a.s.\ finite. The asymptotic behaviour of the probability $\bb P_x \left( \tau_y >n  \right)$ is determined by the harmonic function which we proceed to introduce. For any $(x,y) \in \bb X \times \bb R$, denote by $\mathbf Q(x,y,\cdot)$ the transition probability of the Markov chain $(X_n,y+S_n)_{n\geq 0}$. The restriction of the measure $\mathbf Q(x,y,\cdot)$  on $\bb X \times \mathbb R^*_+$ is defined by
\[
\mathbf{Q}_+(x,y,B) = \mathbf{Q}(x,y,B)
\]
for any measurable set  $B$ on $\bb X \times \bb R_+^*$ and  for any $(x,y) \in \bb X \times \bb R$. For any bounded measurable function $\varphi: \bb X \times \bb R \to \bb R$ set $\mathbf{Q}_+\varphi (x,y)=\int_{\bb X \times \bb R_+^*} \varphi(x',y') \mathbf{Q}_+(x,y,\dd x' \times \dd y')$. A function $V: \bb X \times \bb R \to \bb R$ is said to be $\mathbf{Q}_+$-harmonic if
\[
\mathbf{Q}_+V (x,y)  = V(x,y), \qquad \text{for any } (x,y) \in \bb X \times \bb R.
\]
In the sequel, we deal only with non-negative harmonic functions. 
For any non-negative function $V$, denote by $\mathscr{D}_+(V)$ the set where $V$ is positive,
\[
\mathscr{D}_+(V) := \{ (x,y) \in \bb X \times \bb R, V(x,y)> 0 \}
\]
and by $\mathscr{D}_+(V)^c$ its complement, \textit{i.e.}\ the set where $V$ is $0$. For any $\gamma >0$, introduce the following set
\[
\mathscr{D}_{\gamma} := \left\{ (x,y) \in \bb X \times \bb R, \; \exists n_0 \geq 1,\; \bb P_x \left( y+S_{n_0} > \gamma \left( 1+N \left( X_{n_0} \right) \right) \,,\, \tau_y > n_0 \right) > 0 \right\}.
\]

The following assertion proves the existence of a non-identically zero harmonic function.
\begin{theorem}
\label{thonV}
Assume Hypotheses \ref{BASP}-\ref{CECO}.
\begin{enumerate}[ref=\arabic*, leftmargin=*, label=\arabic*.]
	\item \label{thonV001} For any $x\in \bb X$, $y\in \bb R$, the sequence $\left( \bb E_x \left( y+S_n \,;\, \tau_y > n \right) \right)_{n\geq 0}$ converges to a real number $V(x,y)$:
	\[
	\bb E_x \left( y+S_n \,;\, \tau_y > n \right) \underset{n\to +\infty}{\longrightarrow} V(x,y). 
	\]
	\item \label{thonV002} The function $V$: $\bb X \times \bb R \to \bb R$, defined in the previous point is $\mathbf{Q}_+$-harmonic, \textit{i.e.}\  for any $x\in \bb X$, $y\in \bb R,$
	\[
	\mathbf{Q}_+ V(x,y) = \bb E_x \left( V\left( X_1, y+S_1 \right) \,;\, \tau_y > 1 \right) = V(x,y).
	\]
	\item \label{thonV003} For any $x\in \bb X$, the function $V(x,\cdot)$ is  non-negative and non-decreasing on $\bb R$ and
	\[
	\underset{y\to+\infty}{\lim} \frac{V(x,y)}{y} = 1.
	\]
	Moreover, for any $\delta > 0$, $x \in \bb X$ and $y \in \bb R$,
	\[
	\left( 1-\delta \right) \max(y,0) - c_{\delta} \left( 1+N(x) \right) \leq V(x,y) \leq \left( 1+\delta \right) \max(y,0) + c_{\delta} \left( 1+N(x) \right).
	\]
	\item \label{thonV004} There exists $\gamma_0>0$ such that, for any $\gamma \geq \gamma_0$, 
	\[
	\mathscr{D}_+(V) = \mathscr{D}_{\gamma}.
	\]
\end{enumerate}
\end{theorem}

The following result gives the asymptotic of the exit probability for fixed $(x,y) \in \bb X \times \bb R$.

\begin{theorem}
\label{thontau}
Assume Hypotheses \ref{BASP}-\ref{CECO}.
\begin{enumerate}[ref=\arabic*, leftmargin=*, label=\arabic*.]
	\item \label{thontau001} For any $(x,y) \in \mathscr{D}_+(V)$,
	\[
	\bb P_x \left( \tau_y > n \right) \underset{n\to +\infty}{\sim} \frac{2V(x,y)}{\sqrt{2\pi n} \sigma}.
	\]
	\item \label{thontau002} For any $(x,y) \in \mathscr{D}_+(V)^c$ and $n\geq 1$,
	\[
	\bb P_x \left( \tau_y > n \right) \leq e^{-c n} \left( 1+N(x) \right).
	\]
\end{enumerate}
\end{theorem}

Now we complete the point \ref{thontau001} of the previous theorem by some estimations.
\begin{theorem}
\label{thontau2}
Assume Hypotheses \ref{BASP}-\ref{CECO}.
\begin{enumerate}[ref=\arabic*, leftmargin=*, label=\arabic*.]
	\item \label{thontau001bis} There exists $\ee_0 >0$ such that, for any $\ee \in (0,\ee_0)$, $n\geq 1$ and $(x,y)\in \bb X\times  \bb R$,
	\[
	\abs{\bb P_x \left( \tau_y > n \right) - \frac{2V(x,y)}{\sqrt{2\pi n} \sigma}} \leq c_{\ee}\frac{\max(y,0) +  \left( 1+ y\mathbbm 1_{\{y> n^{1/2-\ee}\}} +N(x) \right)^2}{n^{1/2+\ee/16}}.
	\]
	\item \label{thontau003} Moreover, for any $(x,y) \in \bb X \times \bb R$ and $n\geq 1$,
	\[
	\bb P_x \left( \tau_y > n \right) \leq c\frac{ 1 + \max(y,0) + N(x) }{\sqrt{n}}.
	\] 
\end{enumerate}
\end{theorem}

Finally, we give the asymptotic of the conditional law of $y+S_n.$
\begin{theorem}
\label{loideRayleigh}
Assume Hypotheses \ref{BASP}-\ref{CECO}. 
\begin{enumerate}[ref=\arabic*, leftmargin=*, label=\arabic*.]
	\item \label{AAA001} 
For any $(x,y) \in \mathscr{D}_+(V)$ and $t\geq 0$,
\[
\bb P_x \left( \sachant{\frac{y+S_n}{\sigma \sqrt{n}} \leq t }{\tau_y >n} \right) \underset{n\to+\infty}{\longrightarrow} \mathbf \Phi^+(t),
\]
where $\mathbf \Phi^+(t) = 1-\e^{-\frac{t^2}{2}}$ is the Rayleigh distribution function.
\item \label{AAA002} 
Moreover there exists $\ee_0 >0$ such that, for any $\ee \in (0,\ee_0)$, $n\geq 1$, $t_0 > 0$, $t\in [ 0, t_0 ]$ and $(x,y)\in \bb X \times \bb R$,
\begin{align*}
&\abs{\bb P_x \left( y+S_n \leq t \sqrt{n} \,,\, \tau_y > n \right) -\frac{2V(x,y)}{\sqrt{2\pi n}\sigma} \mathbf \Phi^+\left(\frac{t}{\sigma}\right)} \\
&\hspace{5cm} \leq c_{\ee,t_0}\frac{\max(y,0) +  \left( 1+ y\mathbbm 1_{\{y> n^{1/2-\ee}\}} +N(x) \right)^2}{n^{1/2+\ee/16}}.
\end{align*}
\end{enumerate}
\end{theorem}

We now comment on Theorems \ref{thonV} and \ref{thontau}.  

\begin{remark}
	\label{Demboité}
	The sets $(\mathscr{D}_{\gamma})_{\gamma>0}$ are nested: for any $\gamma_1 \leq \gamma_2$, we have $\mathscr{D}_{\gamma_1} \supseteq \mathscr{D}_{\gamma_2}.$ Moreover, by the point \ref{thonV004} of Theorem \ref{thonV}, 
the sets $\mathscr{D}_{\gamma}$ are equal to $\mathscr{D}_+(V)$ for all $\gamma$ large enough.
\end{remark}

\begin{remark}
	\label{dominclus}
	The set $\mathscr{D}_+(V)$ it is not empty. More precisely there exists $\gamma_1 > 0$ such that
	\[
	\{ (x,y) \in \bb X \times \bb R,\; y > \gamma_1 \left( 1+N(x) \right) \} \subseteq \mathscr{D}_+(V).
	\]
	Example \ref{recstoacst} and Figure \ref{Domdepos} illustrate this property.
\end{remark}

\begin{remark}
	When $( X_n )_{n\geq 1}$ are i.i.d., it is well known that $\bb P_x \left( \tau_y > n \right) = 0$ for any $(x,y) \in \mathscr{D}_+(V)^c$. 
	When the sequence $( X_n )_{n\geq 1}$ from a Markov hain, instead of this property, we have the bound of the point \ref{thontau002} of Theorem \ref{thontau}. Moreover there exist some Markov walks for which this exponential bound is attained. This remark is developed in Example \ref{seulexemple}.
\end{remark}

\begin{example}[Random walks in $\bb R$]
	When $( X_n )_{n\geq 1}$ are i.i.d.\  real random variables of mean $0$ and positive variance with finite absolute moments of order $p>2$, one can take $N=0$ and therefore
	\[
	\mathscr{D}_{\gamma} := \left\{ y \in \bb R, \; \exists n_0 \geq 1, \; \bb P \left( y+S_{n_0} > \gamma  \,,\, \tau_y > n_0 \right) > 0 \right\}.
	\]
	Since the walk is allowed to increase at each step with positive probability, it follows that $\bb P \left( y+S_{n_0} > \gamma  \,,\, \tau_y > n_0 \right) > 0$ if and only if $\bb P \left( \tau_y > 1 \right)= \bb P \left( y+X_1 >0 \right) > 0$. Thus, $[0,+\infty] \subseteq ( -\max \supp( \bs \mu ) , +\infty ) = \mathscr{D}_{\gamma} = \mathscr{D}_+(V)$ for every $\gamma >0$, where $\bs \mu$ is the common law of $X_n$ and $\supp ( \bs \mu )$ is its support.
\end{example}

The following example is intended to illustrate Remark \ref{dominclus}.

\begin{example}
\label{recstoacst}
Consider the following special case of the one dimensional stochastic recursion: $X_{n+1} = a_{n+1} X_n + b_{n+1}$ where $(a_i)_{i\geq 1}$ and $(b_i)_{i\geq 1}$ are two independent sequences of i.i.d.\ random variables.
In this example we consider that the law of $a_i$ is $\frac{1}{2} \bs \delta_{\{-1/2\}}+\frac{1}{2} \bs \delta_{\{1/2\}}$ and that of $b_i$ is uniform on $[-1,1]$.
The state space $\bb X$ is $\bb R$ and the function $N$ is given by $N(x) = \abs{x}^{1+\ee}$ for some $\ee >0$ (see \cite{GLLP_affine_2016} or Section \ref{MarcheaffineRd} for a construction of an appropriate Banach space and for the proof that \ref{BASP}-\ref{CECO} are verified for the stochastic recursion).
One can verify that the domain of positivity of the function $V$ is 
$\mathscr{D}_+ (V) = \{ (x,y) \in \bb R^2,\; y > -\frac{\abs{x}}{2} -1 \} = \mathscr{D}_{\gamma}$, for all $\gamma > 0$. 
Obviously, $\{ (x,y) \in \bb X \times \bb R,\; y > \frac{1}{2} \left( 1+\abs{x}^{1+\ee} \right) \} \subseteq \mathscr{D}_+(V),$ see Figure \ref{Domdepos}.
\end{example}

\begin{figure}[ht]
	\begin{center}
	\begin{tikzpicture}[scale=0.8]
	\fill[black!10] (-5, 1.5) -- (0,-1) -- (5, 1.5) -- (5, 5) -- (-5, 5) -- cycle;
	\fill[black!40] plot [domain=-5:5] (\x,{exp(1.2*ln(abs(\x)))/2+1/2}) -- (5,5) -- (-5,5) -- cycle;
	\draw[thick][domain=-5:5] plot(\x,{exp(1.2*ln(abs(\x)))/2+1/2});
	\draw(-2.5,4) node{$y>\frac{1}{2}(\abs{x}^{1+\ee}+1)$};
	\draw(-3,1.3) node{$\mathscr{D}_+(V)$};
	\draw(-3,-1) node{$\mathscr{D}_+(V)^c$};
	\draw[very thick,->](-5,0)--(5,0)node[below left]{$x$};
	\draw[very thick,->](0,-2.5)--(0,5)node[below left]{$y$};
	\draw(0,0) node[below left]{$0$};
	\draw[thick](-5,1.5)--(0,-1);
	\draw[thick](0,-1)--(5,1.5);
	\draw[thick](-5,-0.1)--(-5,0.1);
	\draw[thick](-4,-0.1)--(-4,0.1);
	\draw[thick](-3,-0.1)--(-3,0.1);
	\draw[thick](-2,-0.1)--(-2,0.1);
	\draw[thick](-1,-0.1)--(-1,0.1);
	\draw[thick](1,-0.1)--(1,0.1);
	\draw[thick](2,-0.1)--(2,0.1);
	\draw[thick](3,-0.1)--(3,0.1);
	\draw[thick](4,-0.1)--(4,0.1);
	\draw[thick](-0.1,-2)--(0.1,-2);
	\draw[thick](-0.1,-1)--(0.1,-1);
	\draw[thick](-0.1,1)--(0.1,1);
	\draw[thick](-0.1,2)--(0.1,2);
	\draw[thick](-0.1,3)--(0.1,3);
	\draw[thick](-0.1,4)--(0.1,4);
	\end{tikzpicture}
	\caption{\label{Domdepos}}
	\end{center}
\end{figure}

The next example is intended to show that the inequality of the point \ref{thontau002} of Theorem \ref{thontau} is attained.

\begin{example}
\label{seulexemple}
Consider the Markov walk $(X_n)_{n\geq 0}$ living on the finite state space $\bb X :=\{-1\,;\,1\,;\,-3\,;\,7/6\}$ 
with the transition probabilities given in Figure \ref{graphedemarkov}.
\begin{figure}[ht]
\begin{center}
\begin{tikzpicture}[scale=0.8]
\tikzset{ville/.style={draw,minimum width=1cm,circle,very thick,fill=black!25},
chemin/.style={very thick,->,>=latex}}
\node[ville] (P-1) at (0,0) {$-1$};
\node[ville] (P1) at (4,0) {$1$};
\node[ville] (P-3) at (2,-3) {$-3$};
\node[ville] (Pa) at (8,-3) {$\frac{7}{6}$};
\coordinate (n) at (10,-3);
\node at (10.5,-3) {$1/2$};
\draw[chemin] (P-1) to [bend left] node[midway,above]{$1/2$} (P1);
\draw[chemin] (P-1) to [bend right] node[midway,left]{$1/2$} (P-3);
\draw[chemin] (P1) to [bend left] node[midway,below]{$1/2$} (P-1);
\draw[chemin] (P1) to [bend left] node[midway,right]{$1/2$} (P-3);
\draw[chemin] (P-3) to [bend right] node[midway,below]{$1$} (Pa);
\draw[chemin] (Pa.north east) to [out=45,in=90] (n) to [out=-90,in=-45] (Pa.south east);
\draw[chemin] (Pa) to [bend right] node[midway,above]{$1/2$} (P1);
\end{tikzpicture}
\end{center}
\caption{\label{graphedemarkov}}
\end{figure}
Suppose that $f$ is the identity function on $\bb X$. 
It is easy to see that the assumptions stated in Remark \ref{casfini} of Section \ref{Compact2}  
are satisfied and thereby so are Hypotheses \ref{BASP}-\ref{CECO}. 
Now, when $x=1$ and $y\in (1,3]$ or when $x =-1$ and $y \in (-1,2]$, 
one can check that the Markov walk $y+S_n$ stays positive if and only if the values of the the variables $X_i$ 
alternate between
$1$ and $-1$ and therefore, for such starting points $(x,y)$, we have $\bb P_x \left( \tau_y > n \right) = \left( \frac{1}{2} \right)^n$.
This shows that, when the random variables $\left( X_n \right)_{n\geq 1}$ form a Markov chain, 
the survival probability $\bb P_x \left( \tau_y > n \right)$ has an asymptotic behaviour
different from that in the independent case where it can be either equivalent to 
$\frac{c_{x,y}}{\sqrt{n}}$ or $0$.

In this example we can make explicit the set $\mathscr{D}_+(V)$. Since $N=0$, we notice that the function $V$ is positive if and only if there exists an integer $n\geq 1$ such that $\bb P_x \left( y+S_n > \gamma \,,\, \tau_y > n \right)>0$ for a $\gamma$ large enough. This is possible only if the chain can reach the state $X_n = 7/6$ within a trajectory of $\left( y+S_k \right)_{n \geq k\geq 1}$ which stays positive, \textit{i.e.}\ $\bb P_x \left( X_n = 7/6 \,,\, \tau_y > n \right) > 0$. Consequently
\begin{align*}
	\mathscr{D}_+(V) &= \{ -1 \} \times (2,+\infty) \cup \{ 1 \} \times (3,+\infty) \cup \{ -3,7/6 \} \times (-7/6,+\infty) \\
	&= \mathscr{D}_3 = \left\{ (x,y) \in \bb X \times \bb R, \; \exists n \geq 1, \; \bb P_x \left( y+S_n > 3 \,,\, \tau_y > n \right) > 0 \right\}.
\end{align*}

To sum up, this model presents the three possible asymptotic behaviours of $\bb P_x \left( \tau_y >n \right)$: for any $(x,y) \in \mathscr{D}_+(V)= \{ -1 \} \times (2,+\infty) \cup \{ 1 \} \times (3,+\infty) \cup \{ -3,7/6 \} \times (-7/6,+\infty)$,
\[
\bb P_x \left( \tau_y >n \right) \underset{n\to+\infty}{\sim} \frac{2V(x,y)}{\sqrt{2\pi n}\sigma},
\]
for any $(x,y) \in \{ -1 \} \times (-1,2] \cup \{ 1 \} \times (1,3]$ and $n\geq 1$,
\[
\bb P_x \left( \tau_y >n \right) = \left(\frac{1}{2}\right)^n,
\]
for any $(x,y) \in \{ -1 \} \times (-\infty,-1] \cup \{ 1 \} \times (-\infty,1] \cup \{ -3,7/6 \} \times (-\infty,-7/6]$ and $n\geq 1$,
\[
\bb P_x \left( \tau_y >n \right) = 0.
\]
\end{example}

\section{Applications}
\label{Applications}

We illustrate the results of Section \ref{sec-not-res} by considering three particular models.  

\subsection{Affine random walk in \texorpdfstring{$\bb R^d$}{} conditioned to stay in a half-space}
\label{MarcheaffineRd}

Let $d\geq 1$ be an integer and $( g_n )_{n\geq 1} = ( A_n, B_n )_{n\geq 1}$ be a sequence of i.i.d.\ random elements in $\GL\left( d, \bb R \right) \times \bb R^d$ following the same distribution $\bs \mu$. Let $( X_n )_{n \geq 0}$ be the Markov chain on $\bb R^d$ defined by
\[
X_0 = x \in \bb R^d, \qquad \qquad X_{n+1} = A_{n+1} X_n + B_{n+1}, \quad n \geq 1.
\]
Set $S_n  = \sum_{k=1}^n f\left( X_k \right)$, $n\geq 1,$ where the function $f(x) = \scal{u}{x}$ is the projection of the vector $x \in \bb R^d$ on the direction defined by the vector $u\in \bb R^d \smallsetminus \{0\}$. For any $y\in \bb R$, consider the first time when the random walk $\left(y+S_n\right)_{n\geq 1}$ becomes non-positive: 
\[
\tau_y = \inf \{ k \geq 1, \; y+S_k \leq 0 \}.
\]
This stopping time coincides with the entry time of the affine walk $\left( \sum_{k=1}^n X_k  \right) _{n \geq 0}$ in the closed half-subspace $\{ s \in \bb R^d, \scal{u}{s} \leq -y \}$.

Introduce the following hypothesis.
\begin{hypothesisH}\ 
\label{hypoH}
\begin{enumerate}[ref=\arabic*, leftmargin=*, label=\arabic*.]
	\item \label{H1} There exists a constant $\delta>0$, such that
	\[
	\bb E \left( \abs{A_1}^{2+2\delta} \right) < +\infty, \quad \bb E \left( \abs{B_1}^{2+2\delta} \right) < +\infty
	\]
	and
	\[
	k(\delta) = \underset{n\to+\infty}{\lim} \bb E^{1/n} \left( \abs{A_nA_{n-1} \dots A_1}^{2+2\delta} \right) < 1.
	\]
	\item \label{H2} There is no proper affine subspace of $\bb R^d$ which is invariant with respect to all the elements of the support of $\bs \mu$.
	\item \label{H3} For any vector $v_0 \in \bb R^d \smallsetminus \{0\}$,
	\[
	\bb P \left( {}^t\!A_1^{-1} v_0 = {}^t\!A_2^{-1} v_0 \right) < 1.
	\] 
	\item \label{H4} The vector $B_1$ is centred: $\bb E \left( B_1 \right) = 0$.
\end{enumerate}
\end{hypothesisH}

\begin{proposition}
\label{PP001}
Under Hypothesis \ref{hypoH}, Theorems \ref{thonV}, \ref{thontau}, \ref{thontau2} and \ref{loideRayleigh} hold true for the affine random walk conditioned to stay in the half-subspace $\{ s \in \bb R^d, \scal{u}{s} \leq 0 \}$. 
\end{proposition}

Proposition \ref{PP001} is proved in Appendix \ref{proof-rec-sto-Mat}  where we construct an appropriate Banach space $\mathscr{B}$
and show that Hypotheses \ref{BASP}-\ref{CECO} are satisfied with $N(x) = \abs{x}^{1+\ee}$, for some $\ee > 0$.

\begin{remark}
	The set $\mathscr{D}_+(V)$ depends on the law of $(A_i,B_i)$. In the case when $A_i$ are independent of $B_i$ and the support of the law of $\scal{u}{B_i}$ contains a sequence converging to $+\infty$, one can verify that $\mathscr{D}_+(V) = \bb R^d \times \bb R$.
\end{remark}

\subsection{Two components Markov chains in compact sets under the Doeblin-Fortet condition}
\label{Compact1}

Let $(X,d_{X})$ be a compact metric space, $\mathscr{C}\left( X \right)$ and $\mathscr{L}\left( X \right)$ be the spaces of continuous and Lipschitz 
complex functions on $X$,  respectively. Define
\[
\abs{h}_{\infty} = \sup_{x \in X} \abs{h(x)}, \quad \forall h \in \mathscr{C}\left( X \right)
\]
and
\[
\left[ h \right]_{X} = \sup_{\substack{(x,y) \in X \\ x\neq y}} \frac{\abs{h(x)-h(y)}}{d_{X}(x,y)}, \quad \forall h \in \mathscr{L}\left( X \right).
\]
We endow $\mathscr{C}\left( X \right)$ with the uniform norm $\abs{\cdot}_{\infty}$ and $\mathscr{L}\left( X \right)$ with the norm $\abs{\cdot}_{\mathscr{L}} = \abs{\cdot}_{\infty} + \left[ \cdot \right]_{X}$, respectively. 
Consider the space $\bb X := X \times X$ with the metric $d_{\bb X}$ on $\bb X$ defined by $d_{\bb X} ( (x_1, x_2) , (y_1,y_2) ) = d_{X} ( x_1,y_1 ) + d_{X} ( x_2,y_2 ),$ for any $(x_1, x_2)$ and $(y_1,y_2)$ in $\bb X$. Denote by $\mathscr{L}\left( \bb X \right)$ the space of the Lipschitz complex function on $\bb X$ endowed with the norm $\norm{\cdot}_{\mathscr{L}} = \norm{\cdot}_{\infty} + \left[ \cdot \right]_{\bb X}$, where
\[
\norm{h}_{\infty} = \sup_{x \in \bb X} \abs{h(x)}, \quad \forall h \in \mathscr{C}\left( \bb X \right)
\]
and
\[
 \left[ h \right]_{ \bb X} = \sup_{\substack{(x,y) \in \bb X \\ x\neq y}} \frac{\abs{h(x)-h(y)}}{d_{\bb X}(x,y)}, \quad \forall h \in \mathscr{L}\left( \bb X \right).
\]
Following Guivarc'h and Hardy \cite{guivarch_theoremes_1988}, consider a Markov chain $\left( \chi_n \right)_{n\geq 0}$ on $X$ with transition probability $P$. Let $\left( X_n \right)_{n\geq 0}$ be the Markov chain on $\bb X$ defined by $X_n = \left( \chi_{n-1}, \chi_{n} \right)$, $n\geq 1$ and $X_0=\left( 0, \chi_0 \right)$:  its transition probability is given by
\[
\mathbf P((x_1,x_2), \dd y_1 \times \dd y_2) = \bs \delta_{x_2} \left( \dd y_1 \right) P\left( x_2, \dd y_2 \right).
\]
For a fixed real function $f$ on $\bb X$, let $S_n := \sum_{k=1}^n f\left( X_n \right)$ be the associated Markov walk and, for any $y\in \bb R$, let $\tau_y := \inf \left\{ n \geq 1, \, y+S_n \leq 0 \right\}$ be the associated exit time.

In order to apply the results stated in the previous section, we need some hypotheses on the function $f$ and the operator $P$ on $\mathscr{C}(X)$ defined by $Ph(x) = \int_{X} h(y) P(x,\dd y)$ for any $x\in X$ and any $h \in \mathscr{C}(X)$.

\begin{hypothesisH}\ 
\label{hypoHcompactI}
\begin{enumerate}[ref=\arabic*, leftmargin=*, label=\arabic*.]
	\item \label{H1compactI} For any $h$ in $\mathscr{C} \left( X \right),$ respectively in $\mathscr{L} \left( X \right)$, the function $Ph$ is an element of $\mathscr{C} (X),$ respectively of $\mathscr{L} (X)$.
	\item \label{H2compactI} There exist constants $n_0 \geq 1$, $0 < \rho < 1$ and $C>0$ such that, for any function $h \in \mathscr{L}(X)$, we have
	\[
	\abs{P^{n_0}h}_{\mathscr{L}} \leq \rho \abs{h}_{\mathscr{L}} + C \abs{h}_{\infty}
	\]
	\item \label{H3compactI} The unique eigenvalue of $P$ of modulus $1$ is $1$ and the associated eigenspace is generated by the function $e$: $x\mapsto 1$, \textit{i.e.}\ if there exist $\theta \in \bb R$ and $h \in \mathscr{L}(X)$ such that	$Ph = \e^{i\theta} h$, then $h$ is constant and $\e^{i\theta} = 1$.
\end{enumerate}
\end{hypothesisH}

Under Hypothesis \ref{hypoHcompactI}, one can check that conditions (a), (b), (c) and (d) of Chapter 3 in Norman \cite{norman1972markov} hold true and we can apply the theorem of Ionescu Tulcea and Marinescu \cite{tulcea_theorie_1950} (see also \cite{guivarch_theoremes_1988}). Coupling this theorem with the point \ref{H3compactI} of Hypothesis \ref{hypoHcompactI} we obtain the following proposition.

\begin{proposition}\
\label{TroupourPpasgras} 
\begin{enumerate}[ref=\arabic*, leftmargin=*, label=\arabic*.]
	\item There exists a unique $P$-invariant probability $\nu$ on $X$. 
	\item For any $n\geq 1$ and $h \in \mathscr{L}(X)$,
	\[
	P^nh = \nu(h) + R^nh,
	\]
	where $R$ is an operator on $\mathscr{L}(X)$ with a spectral radius $r(R)<1$.
\end{enumerate}
\end{proposition}

Suppose that $f$ and $\nu$ satisfy the following hypothesis.

\begin{hypothesisH}\ 
\label{hypoHcompactII}
\begin{enumerate}[ref=\arabic*, leftmargin=*, label=\arabic*.]
	\item \label{H1compactII} The function $f$ belongs to $\mathscr{L}(\bb X)$.
	\item \label{H2compactII} The function $f$ is centred, in the sense that
	\[
	\int_{\bb X} f(x,y) P(x,\dd y) \nu(\dd x) = 0.
	\]
	\item \label{H3compactII} The function $f$ is non-degenerated, that means that there is no function $h\in \mathscr{L}(X)$ such that 
	\[
	f(x,y) = h(x) - h(y),
	\]
	for $P_{\nu}$-almost all $(x,y),$ where $P_{\nu}(\dd x \times \dd y) = P(x,\dd y)\nu(\dd x)$.
\end{enumerate}
\end{hypothesisH}

Assuming Hypotheses \ref{hypoHcompactI} and \ref{hypoHcompactII}, Guivarc'h and Hardy \cite{guivarch_theoremes_1988} have established that 
the sequence $\left( S_n / \sqrt{n} \right)_{n\geq 1}$ converges weakly to a centred Gaussian random variable of variance $\sigma^2 >0$,
under the probability $\bb P_x$  generated by the finite dimensional distributions of the Markov chain $(X_n)_{n\geq 0}$ starting at $X_0=x$, 
for any $x\in X$. Moreover, under the same hypotheses, we show in Appendix \ref{proof-cas-compact} that \ref{BASP}-\ref{CECO} are satisfied, thereby proving the following assertion.
\begin{proposition}
\label{PP002}
Under Hypotheses \ref{hypoHcompactI} and \ref{hypoHcompactII}, Theorems \ref{thonV}, \ref{thontau}, \ref{thontau2} and \ref{loideRayleigh} hold true for the Markov chain $(X_n)_{n\geq 1}$, the function $f$ and the Banach space $\mathscr{L}(\bb X)$. 
\end{proposition}

\subsection{Markov chains in compact sets under spectral gap assumptions}
\label{Compact2}

In the previous Section \ref{Compact1}, we considered a Markov chain with two components satisfying the Doeblin Fortet condition and proved, inter alia, that this chain has a spectral gap (Hypothesis \ref{SPGA}). In this section, for Markov chains with values in a compact set, we give  more general conditions which ensure the applicability of the results of the previous section. 

Let $\left( \bb X, d \right)$ be a compact metric space and $\left( X_n \right)_{n\geq 0}$ be a Markov chain living in $\bb X$. Denote by $\mathbf P$ the transition probability of $\left( X_n \right)_{n\geq 0}$ and by $\mathscr{C}(\bb X)$ the Banach algebra of the continuous complex functions on $\bb X$ endowed with the uniform norm
\[
\abs{h}_{\infty} = \underset{x \in \bb X}{\sup} \abs{h(x)}, \qquad h \in \mathscr{C}(\bb X).
\]
Consider a real function $f$ defined on $\bb X$, the transition operator $\mathbf P$ on $\mathscr{C}(\bb X)$ associated to the transition probability of $\left( X_n \right)_{n\geq 0}$ and the unit function $e$ defined on $\bb X$ by $e(x) =1$, for any $x \in \bb X$.
\begin{hypothesisH}\ 
\label{hypoHcompactIII}
\begin{enumerate}[ref=\arabic*, leftmargin=*, label=\arabic*.]
	\item \label{H1compactIII} For any $h\in \mathscr{C}(\bb X)$, the function $\mathbf Ph$ is an element of $\mathscr{C}(\bb X)$.
	\item \label{H2compactIII} The operator $\mathbf P$ has a unique invariant probability $\bs \nu$.
	\item \label{H3compactIII} For any $n\geq 1$,
	\[
	\mathbf P^n = \Pi + Q^n,
	\]
	where $\Pi$ is the one-dimensional projector on $\mathscr{C}(\bb X)$ defined by $\Pi(h) = \bs \nu(h) e$, for any $h \in \mathscr{C}(\bb X)$, $Q$ is an operator on $\mathscr{C}(\bb X)$ of spectral radius $r(Q) < 1$ satisfying $\Pi Q= Q \Pi = 0$.
	\item \label{H4compactIII} The function $f$  belongs to $\mathscr{C}(\bb X)$ and is $\bs \nu$-centred, \textit{i.e.}\  $\bs \nu(f) = 0$.
	\item \label{H5compactIII} The function $f$ is non-degenerated, that is there is no function $h\in \mathscr{C}(\bb X)$ such that
	\[
	f(X_1) = h(X_0) - h(X_1), \qquad \bb P_{\bs \nu}\text{-a.s.,}
	\]
	where $\bb P_{\bs \nu}$ is the probability generated by the finite dimensional distributions of the Markov chain $(X_n)_{n\geq 0}$ when the initial law of $X_0$ is $\bs \nu$.
\end{enumerate}
\end{hypothesisH}

Consider the Markov walk $S_n = \sum_{k=1}^n f(X_k)$. 
It is well known, that under Hypothesis \ref{hypoHcompactIII} the normalized sum 
$S_n/\sqrt{n}$ converges in law to a centred normal distribution of variance $\sigma^2 > 0$
with respect to the probability $\bb P_x$ generated by the finite dimensional distributions of the Markov chain $(X_n)_{n\geq 0}$ starting at $X_0=x$, 
for any $x\in \bb X$.

\begin{proposition}
\label{PP003}
Under Hypothesis \ref{hypoHcompactIII}, Theorems \ref{thonV}, \ref{thontau}, \ref{thontau2} and \ref{loideRayleigh} 
hold true for the Markov chain $(X_n)_{n\geq 1}$, the function $f$ and the Banach space $\mathscr{C}(\bb X)$. 
\end{proposition}

All the elements of the proof are contained in the proof of Proposition \ref{PP002} (see Appendix \ref{proof-cas-compact}), 
which therefore is left to the reader.

\begin{remark}
\label{casfini}
As a special example of the compact case, consider the
Markov chain $(X_n)_{n\geq 1}$ taking values in a finite space $\bb X.$
Assume that $(X_n)_{n\geq 1}$ is aperiodic and irreducible with transition matrix $\mathbf P$.
Let $f$ be a finite function on $\bb X$.
We shall verify Hypotesis \ref{hypoHcompactIII}. 
The Banach space $\mathscr B$ consists of all finite real functions on $\bb X,$ 
therefore condition \ref{H1compactIII} is obvious.
Moreover, there is a unique invariant measure $\bs \nu,$ which proves condition \ref{H2compactIII}. 
According to Perron-Frobenius theorem, the transition matrix $\mathbf P$ admits $1$ as the only simple eigenvalue of modulus $1$, which implies 
condition \ref{H3compactIII}.
Assume in addition that $\bs \nu(f)=0$ and   
that there exists a path $x_0,\dots,x_n$ in $\bb X$ such that 
$\mathbf P(x_0,x_1)>0,\dots,\mathbf P(x_{n-1},x_n)>0,\mathbf P(x_n,x_0)>0$ 
and $f(x_0)+\dots +f(x_n)\not= 0$ (conditions \ref{H4compactIII} and \ref{H5compactIII} respectively).
As a consequence of Proposition \ref{PP003} 
it follows that the asymptotics as $n\to+\infty$ of the probability $\bb P_x \left( \tau_y >n \right)$ and  of the conditional law 
$\bb P_x \left( \sachant{y+S_n \leq \cdot \sqrt{n}}{\tau_y >n} \right)$ are given by 
Theorems \ref{thontau} and \ref{loideRayleigh}, respectively.
\end{remark}

\section{Martingale approximation}
\label{Mart Approx}
All over Sections \ref{Mart Approx}-\ref{AsCondMarkWalk} we assume Hypotheses \ref{BASP}-\ref{CECO}. 
The aim of Sections \ref{Mart Approx}-\ref{PosHaFun} is to prove the existence 
and the positivity of the harmonic function claimed in Theorem \ref{thonV}. 
To summarize the approach, 
we approximate the walk $(y+S_n)_{n\geq 1}$ by a martingale $(z+M_n)_{n\geq 1}$ in Section \ref{Mart Approx} 
and prove a difficult result on the uniform boundedness in $n$ of the expectation $\bb E_x \left( y+S_n \,;\, \tau_y > n \right)$
in Section \ref{CMWI}.
The key for doing this is the introduction of two stopping times $T_{z}$ and $\hat{T}_{z}$ defined by \eqref{twostopingtimes001}. The existence and the positivity of the harmonic function on a non-empty set are proved 
in Sections \ref{Sec Harm Func} and \ref{PosHaFun}, respectively.

It is well known that the Poisson equation
\[
\Theta - \mathbf P \Theta = f
\]
admits as a solution the real valued function $\Theta$ defined for any $x\in \bb X$ by
\[
\Theta(x)  = f(x) + \sum_{k=1}^{+\infty} \mathbf P^k f(x).
\]
For any $x\in \bb X$, let
\[
r(x) = \mathbf P \Theta(x) = \Theta(x) - f(x) = \sum_{k=1}^{+\infty} \mathbf P^k f(x).
\]
From \eqref{bound_EfXn} we deduce the following assertion.
\begin{lemma}
\label{MTR}
The functions $\Theta$ and $r$ exist on $\bb X$ and for any $x \in \bb X$,
\[
\abs{\Theta(x)} \leq c \left( 1+N(x) \right) \qquad \text{and} \qquad \abs{r(x)} \leq c \left( 1+N(x) \right).
\]
\end{lemma}

Following Gordin \cite{gordin_central_1969},  define the process $\left( M_n \right)_{n\geq 0}$ by setting $M_0=0$ and, for any $n\geq 1$,
\[
M_n = \sum_{k=1}^{n} \left[ \Theta \left(X_k\right) - \mathbf P \Theta \left(X_{k-1}\right) \right] = \sum_{k=1}^{n} \left[ \Theta \left(X_k\right) - r \left(X_{k-1}\right) \right].
\]
It is easy to see that, for any $x \in \bb X,$ the process $(M_n)_{n\geq 0}$ is a zero mean $\bb P_x$-martingale for the natural filtration $\left(\mathscr{F}_n\right)_{n \geq 0}$, where $\mathscr{F}_n$ is the $\sigma$-algebra generated by $X_1,\,X_2,\dots,\,X_n$ and $\mathscr{F}_0$ the trivial $\sigma$-algebra. 
Denote by $\xi_n$ the increments of the martingale $(M_n)_{n\geq 0}$: for any $n\geq 1$, 
\[
\xi_n := \Theta\left(X_n\right) - r\left(X_{n-1}\right).
\]
In the sequel it will be convenient to consider the martingale $(z+M_n)_{n\geq 1}$ starting at 
\[
z=y+r(x).
\]
The reason for this is the following approximation which is an easy consequence of the definition of the martingale $(z+M_n)_{n\geq 1}$: 
for any $x \in \bb X$ and $y\in \bb R,$ on the event $\Omega,$ 
\begin{equation}
	\label{decMSX}
	z+M_n = y+r(x)+\sum_{k=1}^n \left[ r \left( X_k \right) + f \left( X_k \right) - r \left( X_{k-1} \right) \right] = y + S_n + r \left( X_n \right).
\end{equation}

\begin{lemma}\ 
\label{majmart}
\begin{enumerate}[ref=\arabic*, leftmargin=*, label=\arabic*.]
	\item \label{majmart001} For any $p \in [1,\alpha]$, $x\in \bb X$ and $n \geq 1$,
	\[
	\bb E_x^{1/p} \left( \abs{M_n}^p \right) \leq c_p \sqrt{n} \left( 1 +  N(x) \right).
	\]
	\item \label{majmart002} For any $x\in \bb X$ and $n \geq 1$,
	\[
	\bb E_x \left( \abs{M_n} \right) \leq c \left( \sqrt{n} +  N(x) \right).
	\]
\end{enumerate}
\end{lemma}

\begin{proof}
First we control the increments $\xi_n$. By Lemma \ref{MTR}, for any $n\geq 1$,
\begin{equation}
	\label{majdesxi000}
	\abs{\xi_n} \leq c\left( 1 + N\left( X_n \right) + N\left( X_{n-1} \right) \right).
\end{equation}
So, using the point \ref{Momdec001} of Hypothesis \ref{Momdec} and \eqref{decexpN}, for any $n\geq 1$,
\begin{align}
	\label{Majdesxi001} \bb E_x^{1/p} \left( \abs{\xi_n}^p \right) &\leq c_p \left( 1+N(x) \right) \qquad \forall p \in [1,\alpha],\\
	\label{Majdesxi002} \bb E_x \left( \abs{\xi_n} \right) &\leq c  + \e^{-c n} N(x).
\end{align}

\textit{Proof of the claim \ref{majmart001}}. By Burkholder's inequality, for $2< p \leq \alpha$,
\[
\bb E_x^{1/p} \left( \abs{M_n}^p \right) \leq c_p \norm{ \left( \sum_{k=1}^n \xi_k^2 \right)^{1/2} }_p = c_p \bb E_x^{1/p} \left( \left( \sum_{k=1}^n \xi_k^2 \right)^{p/2} \right).
\]
Using H\"older's inequality with the exponents $u=p/2>1$ and $v=\frac{p}{p-2}$, we obtain
\[
\bb E_x^{1/p} \left( \abs{M_n}^p \right) \leq c_p \bb E_x^{1/p} \left[ \left( \sum_{k=1}^n \xi_k^{2u} \right)^{\frac{p}{2u}} n^{\frac{p}{2v}} \right] = c_p n^{\frac{p-2}{2p}} \left(\sum_{k=1}^n \bb E_x \left[ \abs{\xi_k}^{p} \right] \right)^{1/p}.	
\]
From \eqref{Majdesxi001}, for any $p\in (2,\alpha]$,
\begin{equation}
	\label{PMM}
	\bb E_x^{1/p} \left( \abs{M_n}^p \right) \leq c_p n^{\frac{p-2}{2p}} \left( \sum_{k=1}^n c_p \left( 1+N(x) \right)^p \right)^{1/p} \leq c_p \sqrt{n} \left( 1+N(x) \right).
\end{equation}
Using the Jensen inequality for $p\in [1,2]$, we obtain the claim \ref{majmart001}.

\textit{Proof of the claim \ref{majmart002}}. Consider $\ee \in (0,1/2)$. By \eqref{Majdesxi002},
\begin{align*}
	\bb E_x \left( \abs{M_n} \right) &\leq  \sum_{k=1}^{\pent{n^\ee}} \bb E_x \left( \abs{\xi_k} \right) + \bb E_x \left( \abs{M_n-M_{\pent{n^\ee}}} \right)\\
	&\leq c n^\ee + c N(x) + \bb E_x \left( \abs{M_n-M_{\pent{n^\ee}}} \right).
\end{align*}
Since $(X_n, M_n)_{n\geq 0}$ is a Markov chain, by the Markov property, the claim \ref{majmart001} and \eqref{decexpN},
\begin{align*}
	\bb E_x \left( \abs{M_n} \right) &\leq c n^\ee + c N(x) + \bb E_x \left( \bb E\left( \sachant{\abs{M_n-M_{\pent{n^\ee}}}}{\mathscr{F}_{\pent{n^\ee}}} \right) \right) \\
	&\leq c n^\ee + c N(x) + \bb E_x \left[ c \left(n-\pent{n^\ee}\right)^{1/2} \left( 1+N \left( X_{\pent{n^\ee}} \right) \right) \right] \\
	&\leq c \sqrt{n} + c_{\ee} N(x).
\end{align*}
\end{proof}
A key point in the proof of the existence and of the positivity of the harmonic function
is the introduction of two stopping times. The first one is the first time when the martingale $(z+M_n)_{n\geq 1}$ becomes non-positive, say $T_z$, and the second one is the first time, after the time $\tau_y,$ when the martingale $(z+M_n)_{n\geq 1}$ becomes non-positive, say $\hat{T}_z$. Precisely, for any $x\in \bb X$, $z\in \bb R$ and $y=z-r(x)$, set
\begin{equation}
T_z := \inf \left\{ k \geq 1,\, z+M_k \leq 0 \right\} \quad \text{and} \quad \hat{T}_z := \inf \left\{ k \geq \tau_y,\, z+M_k \leq 0 \right\}.
\label{twostopingtimes001}
\end{equation}

The following lemmas will be useful in the next sections.

\begin{lemma}
\label{Mnsubmartingale}
For any $x\in \bb X$, $z \in \bb R$, the sequence $\left( (z+M_n) \mathbbm 1_{\left\{ \hat{T}_z > n \right\}} \right)_{n\geq 0}$ is a $\bb P_x$-submartingale.
\end{lemma}

\begin{proof}
For any $n\geq 0$,
\begin{align*}
	\bb E_x &\left( \sachant{z+M_{n+1} \,;\, \hat{T}_z > n+1}{\mathscr{F}_n} \right) \\
	&\quad= \bb E_x \left( \sachant{z+M_{n+1} \,;\, \hat{T}_z > n}{\mathscr{F}_n} \right) - \bb E_x \left( \sachant{z+M_{n+1} \,;\, \hat{T}_z = n+1}{\mathscr{F}_n} \right)\\
	&\quad= \left( z+M_n \right) \mathbbm 1_{\left\{ \hat{T}_z > n \right\}} - \bb E_x \left( \sachant{z+M_{\hat{T}_z} \,;\, \hat{T}_z = n+1}{\mathscr{F}_n} \right).
\end{align*}
By the definition of $\hat{T}_z$ we have $z+M_{\hat{T}_z} < 0$ a.s.\ and the result follows.
\end{proof}

The proof of the following Lemma shows how to apply the Markov property with the event $\left\{\hat{T}_z >n\right\}.$  
The same approach will be used repeatedly in the case of more complicated functionals, as for exemple
$\bb E_x \left( z+M_n \,;\, \hat{T}_z > n \right)$, without giving the details.  
\begin{lemma}
\label{MarkovpropforTz}
For any $x\in \bb X$, $z \in \bb R$, $n\geq 0$, $k\leq n$ and $y=z-r(x)$,
\begin{align*}
	\bb P_x &\left( \hat{T}_z > n \right) = \int_{\bb X \times \bb R} \bb P_{x'} \left( \hat{T}_{z'} > n-k \right) \bb P_x \left( X_k \in \dd x' \,,\, z+M_k \in \dd z' \,,\, \tau_y > k \right) \\
	&+ \int_{\bb X \times \bb R} \bb P_{x'} \left( T_{z'} > n-k \right) \bb P_x \left( X_k \in \dd x' \,,\, z+M_k \in \dd z' \,,\, \tau_y \leq k \,,\, \hat{T}_{z} > k \right).
\end{align*}
\end{lemma}

\begin{proof}
For any $k\leq n$, we have
\[
\bb P_x \left( \hat{T}_z > n \right) = \bb P_x \left( \tau_y > n \right) + \sum_{i=1}^{n-k} \bb P_x \left( \tau_y =i+k \,,\, \hat{T}_z > n \right) + \bb P_x \left( \tau_y \leq k \,,\, \hat{T}_z > n \right).
\]
By the Markov property and \eqref{decMSX}, with $y'=z'-r(x')$,
\begin{align*}
	\bb P_x \left( \hat{T}_z > n \right) =\;& \int_{\bb X \times \bb R} \bb P_{x'} \left( \tau_{y'} > n-k \right) \bb P_x \left( X_k \in \dd x' \,,\, z+M_k \in \dd z' \,,\, \tau_y > k \right) \\
	&+ \sum_{i=1}^{n-k} \int_{\bb X \times \bb R} \bb P_{x'} \left( \tau_{y'} =i \,,\, z'+M_{i} > 0 \,,\, \dots \,,\, z'+M_{n-k} > 0 \right) \\
	&\hspace{2cm} \times\bb P_x \left( X_k \in \dd x' \,,\, z+M_k \in \dd z' \,,\, \tau_y > k \right) \\
	&+ \int_{\bb X \times \bb R} \bb P_{x'} \left( T_{z'} > n-k \right) \bb P_x \left( X_k \in \dd x' \,,\, z+M_k \in \dd z' \,,\, \tau_{y} \leq k \,,\, \right. \\
	&\hspace{6cm} \left. z+M_{\tau_y} > 0 \,,\, \dots \,,\, z+M_k > 0 \right).
\end{align*}
Putting together the first two terms we get the result.
\end{proof}

\section{Integrability of the killed martingale and of the killed Markov walk}
\label{CMWI}

First we give  a bound of order $n^{1/2-2\ee}$ of the expectation of the martingale $(z+M_n)_{n\geq 0}$ killed at $T_z$ and a similar bound when the martingale is killed at $\hat{T}_z$. 
From these two results we will deduce a uniform in $n$ bound for the second expectation, \textit{i.e.}\ for the expectation of the martingale $(z+M_n)_{n\geq 0}$ killed at $\hat{T}_z$.
We will conclude the section by showing that the expectation of the Markov walk $(y+S_n)_{n\geq 0}$ killed at $\tau_y$ is also bounded uniformly in $n$.

\begin{lemma}
\label{firstupperboundforMnTz}
There exists $\ee_0>0$ such that, for any $\ee \in (0,\ee_0)$, $x\in \bb X$, $z\in \bb R$ and $n \in \bb N$, it holds
\[ 
\bb E_x \left( z+M_n \,;\, T_z > n \right) \leq \max(z,0) + c_{\ee} \left( n^{1/2-2\ee} + N(x) \right).
\]
\end{lemma}

\begin{proof}
Using the fact that $( M_n )_{n\geq 0}$ is a zero mean martingale and the optional stopping theorem,
\[
\bb E_x \left( z+M_n \,;\, T_z > n \right) = z - \bb E_x \left( z + M_n \,;\, T_z \leq n \right) = z - \bb E_x \left( z + M_{T_z} \,;\, T_z \leq n \right).
\]
By the definition of $T_z$, on the event $\left\{T_z > 1\right\}$, we have
\[
\xi_{T_z} = z + M_{T_z} - \left( z+M_{T_z-1} \right) < z + M_{T_z}  \leq 0.
\]
Using this inequality and \eqref{majdesxi000}, we obtain
\begin{align}
	\bb E_x \left( z+M_n \,;\, T_z > n \right) &\leq z \bb P_x \left( T_z > 1 \right) + \bb E_x \left( \abs{\xi_1} \,;\, T_z=1 \right) + \bb E_x \left( \abs{\xi_{T_z}} \,;\, 1 < T_z \leq n \right) \nonumber\\
	&\leq \max(z,0) + c\bb E_x \left( 1 + N\left( X_{T_z} \right) + N\left( X_{T_z-1} \right) \,;\, T_z \leq n \right).
	\label{xitoN}
\end{align}
We bound $\bb E_x \left( N\left( X_{T_z} \right) \,;\, T_z \leq n \right)$ as follows. Let $\ee$ be a real number in $(0,1/6)$ and set $l=\pent{n^{1/2-2\ee}}$. Using the point \ref{Momdec001} of Hypothesis \ref{Momdec} we write
\begin{align*}
	\bb E_x \left( N\left( X_{T_z} \right) \,;\, T_z \leq n \right)	\leq n^{1/2-2\ee} &+ \bb E_x \left( N\left( X_{T_z} \right) \,;\, N\left( X_{T_z} \right) > n^{1/2-2\ee} \,,\, T_z \leq n \right) \\
	\leq n^{1/2-2\ee} &+ \sum_{k=1}^{\pent{n^{\ee}}} \bb E_x \left( N\left( X_k \right) \right) + \sum_{k=\pent{n^{\ee}}+1}^n \bb E_x \left( N_l \left( X_k \right) \right).
\end{align*}
By \eqref{decexpN} and \eqref{decexpNl},
\[
\bb E_x \left( N\left( X_{T_z} \right) \,;\, T_z \leq n \right)	\leq c n^{1/2-2\ee} + c N(x) + \frac{cn}{l^{1+\beta}} + \e^{-cn^{\ee}} \left( 1+N(x) \right).
\]
Choosing $\ee < \min(\frac{\beta}{4(2+\beta)},\frac{1}{6})$, we find that
\begin{equation}
	\label{NXT}
	\bb E_x \left( N\left( X_{T_z} \right) \,;\, T_z \leq n \right) \leq c_{\ee} n^{1/2-2\ee}+ c_{\ee} N(x).
\end{equation}
In the same manner, we obtain that $\bb E_x \left( N\left( X_{T_z-1} \right) \,;\, T_z \leq n \right) \leq c_{\ee} n^{1/2-2\ee}+ c_{\ee} N(x)$. Consequently, from \eqref{NXT} and \eqref{xitoN}, we conclude the assertion of the lemma.
\end{proof}

\begin{lemma}
\label{firstupperbound}
There exists $\ee_0 >0$ such that, for any $\ee \in (0,\ee_0)$, $x\in \bb X$, $z\in \bb R$ and $n \in \bb N$, we have
\[
\bb E_x\left( z+M_n \,;\, \hat{T}_z>n \right) \leq \max(z,0) + c_{\ee} \left( n^{1/2-2\ee} + n^{2\ee} N(x) \right).
\]
\end{lemma}

\begin{proof}
Let $\ee$ be a real number in $(0,1/4)$. Denoting $z_+ := z + n^{1/2-2\ee}$ we have,
\begin{align}
	\bb E_x\left( z+M_n \,;\, \hat{T}_z>n \right) &= \underbrace{\bb E_x\left( z+M_n \,;\, T_{z_+} \leq n \,,\, \hat{T}_z>n \right)}_{=:J_1} \nonumber\\
	&\qquad + \underbrace{\bb E_x\left( z+M_n \,;\, T_{z_+} > n \,,\, \hat{T}_z>n \right)}_{=:J_2}.
	\label{decmartkillpremmaj}
\end{align}

\textit{Bound of $J_1$.} Let $y=z-r(x)$. Using the fact that $\bb P_x \left( \tau_y \leq k \,,\, \hat{T}_z > k \,,\, T_{z_+} =k \right) = 0$ and the Markov property, in the same way as in the proof of Lemma \ref{MarkovpropforTz},
\begin{align*}
	J_1 &= \sum_{k=1}^n \int_{\bb X \times \bb R} \bb E_{x'} \left( z'+ M_{n-k} \,;\, \hat{T}_{z'} > n-k \right) \\
	&\hspace{5cm} \times \bb P_x \left( X_k \in \dd x' \,,\, z+M_k \in \dd z' \,,\, \tau_y > k \,,\, T_{z_+} =k \right).
\end{align*}
Since $z+M_{T_{z_+}} < 0$, using the point \ref{majmart002} of Lemma \ref{majmart}, we have
\[
J_1 \leq c \bb E_x \left( \sqrt{n} +  N\left(X_{T_{z_+}}\right) \,;\, \tau_y > T_{z_+} \,,\, T_{z_+} \leq n \right).
\]
By the approximation \eqref{decMSX}, on the event $\{ \tau_y > T_{z_+} \},$ it holds
\[
r \left( X_{T_{z_+}} \right) = z + M_{T_{z_+}} - \left( y + S_{T_{z_+}} \right) < -n^{1/2-2\ee}.
\]
Therefore, by Lemma \ref{MTR},
\begin{align*}
	J_1 &\leq c n^{2\ee} \bb E_x \left( \abs{r \left( X_{T_{z_+}} \right)} + N\left(X_{T_{z_+}}\right) \,;\, \abs{r \left( X_{T_{z_+}} \right)} > n^{1/2-2\ee} \,,\, T_{z_+} \leq n \right)\\
	&\leq c n^{2\ee} + c n^{2\ee} \bb E_x \left( N\left(X_{T_{z_+}}\right) \,;\, T_{z_+} \leq n \right).
\end{align*}
Choosing $\ee$ small enough, by \eqref{NXT},
\begin{equation}
	\label{MajorJ1}
	J_1 \leq c n^{2\ee} + c_{\ee} n^{2\ee} \left( n^{1/2-4\ee} + N(x) \right) \leq c_{\ee} n^{1/2-2\ee} + c_{\ee} n^{2\ee} N(x).
\end{equation}

\textit{Bound of $J_2$.}
By Lemma \ref{firstupperboundforMnTz}, there exists $\ee_0>0$ such that, for any $\ee\in (0,\ee_0)$,
\[
J_2 \leq \bb E_x\left( z_+ + M_n \,;\, T_{z_+} > n \right) \leq \max(z,0) + c_{\ee} n^{1/2-2\ee} + c_{\ee} N(x).
\]

Inserting this bound and \eqref{MajorJ1} into \eqref{decmartkillpremmaj}, for any $\ee\in (0,\ee_0)$, we deduce the assertion of the lemma.
\end{proof}

Let $\nu_n$ be the first time when the martingale $z+M_n$ exceeds $n^{1/2-\ee}$: for any $n \geq 1$, $\ee \in (0,1/2)$ and $z\in \bb R$,
\[
\nu_n = \nu_{n,\ee,z} := \min \left\{ k \geq 1,\, z+M_k > n^{1/2-\ee} \right\}.
\]

The control on the joint law of $\nu_n$ and $\hat{T}_z$ is given by the following lemma.

\begin{lemma}
\label{concentnu}
There exists $\ee_0>0$ such that, for any $\ee \in (0,\ee_0)$, $\delta>0$, $x \in \bb X$, $z \in \bb R$ and $n \in \bb N$,
\[
\bb P_x \left( \nu_n > \delta n^{1-\ee} \,,\, \hat{T}_z > \delta n^{1-\ee} \right) \leq \e^{-c_{\ee,\delta} n^{\ee}} \left( 1 + N \left( x \right) \right).
\]
\end{lemma}

\begin{proof}
Let $\ee \in (0,1/4)$ and $K:= \pent{n^\ee/2}$. We split the interval $[1,\delta n^{1-\ee}]$ by subintervals of length $l:= \pent{\delta n^{1-2\ee}}$. Introduce the event $A_{k,z} := \{ \max_{1 \leq k' \leq k} \left( z+M_{k'l} \right) \leq n^{1/2-\ee}  \}$. Then 
\begin{equation}
	\bb P_x \left( \nu_n > \delta n^{1-\ee} \,,\, \hat{T}_z > \delta n^{1-\ee} \right) \leq \bb P_x \left( A_{2K,z} \,,\, \hat{T}_z > 2Kl \right).
	\label{Probnun}
\end{equation}
By the Markov property, as in the proof of Lemma \ref{MarkovpropforTz}, with $y=z-r(x)$, we have
\begin{align}
	\bb P_x &\left( A_{2K,z} \,,\, \hat{T}_z > 2Kl \right) \nonumber\\
	&= \int_{\bb X \times \bb R} \bb P_{x'} \left( A_{2,z'} \,,\, \hat{T}_{z'} > 2l  \right) \bb P_x \left( X_{2(K-1)l} \in \dd x' \,,\, z+M_{2(K-1)l} \in \dd z' \,,\, \right. \nonumber\\
	&\hspace{8cm} \left. A_{2(K-1),z} \,,\, \tau_y > 2(K-1)l \right) \nonumber\\
	&\quad + \int_{\bb X \times \bb R} \bb P_{x'} \left( A_{2,z'} \,,\, T_{z'} > 2l  \right) \bb P_x \left( X_{2(K-1)l} \in \dd x' \,,\, z+M_{2(K-1)l} \in \dd z'\,,\, \right. \nonumber\\
	&\qquad \hspace{4cm}\left.  A_{2(K-1),z} \,,\, \tau_y \leq 2(K-1)l \,,\, \hat{T}_z > 2(K-1)l \right).
	\label{PA00}
\end{align}

\textit{Bound of $\bb P_{x'} \left( A_{2,z'} \,,\, \hat{T}_{z'} > 2l  \right)$.} With $y'=z'-r(x')$, we write
\begin{align}
	\bb P_{x'} &\left( A_{2,z'} \,,\, \hat{T}_{z'} > 2l  \right) \nonumber\\
	&= \int_{\bb X \times \bb R} \bb P_{x''} \left( A_{1,z''} \,,\, \hat{T}_{z''} > l \right) \bb P_{x'} \left( X_l \in \dd x'' \,,\, z'+M_l \in \dd z'' \,,\, A_{1,z'} \,,\, \tau_{y'} > l \right) \nonumber\\
	\label{PA1}
	&\hspace{1cm} + \int_{\bb X \times \bb R} \bb P_{x''} \left( A_{1,z''} \,,\, T_{z''} > l \right) \\
	&\hspace{3cm} \times \bb P_{x'} \left( X_l \in \dd x'' \,,\, z'+M_l \in \dd z'' \,,\, A_{1,z'} \,,\, \tau_{y'} \leq l \,,\, \hat{T}_{z'} > l \right). \nonumber
\end{align}

\textit{Bound of $\bb P_{x''} \left( A_{1,z''} \,,\, \hat{T}_{z''} > l \right)$.} Note that on the event $\{ \tau_{y'} > l \}$ we have $y' + S_{l} > 0$. So it is enough to consider that $y''=z''-r(x'')>0$. Using \eqref{decMSX} we have,
\begin{align*}
	\bb P_{x''} \left( A_{1,z''} \,,\, \hat{T}_{z''} > l \right) &\leq \bb P_{x''} \left( y''+S_l \leq 2 n^{1/2-\ee} \,,\, \abs{r \left( X_l \right)} \leq n^{1/2-\ee} \right) \\
	&\hspace{6cm} + \bb P_{x''} \left( \abs{r \left( X_l \right)} > n^{1/2-\ee} \right).
\end{align*}
Therefore, there exists a constant $c_{\ee,\delta}$ such that
\[
\bb P_{x''} \left( A_{1,z''} \,,\, \hat{T}_{z''} > l \right) \leq \bb P_{x''} \left( \frac{S_l}{\sqrt{l}} \leq c_{\ee,\delta} \right) + \bb E_{x''} \left( \frac{\abs{r \left( X_l \right)}}{n^{1/2-\ee}} \right).
\]
Using Corollary \ref{BerEss} and Lemma \ref{MTR}, there exists $\ee_0 \in (0,1/4)$, such that, for any $\ee \in (0,\ee_0)$,
\[
\bb P_{x''} \left( A_{1,z''} \,,\, \hat{T}_{z''} > l \right) \leq \int_{-\infty}^{c_{\ee,\delta}} \e^{-\frac{u^2}{2\sigma^2}} \frac{\dd u}{\sqrt{2\pi} \sigma} + \frac{c_{\ee}}{l^\ee} \left(1+N(x'')\right) + \frac{c}{n^{1/2-\ee}} \bb E_{x''} \left( 1+N \left( X_l \right) \right).
\]
Using the point \ref{Momdec001} of Hypothesis \ref{Momdec} and the fact that $l^\ee \geq n^{\ee/2}/c_{\ee,\delta}$ for $\ee < 1/4$, we have,
\begin{equation}
	\bb P_{x''} \left( A_{1,z''} \,,\, \hat{T}_{z''} > l \right) \leq q_{\ee,\delta} + \frac{c_{\ee,\delta}}{n^{\ee/2}} \left(1+N(x'')\right),
	\label{majq001}
\end{equation}
with $q_{\ee,\delta} := \int_{-\infty}^{c_{\ee,\delta}} \e^{-\frac{u^2}{2\sigma^2}} \frac{\dd u}{\sqrt{2\pi} \sigma} < 1$.

\textit{Bound of $\bb P_{x''} \left( A_{1,z''} \,,\, T_{z''} > l \right)$.} 
On the event $\{ T_{z''} > l \}$ we have $z''+M_l > 0$. Using \eqref{decMSX} and Corollary \ref{BerEss}, in the same way as in the proof of the bound \eqref{majq001}, we obtain
\begin{align}
	\bb P_{x''} \left( A_{1,z''} \,,\, T_{z''} > l \right) &\leq \bb P_{x''} \left( 0 < z''+M_l \leq n^{1/2-\ee} \right) \nonumber \\
	&\leq \int_{\frac{-y''}{\sqrt{l}}-c_{\ee,\delta}}^{\frac{-y''}{\sqrt{l}}+c_{\ee,\delta}} \e^{-\frac{u^2}{2\sigma^2}} \frac{\dd u}{\sqrt{2\pi} \sigma} + \frac{c_{\ee,\delta}}{n^{\ee/2}} 
	\left( 1+N(x'') \right) \nonumber\\
        &\leq q_{\ee,\delta} + \frac{c_{\ee,\delta}}{n^{\ee/2}} \left(1+N(x'')\right).
	\label{majq002}
\end{align}

Inserting \eqref{majq001} and \eqref{majq002} into \eqref{PA1} and using \eqref{decexpN}, we have
\begin{align}
	\bb P_{x'} \left( A_{2,z'} \,,\, \hat{T}_{z'} > 2l  \right) &\leq q_{\ee,\delta} + \frac{c_{\ee,\delta}}{n^{\ee/2}} + \frac{c_{\ee,\delta}}{n^{\ee/2}} \bb E_{x'} \left( N \left( X_l \right) \right) \nonumber\\
	&\leq q_{\ee,\delta} + \frac{c_{\ee,\delta}}{n^{\ee/2}} + \e^{-c_{\ee,\delta} n^{1-2\ee}} N(x').
	\label{majq003}
\end{align}

\textit{Bound of $\bb P_{x'} \left( A_{2,z'} \,,\, T_{z'} > 2l  \right)$.} By the Markov property, 
\begin{align*}
	\bb P_{x'} \left( A_{2,z'} \,,\, T_{z'} > 2l  \right) &= \int_{\bb X \times \bb R} \bb P_{x''} \left( A_{1,z''} \,,\, T_{z''} > l \right) \\
	&\qquad \times \bb P_{x'} \left( X_l \in \dd x'' \,,\, z'+M_{l} \in \dd z'' \,,\, A_{1,z'} \,,\, T_{z'} > l \right).
\end{align*}
Using \eqref{majq002} to bound the probability inside the integral, we get
\begin{equation}
	\label{majq004}
	\bb P_{x'} \left( A_{2,z'} \,,\, T_{z'} > 2l  \right) \leq q_{\ee,\delta} + \frac{c_{\ee,\delta}}{n^{\ee/2}} + \e^{-c_{\ee,\delta} n^{1-2\ee}} N(x').
\end{equation}

Inserting the bounds \eqref{majq003} and \eqref{majq004} into \eqref{PA00}, we find that
\begin{align*}
	\bb P_x &\left( A_{2K,z} \,,\, \hat{T}_z > 2Kl \right) \\
	&\leq \left( q_{\ee,\delta} + \frac{c_{\ee,\delta}}{n^{\ee/2}} \right) \bb P_x \left( A_{2(K-1),z} \,,\, \hat{T}_z > 2(K-1)l \right) + \e^{-c_{\ee,\delta} n^{1-2\ee}} \left( 1+N(x) \right).
\end{align*}
Iterating this inequality, we get
\[
\bb P_x \left( A_{2K,z} \,,\, \hat{T}_z > 2Kl \right) \leq \left( q_{\ee,\delta} + \frac{c_{\ee,\delta}}{n^{\ee/2}} \right)^K + \e^{-c_{\ee,\delta} n^{1-2\ee}} \left( 1+N(x) \right) \sum_{k=0}^{K-1} \left( q_{\ee,\delta} + \frac{c_{\ee,\delta}}{n^{\ee/2}} \right)^k.
\]
As $K=\pent{n^\ee/2}$ and $q_{\ee,\delta} < 1$ it follows that, for $n$ large enough, $\left( q_{\ee,\delta} + \frac{c_{\ee,\delta}}{n^{\ee/2}} \right)^K \leq \e^{-c_{\ee,\delta} n^\ee}$,
which, in turn, implies
\[
\bb P_x \left( A_{2K,z} \,,\, \hat{T}_z > 2Kl \right) \leq \e^{-c_{\ee,\delta} n^{\ee}} \left( 1+N(x) \right).
\]
\end{proof}

\begin{lemma}
\label{intofthemartcond}
There exists $\ee_0 > 0$ such that, for any $\ee \in (0,\ee_0)$, $x\in \bb X$, $z \in \bb R$, $n\geq 2$ and any integer $n_f \in \{2, \dots, n \}$,
\[
\bb E_x \left( z+M_n \,;\, \hat{T}_z > n \right) \leq \left( 1 + \frac{c_{\ee}}{n_f^\ee} \right) \left( \max(z,0) + cN(x) \right) + c_{\ee} n_f^{1/2} + \e^{-c_{\ee} n_f^{\ee}} N(x).
\]
\end{lemma}

\begin{proof}
Let $\ee \in (0,1/2)$ and $y=z-r(x)$. Considering the stopping time 
\[
\nu_n^\ee:= \nu_n + \pent{n^\ee},
\]
we have,
\begin{align}
	\bb E_x \left( z+M_n \,;\, \hat{T}_z > n \right) \leq\;& \underbrace{\bb E_x \left( z+M_n \,;\, \hat{T}_z > n \,,\, \nu_n^\ee > \pent{n^{1-\ee}} \right)}_{=:J_1} \nonumber\\
	&+ \underbrace{\bb E_x \left( z+M_n \,;\, \hat{T}_z > n \,,\, \nu_n^\ee \leq \pent{n^{1-\ee}} \right)}_{=:J_2}.
	\label{UBJ1ET2}
\end{align}

\textit{Bound of $J_1$.} 
Set $m_{\ee} = \pent{n^{1-\ee}}-\pent{n^\ee}$. Using the fact that $\{ \nu_n^\ee > \pent{n^{1-\ee}} \} = \{ \nu_n > m_\ee \}$ and the Markov property as in the proof of Lemma \ref{MarkovpropforTz},
\begin{align*}
	J_1 =\; &\int_{\bb X \times \bb R} \bb E_{x'} \left( z'+M_{n-m_\ee} \,;\, \hat{T}_{z'} > n-m_\ee \right) \\
	&\qquad \times \bb P_x \left( X_{m_\ee} \in \dd x' \,,\, z+M_{m_\ee} \in \dd z' \,,\, \tau_y > m_\ee \,,\, \nu_n > m_\ee \right)\\
	&+\int_{\bb X \times \bb R} \bb E_{x'} \left( z'+M_{n-m_\ee} \,;\, T_{z'} > n-m_\ee \right) \\
	&\qquad \times \bb P_x \left( X_{m_\ee} \in \dd x' \,,\, z+M_{m_\ee} \in \dd z' \,,\, \tau_y \leq m_\ee \,,\, \hat{T}_z > m_\ee \,,\, \nu_n > m_\ee \right).
\end{align*}
On the event $\{ \nu_n > m_\ee \}$, we have $z'=z+M_{m_\ee} \leq n^{1/2-\ee}$. Moreover by the point \ref{majmart002} of Lemma \ref{majmart}, $\bb E_{x'} \left( \abs{M_{n-m_\ee}} \right) \leq c n^{1/2} + c N(x')$. Therefore,
\[
J_1 \leq c \bb E_x \left( n^{1/2} + N\left(X_{m_\ee}\right) \,;\, \hat{T}_z > m_\ee \,,\, \nu_n > m_\ee \right).
\]
Set $m_{\ee}'=m_{\ee}-\pent{n^{\ee}} = \pent{n^{1-\ee}} - 2\pent{n^{\ee}}$. Using the Markov property and \eqref{decexpN},
\begin{align*}
	J_1 &\leq c \int_{\bb X} n^{1/2} +\bb E_{x'} \left( N\left( X_{\pent{n^{\ee}}} \right) \right) \bb P_x\left( X_{m_{\ee}'} \in \dd x' \,,\, \hat{T}_z > m_\ee' \,,\, \nu_n > m_\ee' \right) \\
	&\leq c n^{1/2} \bb P_x\left( \hat{T}_z > m_\ee' \,,\, \nu_n > m_\ee' \right) + \e^{-c n^\ee} \bb E_x \left( N\left( X_{m_\ee'} \right) \right).
\end{align*}
By Lemma \ref{concentnu} and the point \ref{Momdec001} of Hypothesis \ref{Momdec},
\begin{equation}
	\label{MajJ1}
	J_1 \leq c n^{1/2} \e^{-c_{\ee} n^\ee} \left( 1 + N(x) \right) + \e^{-c n^\ee} \left( 1+N(x) \right) \leq \e^{-c_\ee n^\ee} \left( 1+N(x) \right).
\end{equation}

\textit{Bound of $J_2$.}  By the Markov property, as in the proof of Lemma \ref{MarkovpropforTz}, we have
\begin{align*}
	J_2 = \sum_{k=1}^{\pent{n^{1-\ee}}} &\int_{\bb X \times \bb R} \bb E_{x'} \left( z'+M_{n-k} \,;\, \hat{T}_{z'} > n-k \right) \\
	&\qquad \times \bb P_x \left( X_k \in \dd x' \,,\, z+M_k \in \dd z' \,,\, \tau_y > k \,,\, \nu_n^\ee = k \right) \\
	&+ \int_{\bb X \times \bb R} \bb E_{x'} \left( z'+M_{n-k} \,;\, T_{z'} > n-k \right) \\
	&\qquad \times \bb P_x \left( X_k \in \dd x' \,,\, z+M_k \in \dd z' \,,\, \tau_y \leq k \,,\, \hat{T}_{z} > k \,,\, \nu_n^\ee = k \right).
\end{align*}
By Lemmas \ref{firstupperbound} and \ref{firstupperboundforMnTz},
\begin{align}
	J_2 &\leq \underbrace{c_\ee \bb E_x \left( n^{1/2-2\ee} + n^{2\ee} N\left( X_{\nu_n^\ee} \right) \,;\, \hat{T}_{z} > \nu_n^\ee \,,\, \nu_n^\ee \leq \pent{n^{1-\ee}} \right)}_{=:J_{21}} \nonumber\\
	&\qquad+ \underbrace{\bb E_x \left( \max \left( z+M_{\nu_n^\ee} , 0 \right) \,;\, \hat{T}_{z} > \nu_n^\ee \,,\, \nu_n^\ee \leq \pent{n^{1-\ee}} \right)}_{=:J_{22}}.
	\label{decJ2}
\end{align}

\textit{Bound of $J_{21}$.} Using the Markov property and \eqref{decexpN},
\begin{align}
	J_{21} &\leq c_\ee \int_{\bb X} \bb E_{x'} \left( n^{1/2-2\ee} + n^{2\ee} N\left( X_{\pent{n^\ee}} \right) \right) \bb P_x \left( X_{\nu_n} \in \dd x' \,,\, \hat{T}_{z} > \nu_n \,,\, \nu_n \leq \pent{n^{1-\ee}} \right) \nonumber\\
	&\leq c_{\ee} \bb E_x \left( n^{1/2-2\ee} + \e^{-c_{\ee} n^{\ee}} N\left( X_{\nu_n} \right) \,;\, \hat{T}_{z} > \nu_n \,,\, \nu_n \leq \pent{n^{1-\ee}} \right) \nonumber\\
	&\leq \underbrace{c_{\ee} \bb E_x \left( n^{1/2-2\ee} \,;\, \hat{T}_{z} > \nu_n \,,\, \nu_n \leq \pent{n^{1-\ee}} \right)}_{=:J_{21}'} + \e^{-c_{\ee} n^{\ee}} n^{1-\ee} \left( 1 + N(x) \right).
	\label{decJ21}
\end{align}
By the definition of $\nu_n$, we have $n^{1/2-2\ee} < \frac{z+M_{\nu_n}}{n^\ee}$. 
So
\[
J_{21}' \leq \frac{c_{\ee}}{n^\ee} \bb E_x \left( z+M_{\nu_n} \,;\, \hat{T}_{z} > \nu_n \,,\, \nu_n \leq \pent{n^{1-\ee}} \right).
\]
Using Lemma \ref{Mnsubmartingale},
\begin{align}
	J_{21}' \leq\;& \frac{c_{\ee}}{n^\ee} \bb E_x \left( z+M_{\pent{n^{1-\ee}}} \,;\, \hat{T}_{z} > \pent{n^{1-\ee}} \right) \nonumber\\
	&- \frac{c_{\ee}}{n^\ee}\underbrace{ \bb E_x \left( z+M_{\pent{n^{1-\ee}}} \,;\, \hat{T}_{z} > \pent{n^{1-\ee}} \,,\, \nu_n > \pent{n^{1-\ee}} \right)}_{=:J_{21}''}.
	\label{decJ21'}
\end{align}
On the event $\{ \tau_y > \pent{n^{1-\ee}} \}$, by \eqref{decMSX}, it holds $z+M_{\pent{n^{1-\ee}}} > r \left( X_{\pent{n^{1-\ee}}} \right)$ 
and on the event $\{\tau_{y} \leq \pent{n^{1-\ee}} \,,\, \hat{T}_{z} > \pent{n^{1-\ee}} \}$ 
it holds $z+M_{\pent{n^{1-\ee}}} >0$. 
So, by the definition of $\hat{T}_z$,
\begin{align*}
	-J_{21}''  \leq\;& -\bb E_x \left( r \left( X_{\pent{n^{1-\ee}}} \right)  \,;\, \tau_{y} > \pent{n^{1-\ee}} \,,\, \nu_n > \pent{n^{1-\ee}} \right)\\
	\leq\;& c\bb E_x \left( 1+N \left( X_{\pent{n^{1-\ee}}} \right)  \,;\, \hat{T}_{z} > \pent{n^{1-\ee}} \,,\, \nu_n > \pent{n^{1-\ee}} \right).
\end{align*}
Denoting $m_{\ee} = \pent{n^{1-\ee}} - \pent{n^\ee}$ and using the Markov property and \eqref{decexpN},
\begin{align*}
	-J_{21}'' &\leq c \bb E_x \left( 1+e^{-c_{\ee} n^{\ee}} N\left( X_{m_{\ee}} \right) \,;\, \hat{T}_{z} > m_{\ee} \,,\, \nu_n > m_{\ee} \right) \\
	&\leq c \bb P_x \left( \nu_n > m_{\ee} \,,\, \hat{T}_{z} > m_{\ee} \right) + e^{-c_{\ee} n^{\ee}} \left( 1+N(x) \right).
\end{align*}
By Lemma \ref{concentnu},
\begin{equation}
	\label{MajJ21''}
	-J_{21}'' \leq e^{-c_{\ee} n^{\ee}} \left( 1+N(x) \right).
\end{equation}
Putting together \eqref{MajJ21''} and \eqref{decJ21'},
\begin{equation}
	\label{MajJ21'}
	J_{21}' \leq \frac{c_{\ee}}{n^\ee} \bb E_x \left( z+M_{\pent{n^{1-\ee}}} \,;\, \hat{T}_{z} > \pent{n^{1-\ee}} \right) + e^{-c_{\ee} n^{\ee}} \left( 1+N(x) \right).
\end{equation}
So, using \eqref{decJ21},
\begin{equation}
	\label{MajJ21}
	J_{21} \leq \frac{c_{\ee}}{n^\ee} \bb E_x \left( z+M_{\pent{n^{1-\ee}}} \,;\, \hat{T}_{z} > \pent{n^{1-\ee}} \right) + \e^{-c_{\ee} n^{\ee}} \left( 1+N(x) \right).
\end{equation}

\textit{Bound of $J_{22}$.} On the event $\{ \hat{T}_{z} > \nu_n^\ee \,,\, \tau_y \leq \nu_n^\ee \}$ we have $z+M_{\nu_n^\ee} > 0$. Consequently
\begin{align*}
	J_{22} =\;& \bb E_x \left( z+M_{\nu_n^\ee} \,;\, \hat{T}_{z} > \nu_n^\ee \,,\, \nu_n^\ee \leq \pent{n^{1-\ee}} \right) \\
	&+ \bb E_x \left( \max \left( z+M_{\nu_n^\ee} , 0 \right) - \left(z+M_{\nu_n^\ee}\right) \,;\, \tau_y > \nu_n^\ee \,,\, \nu_n^\ee \leq \pent{n^{1-\ee}} \right).
\end{align*}
By Lemma \ref{Mnsubmartingale},
\begin{align}
	J_{22} \leq\;& \bb E_x \left( z+M_{\pent{n^{1-\ee}}} \,;\, \hat{T}_{z} > \pent{n^{1-\ee}} \right) \nonumber\\
	\label{decJ22}
	&-\underbrace{\bb E_x \left( z+M_{\pent{n^{1-\ee}}} \,;\, \hat{T}_{z} > \pent{n^{1-\ee}} \,,\, \nu_n^\ee > \pent{n^{1-\ee}} \right)}_{=:J_{22}''}\\
	&\underbrace{- \bb E_x \left( z+M_{\nu_n^\ee} \,;\, z+M_{\nu_n^\ee} < 0 \,,\, \tau_y > \nu_n^\ee \,,\, \nu_n^\ee \leq \pent{n^{1-\ee}} \right)}_{=:J_{22}'}. \nonumber	
\end{align}
In the same way as in the proof of the bound of $J_{21}''$, replacing $\nu_n$ by $\nu_n^{\ee}$, one can prove that
\begin{equation}
	\label{MajJ22''}
	-J_{22}'' \leq e^{-c_{\ee} n^{\ee}} \left( 1+N(x) \right).
\end{equation}
Moreover, using \eqref{decMSX}, on the event $\{ \tau_y > \nu_n^\ee \}$, we have $-(z+M_{\nu_n^\ee}) < -r \left( X_{\nu_n^\ee} \right)$. So, by Lemma \ref{MTR}, the Markov property and \eqref{decexpN},
\begin{align}
	J_{22}' &\leq \bb E_x \left( \abs{ r \left( X_{\nu_n^\ee} \right) } \,;\, \hat{T}_{z} > \nu_n^\ee \,,\, \nu_n^\ee \leq \pent{n^{1-\ee}} \right) \nonumber\\
	\label{MajJ22'}
	&\leq c \bb E_x \left( 1+\e^{-c n^\ee} N \left( X_{\nu_n} \right) \,;\, \hat{T}_{z} > \nu_n \,,\, \nu_n \leq \pent{n^{1-\ee}} \right) \\
	&\leq J_{21}' + \e^{-c_{\ee} n^\ee} \left( 1+N(x) \right). \nonumber
\end{align}
With \eqref{MajJ21'}, \eqref{decJ22} and \eqref{MajJ22''} we obtain,
\begin{align}
	\label{MajJ22}
	J_{22} &\leq \left( 1 + \frac{c_ {\ee}}{n^\ee} \right) \bb E_x \left( z+M_{\pent{n^{1-\ee}}} \,;\, \hat{T}_{z} > \pent{n^{1-\ee}} \right) + \e^{-c_{\ee} n^\ee} \left( 1+N(x) \right).
\end{align}

Inserting \eqref{MajJ22} and \eqref{MajJ21} into \eqref{decJ2},
\begin{equation}
\label{MajJ2}
J_2 \leq \left( 1 + \frac{c_{\ee}}{n^\ee} \right) \bb E_x \left( z+M_{\pent{n^{1-\ee}}} \,;\, \hat{T}_{z} > \pent{n^{1-\ee}} \right) + \e^{-c_{\ee} n^{\ee}} \left( 1 + N(x) \right).
\end{equation}

Now, inserting \eqref{MajJ1} and \eqref{MajJ2} into \eqref{UBJ1ET2}, we find that
\begin{align}
	\bb E_x \left( z+M_n \,;\, \hat{T}_z > n \right) &\leq \left( 1 + \frac{c_{\ee}}{n^\ee} \right) \bb E_x \left( z+M_{\pent{n^{1-\ee}}} \,;\, \hat{T}_{z} > \pent{n^{1-\ee}} \right) \nonumber\\
	&\qquad+ \e^{-c_{\ee} n^{\ee}} \left( 1 + N(x) \right).
	\label{REC001}
\end{align}
By Lemma \ref{Mnsubmartingale}, the sequence $( \bb E_x ( z+M_n \,;\, \hat{T}_z > n ) )_{n\geq 1}$ is non-decreasing. Therefore, using Lemma \ref{lemanalyse}, we obtain that for any $n\geq 2$ and $n_f \in \{ 2, \dots, n \}$,
\[
\bb E_x \left( z+M_n \,;\, \hat{T}_z > n \right) \leq \left( 1 + \frac{c_{\ee}}{n_f^\ee} \right) \bb E_x \left( z+M_{n_f} \,;\, \hat{T}_z > n_f \right) + \e^{-c_{\ee} n_f^{\ee}} \left( 1 + N(x) \right).
\]
Finally, by the point \ref{majmart002} of Lemma \ref{majmart},
\[
\bb E_x \left( z+M_n \,;\, \hat{T}_z > n \right) \leq \left( 1 + \frac{c_{\ee}}{n_f^\ee} \right) \left( \max(z,0) + cN(x) \right) + c_{\ee} n_f^{1/2} + \e^{-c_{\ee} n_f^{\ee}} N(x).
\]
\end{proof}

\begin{corollary}
\label{unifmajdelafoncinv}
There exists $\ee_0 > 0$ such that, for any $\ee \in (0,\ee_0)$, $x\in \bb X$, $y \in \bb R$, $n\geq 0$ and any integer $n_f \in \{2, \dots, n \}$,
	\[
	\bb E_x \left( y+S_n \,;\, \tau_y > n \right) \leq \left( 1 + \frac{c_{\ee}}{n_f^\ee} \right) \left( \max(y,0) + c N(x) \right) + c_{\ee} n_f^{1/2} + \e^{-c_{\ee} n_f^{\ee}} N(x).
\]
\end{corollary}

\begin{proof} First, using the definition of $\hat{T}_z$ and Lemma \ref{intofthemartcond}, with $z=y+r(x)$,
\begin{align}
	\bb E_x \left( z+M_n \,;\, \tau_y > n \right) &= \bb E_x \left( z+M_n \,;\, \hat{T}_z>n \right) - \bb E_x \left( z+M_n \,;\, \tau_y \leq n \,,\, \hat{T}_z>n \right) \nonumber\\
	\label{VnPPWn}
	&\leq \bb E_x \left( z+M_n \,;\, \hat{T}_z>n \right) \\
	\label{VnPPWn2}
	&\leq \left( 1 + \frac{c_{\ee}}{n_f^\ee} \right) \left( \max(z,0) + cN(x) \right) + c_{\ee} n_f^{1/2} + \e^{-c_{\ee} n_f^{\ee}} N(x).
\end{align}
Now, using \eqref{decMSX}, Lemma \ref{MTR} and \eqref{decexpN},
\begin{align*}
	\bb E_x \left( y+S_n \,;\, \tau_y > n \right) &= \bb E_x \left( z+M_n \,;\, \tau_y > n \right) - \bb E_x \left( r \left(X_n\right) \,;\, \tau_y > n \right) \\
	&\leq \bb E_x \left( z+M_n \,;\, \tau_y > n \right) + c \left( 1+\e^{-cn} N(x) \right)\\
	&\leq \left( 1 + \frac{c_{\ee}}{n_f^\ee} \right) \left( \max(z,0) + cN(x) \right) + c_{\ee} n_f^{1/2} + \e^{-c_{\ee} n_f^{\ee}} N(x).
\end{align*}
Using the definition of $z$ concludes the proof.
\end{proof}

\section{Existence of the harmonic function}
\label{Sec Harm Func}

In this section we rely upon the results of the previous Sections \ref{Mart Approx} and \ref{CMWI} to construct a non-trivial harmonic function $V$ (Corollary \ref{ExofVW}) and state some of its properties (Proposition \ref{IFP}). The idea consists in establishing the existence of the limit as $n\to +\infty$ of the expectation $-\bb E_x \left( M_{\tau_y} \,;\, \tau_y \leq n \right)$ using the Lebesgue dominated convergence theorem. 
To this end, we state the following Lemma.

\begin{lemma}
\label{ExitfinitTmanuz}
For any $x \in \bb X$ and $z \in \bb R$,
\[
\hat{T}_z < + \infty \quad \text{$\bb P_x$-a.s.}
\]
\end{lemma}

\begin{proof}
In order to apply Lemmas \ref{Exitfinit} and \ref{ExitfinitTz}, we write, with $y=z-r(x)$,
\begin{align*}
	\bb P_x \left( \hat{T}_z > n \right) &\leq \bb P_x \left( \tau_y > \pent{n/2} \right) + \int_{\bb X \times \bb R} \bb P_{x'} \left( T_{z'} > n - \pent{n/2} \right)	\bb P_x \left( X_{\pent{n/2}} \in \dd x' \,,\, \right. \\
	&\hspace{0.5cm} \left. z+M_{\pent{n/2}} \in \dd z' \,,\, \tau_y \leq \pent{n/2} \,,\, \hat{T}_{z} > \pent{n/2} \right).
\end{align*}
Using \eqref{tauyto0}, \eqref{Tzto0} and the definition of $y$, we have
\begin{align*}
	\bb P_x \left( \hat{T}_z > n \right) \leq\;& \frac{c_{\ee}}{n^{\ee}} \left( 1+\max(y,0)+N(x) \right) \\
	&+ \frac{c_{\ee}}{n^{\ee}} \bb E_x \left( 1+z+M_{\pent{n/2}} + N \left( X_{\pent{n/2}} \right) \,;\, \tau_y \leq \pent{n/2} \,,\, \hat{T}_z > \pent{n/2} \right).
\end{align*}
By the point \ref{Momdec001} of Hypothesis \ref{Momdec},
\begin{align*}
	\bb P_x \left( \hat{T}_z > n \right) \leq\;& \frac{c_{\ee}}{n^{\ee}} \left( 1+\max(y,0)+N(x) \right) + \frac{c_{\ee}}{n^{\ee}} \bb E_x \left( z+M_{\pent{n/2}} \,;\,  \hat{T}_z > \pent{n/2} \right) \\
	&- \frac{c_{\ee}}{n^{\ee}} \bb E_x \left( z+M_{\pent{n/2}} \,;\,  \tau_y > \pent{n/2} \right).
\end{align*}
Using \eqref{decMSX}, we see that on the event $\{ \tau_y > \pent{n/2} \}$ we have $z+M_{\pent{n/2}} > r \left( X_{\pent{n/2}} \right)$. 
Then, by Lemma \ref{MTR} and the point \ref{Momdec001} of Hypothesis \ref{Momdec},
\[
\bb P_x \left( \hat{T}_z > n \right) \leq \frac{c_{\ee}}{n^{\ee}} \left( 1+\max(y,0)+N(x) \right) + \frac{c_{\ee}}{n^{\ee}} \bb E_x \left( z+M_{\pent{n/2}} \,;\, \hat{T}_z > \pent{n/2} \right).
\]
Using Lemma \ref{intofthemartcond}, we have
\[
\bb P_x \left( \hat{T}_z > n \right) \leq \frac{c_{\ee}}{n^{\ee}} \left( 1+\max(y,0)+N(x) \right).
\]
Finally, we conclude that
\[
\bb P_x \left( \hat{T}_z = + \infty \right) = \underset{n\to +\infty}{\lim} \bb P_x \left( \hat{T}_z > n \right) = 0.
\]
\end{proof}

\begin{lemma}
\label{intdeMtau}
Let $x \in \bb X,$ $y \in \bb R$ and $z=y+r(x)$. The random variables $M_{\hat{T}_z}$, $M_{T_z}$ and $M_{\tau_y}$ are integrable and
\[
\max \left\{ \bb E_x \left( \abs{M_{\hat{T}_z}} \right), \bb E_x \left( \abs{M_{T_z}} \right), \bb E_x \left( \abs{M_{\tau_y}} \right) \right\} \leq c \left( 1+ \abs{z} + N(x) \right) < + \infty.
\]
\end{lemma}

\begin{proof}
The stopping times $\tau_y \wedge n$, $T_z \wedge n$ and $\hat{T}_z \wedge n$ are bounded and
 satisfy $\tau_y \wedge n \leq \hat{T}_z \wedge n$ and $T_z \wedge n \leq \hat{T}_z \wedge n.$ 
Since $(\abs{M_n})_{n\geq 0}$ is a submartingale, we have
\begin{equation}
	\label{MtauinfnleqMTinfn}
	\max \left\{ \bb E_x \left( \abs{M_{\tau_y \wedge n}} \right), \bb E_x \left( \abs{M_{T_z \wedge n}} \right) \right\}  \leq \bb E_x \left( \abs{M_{\hat{T}_z \wedge n}} \right).
\end{equation}
Using the optional stopping theorem,
\begin{align*}
	\bb E_x \left( \abs{M_{\hat{T}_z \wedge n}} \right) \leq\;& -\bb E_x \left( z+M_{\hat{T}_z} \,;\, \hat{T}_z \leq n \right) + \bb E_x \left( \abs{z+M_{n}} \,;\, \tau_y > n \right) \\
	& + \bb E_x \left( z+M_{n} \,;\, \tau_y \leq n \,,\, \hat{T}_z > n \right) + \abs{z} \\
	=\;& -\bb E_x \left( z+M_{n} \,;\, \hat{T}_z \leq n \right) - 2\bb E_x \left( z+M_{n} \,;\, z+M_{n} \leq 0 \,,\, \tau_y > n \right) \\
	& + \bb E_x \left( z+M_{n} \,;\, \tau_y > n \right) + \bb E_x \left( z+M_{n} \,;\, \tau_y \leq n \,,\, \hat{T}_z > n \right) + \abs{z}\\
	=\;& -z + 2\bb E_x \left( z+M_{n} \,;\, \hat{T}_z > n \right) \\
	&- 2\bb E_x \left( z+M_{n} \,;\, z+M_{n} \leq 0 \,,\, \tau_y > n \right) + \abs{z}.
\end{align*}
On the event $\{ z+M_{n} \leq 0 \,,\, \tau_y > n \}$, by \eqref{decMSX}, it holds $\abs{z+M_n} \leq \abs{r \left( X_n \right)}$.
Therefore, by Lemma \ref{MTR} and the point \ref{Momdec001} of Hypothesis \ref{Momdec}, we have
\[
- 2\bb E_x \left( z+M_{n} \,;\, z+M_{n} \leq 0 \,,\, \tau_y > n \right) \leq c \left( 1+N(x) \right),
\]
Using Lemma \ref{intofthemartcond},
\begin{equation}
	\label{majMTzinfn}
	\bb E_x \left( \abs{M_{\hat{T}_z}} \,;\, \hat{T}_z \leq n \right) \leq \bb E_x \left( \abs{M_{\hat{T}_z \wedge n}} \right) \leq c \left( 1+ \abs{z} + N(x) \right).
\end{equation}
By the Lebesgue monotone convergence theorem and the fact that $\hat{T}_z < + \infty,$ 
we deduce that $M_{\hat{T}_z}$ is $\bb P_x$-integrable and
\[
\bb E_x \left( \abs{M_{\hat{T}_z}} \right) \leq c \left( 1+ \abs{z} + N(x) \right).
\]
In the same manner, using (\ref{MtauinfnleqMTinfn}), (\ref{majMTzinfn}) and Lemmas \ref{Exitfinit} and \ref{ExitfinitTz}, we conclude that $M_{\tau_y}$ and $M_{T_z}$ are $\bb P_x$-integrable and
\[
\max \left\{ \bb E_x \left( \abs{M_{\tau_y}} \right), \bb E_x \left( \abs{M_{T_z}} \right) \right\} \leq c \left( 1+ \abs{z} + N(x) \right).
\]
The assertion of the lemma follows obviously from the last two inequalities.
\end{proof}

\begin{corollary}
\label{ExofVW}
For any $x \in \bb X$, $y \in \bb R$ and $z \in \bb R$, the following functions are well defined:
\[
W(x,z) :=  -\bb E_x \left( M_{T_z} \right), \quad \hat{W}(x,z) :=  -\bb E_x \left( M_{\hat{T}_z} \right) \quad \text{and} \quad V(x,y) :=  -\bb E_x \left( M_{\tau_y} \right).
\]
\end{corollary}

\begin{proposition}\ 
\label{IFP}
\begin{enumerate}[ref=\arabic*, leftmargin=*, label=\arabic*.]
	\item \label{IFP001} Let $x\in \bb X,$ $y \in \bb R$ and $z=y+r(x)$. Then
	\[
	V(x,y) = \underset{n\to +\infty}{\lim} \bb E_x \left( z+M_n \,;\, \tau_y > n \right) = \underset{n\to +\infty}{\lim} \bb E_x \left( y+S_n \,;\, \tau_y > n \right)
	\]
	and
	\begin{align*}
		W(x,z) &= \underset{n\to +\infty}{\lim} \bb E_x \left( z+M_n \,;\, T_z > n \right),\\
		\hat{W}(x,z) &= \underset{n\to +\infty}{\lim} \bb E_x \left( z+M_n \,;\, \hat{T}_z > n \right).
	\end{align*}
	\item \label{IFP002} For any $x \in \bb X$, the functions $y\mapsto V(x,y)$, $z\mapsto W(x,z)$ and $z\mapsto \hat{W}(x,z)$ are non-decreasing on $\bb R$.
	\item \label{IFP003} There exists $\ee_0>0$ such that, for any $\ee \in (0,\ee_0)$, $x\in \bb X$, $z\in \bb R$ and any integer $n_f \geq 2$,
	\begin{equation}
		\label{WLZX}
		\hat{W}(x,z) \leq \left( 1 + \frac{c_{\ee}}{n_f^\ee} \right) \left( \max(z,0) + cN(x) \right) + c_{\ee} n_f^{1/2} + \e^{-c_{\ee} n_f^{\ee}} N(x)
	\end{equation}
	and, for any $x\in \bb X$, $y\in \bb R$ and $z=y+r(x)$,
	\begin{equation}
		\label{VWppthW}
		0 \leq \min \left\{ V(x,y), W(x,z) \right\} \leq \max \left\{ V(x,y), W(x,z) \right\} \leq \hat{W}(x,y).
	\end{equation}
	In particular, for any $x\in \bb X$ and $y\in \bb R$,
	\begin{equation}
		\label{encadrV}
		0 \leq V(x,y) \leq c \left( 1 + \max(y,0) + N(x) \right).
	\end{equation}
	\item \label{IFP004} For any $x\in \bb X$ and $y\in \bb R$,
	\[
	V(x,y) = \bb E_x \left( V(X_1,y+S_1) \,;\, \tau_y > n \right)
	\]
	and $\left( V\left( X_n, y+S_n \right) \mathbbm 1_{\left\{ \tau_y > n \right\}} \right)_{n\geq 0}$ is a $\bb P_x$-martingale.
\end{enumerate}
\end{proposition}

\begin{proof} \textit{Claim \ref{IFP001}.} 
Let $\mathfrak{T}$ be any of the stopping times $\tau_y, T_z,$ or $\hat{T}_z$.
By the martingale property, 
\[
\bb E_x \left( z+M_n \,;\, \mathfrak{T} > n \right) = z\bb P_x \left( \mathfrak{T} > n \right) - \bb E_x \left( M_{\mathfrak{T}} \,;\, \mathfrak{T} \leq n \right).
\]
Using Lemmas \ref{Exitfinit}, \ref{ExitfinitTz}, \ref{ExitfinitTmanuz}, \ref{intdeMtau} and the Lebesgue dominated convergence theorem,
\[
\bb E_x \left( z+M_n \,;\, \mathfrak{T} > n \right) = -\bb E_x \left( M_{\mathfrak{T}} \right).
\]
Moreover, by \eqref{decMSX},
\[
\bb E_x \left( y+S_n \,;\, \tau_y > n \right) = \bb E_x \left( z+M_n \,;\, \tau_y > n \right) - \bb E_x \left( r\left(X_n\right) \,;\, \tau_y > n \right).
\]
Since, by Lemma \ref{MTR}, the point \ref{Momdec001} of Hypothesis \ref{Momdec} and Lemma \ref{Exitfinit}, we have
\begin{align}
	\abs{\bb E_x \left( r\left(X_n\right) \,;\, \tau_y > n \right)} &\leq c\bb E_x^{1/2} \left( \left( 1+N \left( X_n \right) \right)^2 \right) \bb P_x^{1/2} \left( \tau_y > n \right) \nonumber\\
	&\leq c \left( 1+N(x) \right) \bb P_x^{1/2} \left( \tau_y > n \right) \underset{n\to +\infty}{\longrightarrow} 0,
	\label{intrXn}
\end{align}
the claim \ref{IFP001} follows.

\textit{Proof of the claim \ref{IFP002}.} Let $x\in \bb X$. For any $y' \leq y$, we obviously have $\tau_{y'} \leq \tau_y$. Therefore,
\[
\bb E_x \left( y'+S_n \,;\, \tau_{y'} > n \right) \leq \bb E_x \left( y+S_n \,;\, \tau_{y'} > n \right) \leq \bb E_x \left( y+S_n \,;\, \tau_{y} > n \right).
\]
Taking the limit as $n\to +\infty$ and using the claim \ref{IFP001}, it follows that $V(x,y')\leq V(x,y)$. 
In the same way $W(x,z') \leq W(x,z)$ for $z'\leq z$.
To prove the monotonicity of
$\hat{W}$, we note that, for any $z' \leq z,$ $y'=z'-r(x)$ and $y=z-r(x)$, we have $\hat{T}_{z'} = \min \{ k \geq \tau_{y'},\, z'+M_k \leq 0 \} \leq \min \{ k \geq \tau_y,\, z'+M_k \leq 0 \} \leq \hat{T}_z$. So
\begin{align*}
	\bb E_x \left( z'+M_n \,;\, \hat{T}_{z'} > n \right) \leq\;& \bb E_x \left( z+M_n \,;\, \hat{T}_{z'} > n \,,\, \hat{T}_z > n \right)\\
	\leq\;& \bb E_x \left( y+S_n \,;\, \tau_y > n \right) + \bb E_x \left( \abs{r\left( X_n \right)} \,;\, \tau_y > n \right) \\
	&+ \bb E_x \left( z+M_n \,;\, \tau_y \leq n \,,\, \hat{T}_z > n \right) \\
	\leq\;& \bb E_x \left( z+M_n \,;\, \hat{T}_z>n \right) + 2\bb E_x \left( \abs{r\left( X_n \right)} \,;\, \tau_y > n \right).
\end{align*}
As in \eqref{intrXn}, taking the limit as $n\to +\infty$, by the claim \ref{IFP001}, we have $\hat{W}(x,z')\leq \hat{W}(x,z)$.

\textit{Proof of the claim \ref{IFP003}.} The inequality \eqref{WLZX} is a direct consequence of the claim \ref{IFP001} and Lemma \ref{intofthemartcond}. Moreover, taking the limit as $n\to \infty$ in \eqref{VnPPWn}, we get $V(x,y) \leq \hat{W}(x,z)$.

To bound $W$, we write
\begin{align*}
	\bb E_x \left( z+M_n \,;\, T_z > n \right) &\leq \bb E_x \left( z+M_n \,;\, \tau_y \leq n \,,\, \hat{T}_z > n \,,\, T_z > n \right) \\
	&\qquad \qquad + \bb E_x \left( z+M_n \,;\, z+M_n > 0 \,,\, \tau_y > n \,,\, T_z > n \right).
\end{align*}
Since $z+M_n >0$ on the event $\{ \tau_y \leq n \,,\, \hat{T}_z > n \}$,
\begin{align*}
	\bb E_x \left( z+M_n \,;\, T_z > n \right) &\leq \bb E_x \left( z+M_n \,;\, \tau_y \leq n \,,\, \hat{T}_z > n \right) \\
	&\qquad \qquad+ \bb E_x \left( z+M_n \,;\, z+M_n > 0 \,,\, \tau_y > n \right) \\
	&= \bb E_x \left( z+M_n \,;\, \hat{T}_z > n \right) \\
	&\qquad \qquad- \bb E_x \left( z+M_n \,;\, z+M_n \leq 0 \,,\, \tau_y > n \right).
\end{align*}
Using the approximation \eqref{decMSX},
\begin{equation}
	\label{MetMchap}
	\bb E_x \left( z+M_n \,;\, T_z > n \right) \leq \bb E_x \left( z+M_n \,;\, \hat{T}_z > n \right) + \bb E_x \left( \abs{r\left(X_n\right)} \,;\, \tau_y > n \right).
\end{equation}
As in \eqref{intrXn}, using the claim \ref{IFP001},
\[
W(x,z) \leq \hat{W}(x,z).
\]
Now, since $y+S_n$ is positive on the event $\{ \tau_y > n \}$, by the claim \ref{IFP001}, we see that $V(x,y) \geq 0$ and in the same way $W(x,z) \geq 0$. This proves \eqref{VWppthW}.

Inequality \eqref{encadrV} follows from \eqref{WLZX} and \eqref{VWppthW}.

\textit{Proof of the claim \ref{IFP004}.} By the Markov property,
\begin{align*}
	\bb E_x \left( y+S_{n+1} \,;\, \tau_y > n+1 \right) &= \int_{\bb X \times \bb R} \bb E_{x'} \left( y'+S_n \,;\, \tau_{y'} > n \right) \\
	&\qquad \qquad \times \bb P_x \left( X_1 \in \dd x' \,,\, y+S_1 \in \dd y' \,,\, \tau_y > 1 \right),
\end{align*}
where, by Corollary \ref{unifmajdelafoncinv}, $\bb E_{x'} \left( y'+S_n \,;\, \tau_{y'} > n \right) \leq c \left( 1+\abs{y'} + N\left(x'\right) \right)$ and by the point \ref{Momdec001} of Hypothesis \ref{Momdec},
\[
c\bb E_x \left( 1+\abs{y+S_1} + N\left(X_1\right) \right) \leq c \left( 1+\abs{y}+N(x) \right) < +\infty.
\]
The Lebesgue dominated convergence theorem implies that
\[
V(x,y) = \bb E_x \left( V \left( X_1, y+S_1 \right) \,;\, \tau_y > 1 \right).
\]
\end{proof}

\section{Positivity of the harmonic function}
\label{PosHaFun}

The aim of this section is to prove that the harmonic function $V$ is non-identically zero and to precise its domain of positivity $\mathscr{D}_{+}(V)$. 
For any $x\in \bb X$, $z\in \bb R$ and $n\in \bb N$, denote for brevity,
\[
\hat{W}_n(x,z) = \hat{W} \left( X_n, z+M_n \right) \mathbbm 1_{\left\{\hat{T}_z > n \right\}}.
\]
Although it is easy to verify that $\hat{W}(x,z) \geq z$ (see Lemma \ref{prophW}) which, in turn, ensures that $\hat{W}(x,z)>0$ for any $z>0$, 
there is no straightforward way to bound from below the function $V$. We take advantage of the lower bound $V(x,y) \geq \lim_{n} \bb E_x ( \hat{W}_n(x,z) \,;\, \tau_y > n )$ (Lemma \ref{minpos001}) and of the fact that $( \hat{W}_n(x,z) \mathbbm 1_{\left\{ \tau_y > n \right\}} )_{n\geq 0}$ is a $\bb P_x$-supermartingale (Lemma \ref{prophW}). By a recurrent procedure we obtain a lower bound for $V$ (Lemma \ref{PosdeV}) which subsequently is used to prove the positivity of $V$ (Lemma \ref{posdeVsurDgamma}).

\begin{lemma}\ 
\label{prophW}
\begin{enumerate}[ref=\arabic*, leftmargin=*, label=\arabic*.]
	\item \label{prophW001} For any $x\in \bb X$ and $z\in \bb R$,
	\[
	\hat{W}(x,z) \geq z.
	\]
	\item \label{prophW001bis} For any $x\in \bb X$,
	\[
	\underset{z\to +\infty}{\lim} \frac{\hat{W}(x,z)}{z} = 1.
	\]
	\item \label{prophW002} For any $x\in \bb X$ and $z\in \bb R$,
	\[
	\bb E_x \left( \hat{W}_n(x,z) \right) \geq \hat{W}(x,z).
	\]
	\item \label{prophW003} For any $x\in \bb X$ and $z\in \bb R$, $\left( \hat{W}_n(x,z) \mathbbm 1_{\left\{ \tau_y > n \right\}} \right)_{n\geq 0}$ is a $\bb P_x$-supermartingale.
\end{enumerate}
\end{lemma}

\begin{proof}
\textit{Claim \ref{prophW001}.} By the Doob optional theorem and the definition of $\hat{T}_z$,
\[
\bb E_x \left( z+M_n \,;\, \hat{T}_z > n \right) = z - \bb E_x \left( z+M_{\hat{T}_z} \,;\, \hat{T}_z \leq n \right) \geq z.
\]
Taking the limit as $n\to+\infty$ and using the point \ref{IFP001} of Proposition \ref{IFP} proves the claim \ref{prophW001}.

\textit{Proof of the claim \ref{prophW001bis}.} By the claim \ref{prophW001}, $\liminf_{z\to +\infty} \hat{W}(x,z)/z \geq 1$. Moreover, by \eqref{WLZX}, for any $n_f \geq 2$,
\[
\underset{z\to \infty}{\limsup} \frac{\hat{W}(x,z)}{z} \leq \left( 1 + \frac{c_{\ee}}{n_f^\ee} \right).
\]
Taking the limit as $n_f \to +\infty$, the claim follows.

\textit{Proof of the claim \ref{prophW002}.} Let $y=z-r(x)$. Using the Markov property, as in the proof of Lemma \ref{MarkovpropforTz},
\begin{align}
	\bb E_x \left( z+M_{n+k} \,;\, \hat{T}_z > n+k \right) =\;& \int_{\bb X \times \bb R} \bb E_{x'} \left( z'+M_n \,;\, \hat{T}_{z'} > n \right) \nonumber\\
	&\quad \times \bb P_x \left( X_k \in \dd x' \,,\, z+M_k \in \dd z' \,,\, \tau_y > k \right) \nonumber\\
	\label{decint}
	&+\int_{\bb X \times \bb R} \bb E_{x'} \left( z'+M_n \,;\, T_{z'} > n \right) \\
	&\quad \times \bb P_x \left( X_k \in \dd x' \,,\, z+M_k \in \dd z' \,,\, \tau_y \leq k \,,\, \hat{T}_z > k \right). \nonumber
\end{align}
We will find the limits as $n \to +\infty$ of the two terms in the right hand side. By Lemmas \ref{intofthemartcond} and \ref{MTR}, $\bb E_{x'} \left( z'+M_n \,;\, \hat{T}_{z'} > n \right) \leq c \left( 1+\abs{y'} + N\left(x'\right) \right)$, with $y'=z'-r(x')$. Moreover by the point \ref{Momdec001} of Hypothesis \ref{Momdec}, $\bb E_x \left( 1+\abs{y+S_k} + N\left(X_k\right) \right) \leq ck \left( 1+\abs{y}+N(x) \right) < +\infty$. So, by the Lebesgue dominated convergence theorem and the point \ref{IFP001} of Proposition \ref{IFP},
\begin{align}
	\int_{\bb X \times \bb R} \bb E_{x'} \left( z'+M_n \,;\, \hat{T}_{z'} > n \right) &\bb P_x \left( X_k \in \dd x' \,,\, z+M_k \in \dd z' \,,\, \tau_y > k \right) \nonumber\\
	&\hspace{2cm} \underset{n\to+\infty}{\longrightarrow} \bb E_x \left( \hat{W} \left( X_k, z+M_k \right) \,;\, \tau_y > k \right).
	\label{limint001}
\end{align}
Moreover, using \eqref{MetMchap}, Lemmas \ref{intofthemartcond} and \ref{MTR} and the point \ref{Momdec001} of Hypothesis \ref{Momdec},
\[
\bb E_{x'} \left( z'+M_n \,;\, T_{z'} > n \right) \leq c \left( 1+\abs{z'} + N\left(x'\right) \right).
\]
Again, by the Lebesgue dominated convergence theorem and the point \ref{IFP001} of Proposition \ref{IFP}, we have
\begin{align}
	&\int_{\bb X \times \bb R} \bb E_{x'} \left( z'+M_n \,;\, T_{z'} > n \right) \bb P_x \left( X_k \in \dd x' \,,\, z+M_k \in \dd z' \,,\, \tau_y \leq k \,,\, \hat{T}_z > k \right) \nonumber\\
	&\hspace{5cm} \underset{n\to+\infty}{\longrightarrow} \bb E_x \left( W \left( X_k, z+M_k \right) \,;\, \tau_y \leq k \,,\, \hat{T}_z > k \right).
	\label{limint002}
\end{align}
Putting together \eqref{decint}, \eqref{limint001}, \eqref{limint002} and using the point \ref{IFP001} of Proposition \ref{IFP},
\begin{align}
	\hat{W}(x,z) =\;& \bb E_x \left( \hat{W} \left( X_k, z+M_k \right) \,;\, \tau_y > k \right) \nonumber\\
	&+ \bb E_x \left( W \left( X_k, z+M_k \right) \,;\, \tau_y \leq k \,,\, \hat{T}_z > k \right).
	\label{marthW}
\end{align}
Now, taking into account \eqref{VWppthW} and the identity $\{ \tau_y > k \} = \{ \tau_y > k,\; \hat{T}_z > k \}$, we obtain the claim \ref{prophW002}.

\textit{Proof of the claim \ref{prophW003}.} By the point \ref{IFP003} of Proposition \ref{IFP}, $W$ is a non-negative function. 
Therefore, using \eqref{marthW},
\[
\hat{W}(x,z) \geq \bb E_x \left( \hat{W} \left( X_1, z+M_1 \right) \,;\, \tau_y > 1 \right),
\]
which implies that $\left( \hat{W}_n(x,z) \mathbbm 1_{\left\{ \tau_y > n \right\}} \right)_{n\geq 0}$ is a supermartingale.
\end{proof}

\begin{lemma}
\label{minpos001}
For any $x\in \bb X,$ $y \in \bb R$ and $z=y+r(x)$,
\[
V(x,y) = \underset{n\to+\infty}{\lim}\, \bb E_x \left( \hat{W}_n(x,z) \,;\, \tau_y > n \right).
\]
\end{lemma}

\begin{proof}
For any $n\in \bb N$, $x\in \bb X,$ $y \in \bb R$ and $z=y+r(x)$,
\[
\bb E_x \left( z+M_n \,;\, \tau_y > n \right) = \bb E_x \left( z+M_n \,;\, \hat{T}_z > n \right) - \bb E_x \left( z+M_n \,;\, \tau_y \leq n \,,\, \hat{T}_z > n  \right).	
\]
By the point \ref{prophW001} of Lemma \ref{prophW}, $z+M_n \leq \hat{W}_n(x,z)$ and therefore
\begin{align}
	\bb E_x \left( z+M_n \,;\, \tau_y > n \right) \geq \bb E_x \left( z+M_n \,;\, \hat{T}_z > n \right) &- \bb E_x \left( \hat{W}_n(x,z)  \right) \nonumber\\
	\label{minorzMn001}
	&+ \bb E_x \left( \hat{W}_n(x,z) \,;\, \tau_y > n  \right).
\end{align}
Moreover, by \eqref{WLZX}, for any $\delta >0$,
\begin{align*}
	\bb E_x \left( \hat{W}_n(x,z) \right) \leq\;& \left( 1+\delta \right) \bb E_x \left( z+M_n \,;\, \hat{T}_z > n \right) + c_{\delta} \bb E_x \left( 1+N\left(X_n\right) \,;\, \hat{T}_z > n \right) \\
	& - (1+\delta) \bb E_x \left( z+M_n \,;\, z+M_n < 0 \,,\, \tau_y > n \right).
\end{align*}
On the event $\{ z+M_n < 0 \,,\, \tau_y > n \}$, by \eqref{decMSX}, it holds $r\left( X_n \right) < z+M_n < 0$. 
Therefore, using Lemma \ref{MTR},
\[
\bb E_x \left( \hat{W}_n(x,z) \right) \leq \left( 1+\delta \right) \bb E_x \left( z+M_n \,;\, \hat{T}_z > n \right) + c_{\delta} \bb E_x\left( 1+N\left(X_n\right) \,;\, \hat{T}_z > n \right).
\]
By the Markov property and \eqref{decexpN},
\begin{align*}
	c_{\delta} \bb E_x\left( 1+N\left(X_n\right) \,;\, \hat{T}_z > n \right) &\leq c_{\delta} \bb E_x \left( 1+\e^{-cn/2} N\left( X_{\pent{n/2}} \right) \,;\, \hat{T}_z > \pent{n/2} \right) \\
	&\leq c_{\delta} \bb P_x \left( \hat{T}_z > \pent{n/2} \right) + \e^{-c_{\delta}n} \left( 1+N(x) \right).
\end{align*}
By Lemma \ref{ExitfinitTmanuz} and the point \ref{IFP001} of Lemma \ref{IFP},
\begin{equation}
	\label{limdeWn}
	\underset{n\to+\infty}{\lim} \bb E_x \left( \hat{W}_n(x,z) \right) \leq \left( 1+\delta \right) \hat{W}(x,z).
\end{equation}
Taking the limit as $n\to +\infty$ in \eqref{minorzMn001} and using the previous bound, we obtain that
\[
V(x,y) \geq -\delta \hat{W}(x,z) + \underset{n\to+\infty}{\lim}\, \bb E_x \left( \hat{W}_n(x,z) \,;\, \tau_y > n \right).
\]
Since this inequality holds true for any $\delta >0$ small enough, we obtain the bound
\begin{equation}
	\label{VppquelimWn}
	\underset{n\to+\infty}{\lim}\, \bb E_x \left( \hat{W}_n(x,z) \,;\, \tau_y > n \right) \leq V(x,y).
\end{equation}
Now, by the point \ref{prophW001} of Lemma \ref{prophW},
\[
\bb E_x \left( z+M_n \,;\, \tau_y > n \right) \leq \bb E_x \left( \hat{W}(X_n,z+M_n) \,;\, \tau_y > n \right).
\]
Taking the limit as $n\to +\infty$ and using the point \ref{IFP001} of Proposition \ref{IFP}, we obtain that
\[
V(x,y) \leq \underset{n\to+\infty}{\lim}\, \bb E_x \left( \hat{W}_n(x,z) \,;\, \tau_y > n \right).
\]
Together with \eqref{VppquelimWn}, this concludes the proof.
\end{proof}

\begin{remark}
Taking the limit in the point \ref{prophW002} of Lemma \ref{prophW}, we can deduce that $\lim_{n\to+\infty} \bb E_x \left( \hat{W}_n(x,z) \right) \geq \hat{W}(x,z)$. Coupling this result with \eqref{limdeWn}, it follows that
\[
\underset{n\to+\infty}{\lim} \bb E_x \left( \hat{W}_n(x,z) \right) = \hat{W}(x,z).
\]
\end{remark}

\begin{lemma}
\label{mWtau}
There exists $\ee_0>0$ such that, for any $\ee \in (0,\ee_0)$, $n\in \bb N$, $x\in \bb X,$ $z\in \bb R$ and $y=z-r(x)$, we have
\[
\bb E_x \left( \hat{W}_n(x,z) \,;\, \tau_y > n \right) \geq \hat{W}(x,z) + c \min(z,0) -c_{\ee} \left( n^{1/2-2\ee} + n^{2\ee} N(x) \right).
\]
\end{lemma}

\begin{proof}
Using the point \ref{prophW002} of Lemma \ref{prophW}, the upper bound for $\hat{W}(x,y)$ given by \eqref{WLZX} and the point \ref{Momdec001} of Hypothesis \ref{Momdec},
\begin{align*}
	\bb E_x \left( \hat{W}_n(x,z) \,;\, \tau_y > n \right) =\;& \bb E_x \left( \hat{W}_n(x,z) \right) - \bb E_x \left( \hat{W}_n(x,z) \,;\, \tau_y \leq n \right) \\
	\geq\;& \hat{W}(x,z) - c \bb E_x \left( z+M_n \,;\, \tau_y \leq n \,,\, \hat{T}_z > n \right) - c  \left( 1+N\left(x\right) \right).
\end{align*}
By the point \ref{Momdec001} of \ref{Momdec}, Lemma \ref{firstupperbound} and the Doob optional stopping theorem,
\begin{align*}
	\bb E_x \left( \hat{W}_n(x,z) \,;\, \tau_y > n \right) \geq \hat{W}(x,z) &- c \left[ \bb E_x \left( z+M_n \,;\, \hat{T}_z > n \right) - \bb E_x \left( z+M_n \,;\, \tau_y > n \right) \right]\\
	&- c  \left( 1+N\left(x\right) \right) \\
	\geq \hat{W}(x,z) &- c \left[ \max(z,0) - z + \bb E_x \left( z+M_{\tau_y} \,;\, \tau_y \leq n \right) \right] \\
	&-c_{\ee} \left( n^{1/2-2\ee} + n^{2\ee} N(x) \right) - c  \left( 1+N\left(x\right) \right).
\end{align*}
By \eqref{decMSX}, $z+M_{\tau_y} \leq r\left( X_{\tau_y} \right)$. 
Therefore, in the same way as in the proof of \eqref{NXT},
\[
\bb E_x \left( z+M_{\tau_y} \,;\, \tau_y \leq n \right) 
\leq c\bb E_x \left( 1+N \left( X_{\tau_y} \right) \,;\, \tau_y \leq n \right) \leq c_{\ee} n^{1/2-2\ee} + c_{\ee} N(x).
\]
Together with the previous bound this implies that
\[
\bb E_x \left( \hat{W}_n(x,z) \,;\, \tau_y > n \right) \geq \hat{W}(x,z) + c \min(z,0) -c_{\ee} \left( n^{1/2-2\ee} + n^{2\ee} N(x) \right).
\]
\end{proof}

\begin{lemma}
\label{mWtau2}
There exists $\ee_0>0$ such that, for any $\ee \in (0,\ee_0)$, $n\geq 2$, $n_f \in \{2, \dots, n\}$, $x\in \bb X$ and $z\in \bb R$, with $y=z-r(x)$, we have
\[
\bb E_x \left( \hat{W}_n(x,z) \,;\, \tau_y > n \right) \geq \bb E_x \left( \hat{W}_{n_f}(x,z) \,;\, \tau_y > n_f \right) - \frac{c_{\ee}}{n_f^\ee} \left( \max(z,0) + 1 + N(x) \right).
\]
\end{lemma}

\begin{proof} Let $\ee \in (0,1).$ Consider the stopping time $\nu_n^\ee = \nu_n + \pent{n^\ee}$. By the Markov property, with $y'=z'-r(x')$,
\begin{align*}
	\bb E_x \left( \hat{W}_n(x,z) \,;\, \tau_y > n \right) &\geq \bb E_x \left( \hat{W}_n(x,z) \,;\, \tau_y > n \,,\, \nu_n^\ee \leq \pent{n^{1-\ee}} \right) \\
	&= \sum_{k=\pent{n^{\ee}}+1}^{\pent{n^{1-\ee}}} \int_{\bb X \times \bb R} \bb E_{x'} \left( \hat{W}_{n-k}(x',z') \,;\, \tau_{y'} > n-k \right) \\
	&\hspace{2cm} \bb P_x \left( X_k \in \dd x' \,,\, z+M_k \in \dd z' \,,\, \tau_y > k \,,\, \nu_n^\ee =k \right).
\end{align*}
Using Lemma \ref{mWtau}, we obtain,
\begin{align*}
	\bb E_x \left( \hat{W}_n(x,z) \,;\, \tau_y > n \right) \geq\;& \bb E_x \left( \hat{W}_{\nu_n^\ee}(x,z) \,;\, \tau_y > \nu_n^\ee \,,\, \nu_n^\ee \leq \pent{n^{1-\ee}} \right)\\
	&+ c \bb E_x \left( \min \left(z+M_{\nu_n^\ee},0\right) \,;\, \tau_y > \nu_n^\ee \,,\, \nu_n^\ee \leq \pent{n^{1-\ee}} \right) \\
	&- c_{\ee} \bb E_x \left( n^{1/2-2\ee} + n^{2\ee} N\left( X_{\nu_n^\ee} \right) \,;\, \tau_y > \nu_n^\ee \,,\, \nu_n^\ee \leq \pent{n^{1-\ee}} \right).
\end{align*}
On the event $\{z+M_{\nu_n^\ee} \leq 0 \,,\, \tau_y > \nu_n^\ee \}$,  by \eqref{decMSX}, 
we have $0 \geq z+M_{\nu_n^\ee} \geq r\left( X_{\nu_n^\ee} \right)$. 
Therefore, by Lemma \ref{MTR},
\begin{align*}
	&\bb E_x \left( \min\left(z+M_{\nu_n^\ee},0\right) \,;\, \tau_y > \nu_n^\ee \,,\, \nu_n^\ee \leq \pent{n^{1-\ee}} \right) \\
	&\hspace{6cm} \geq c\bb E_x \left( 1+N\left( X_{\nu_n^\ee} \right) \,;\, \tau_y > \nu_n^\ee \,,\, \nu_n^\ee \leq \pent{n^{1-\ee}} \right).
\end{align*}
Consequently, using the point \ref{prophW003} of Lemma \ref{prophW} and \eqref{decexpN},
\begin{align}
	\bb E_x \left( \hat{W}_n(x,z) \,;\, \tau_y > n \right) \geq\;& \bb E_x \left( \hat{W}_{\pent{n^{1-\ee}}}(x,z) \,;\, \tau_y > \pent{n^{1-\ee}} \,,\, \nu_n^\ee \leq \pent{n^{1-\ee}} \right) \nonumber\\
	&- c_{\ee} \bb E_x \left( n^{1/2-2\ee} + \e^{-c_{\ee}n^\ee} N\left( X_{\nu_n} \right) \,;\, \tau_y > \nu_n \,,\, \nu_n \leq \pent{n^{1-\ee}} \right) \nonumber\\
	\geq\;& \bb E_x \left( \hat{W}_{\pent{n^{1-\ee}}}(x,z) \,;\, \tau_y > \pent{n^{1-\ee}} \right) - \e^{-c_{\ee} n^\ee} \left( 1+N(x) \right) \nonumber\\
	\label{minordec}
	&- \frac{c_{\ee}}{n^\ee} \underbrace{\bb E_x \left( z+M_{\nu_n} \,;\, \tau_y > \nu_n \,,\, \nu_n \leq \pent{n^{1-\ee}} \right)}_{=:I_1} \\
	&-\underbrace{\bb E_x \left( \hat{W}_{\pent{n^{1-\ee}}}(x,z) \,;\, \tau_y > \pent{n^{1-\ee}} \,,\, \nu_n^\ee > \pent{n^{1-\ee}} \right)}_{=:I_2}. \nonumber
\end{align}

\textit{Bound of $I_1$.} Using the fact that $\{ \tau_y > \nu_n \} \subseteq \{ \hat{T}_z > \nu_n \}$ combined with the positivity of $z+M_{\nu_n}$ and using Lemma \ref{Mnsubmartingale}, we have
\begin{align*}
	I_1 &\leq \bb E_x \left( z+M_{\pent{n^{1-\ee}}} \,;\, \hat{T}_z > \pent{n^{1-\ee}} \,,\, \nu_n \leq \pent{n^{1-\ee}} \right) \\
	&\leq \bb E_x \left( z+M_{\pent{n^{1-\ee}}} \,;\, \hat{T}_z > \pent{n^{1-\ee}} \right) - J_{21}'',
\end{align*}
where $J_{21}''$ is defined in \eqref{decJ21'}. Now, it follows from Lemma \ref{Mnsubmartingale} and the point \ref{IFP001} of Proposition \ref{IFP}, that $( \bb E_x ( z+M_{\pent{n^{1-\ee}}} \,;\, \hat{T}_z > \pent{n^{1-\ee}} ) )_{n\in \bb N}$ is a non-decreasing sequence which converges to $\hat{W}(x,z)$ and so $\bb E_x ( z+M_{\pent{n^{1-\ee}}} \,;\, \hat{T}_z > \pent{n^{1-\ee}} ) \leq \hat{W}(x,z)$. Using \eqref{MajJ21''}, we find that
\begin{equation}
	\label{minorWntau001}
	I_1 \leq \hat{W}(x,z) + \e^{-c_{\ee} n^{\ee}} \left( 1+N(x) \right).
\end{equation}

\textit{Bound of $I_2$.} By \eqref{WLZX},
\begin{align*}
	I_2	&\leq c \bb E_x \left( z+M_{\pent{n^{1-\ee}}} \left( 1- \mathbbm 1_{\left\{ z+M_{\pent{n^{1-\ee}}} < 0 \right\}} \right) \,;\, \hat{T}_z > \pent{n^{1-\ee}} \,,\, \nu_n^\ee > \pent{n^{1-\ee}} \right) \\
	&\hspace{1.5cm}+ c \bb E_x \left( 1+N \left( X_{\pent{n^{1-\ee}}} \right) \,;\, \hat{T}_z > \pent{n^{1-\ee}} \,,\, \nu_n^\ee > \pent{n^{1-\ee}} \right).
\end{align*}
On the event $\{ z+M_{\pent{n^{1-\ee}}} < 0 \,,\, \hat{T}_z > \pent{n^{1-\ee}} \} = \{z+M_{\pent{n^{1-\ee}}} < 0\,,\, \tau_y > \pent{n^{1-\ee}} \}$,  it holds $z+M_{\pent{n^{1-\ee}}} > r \left( X_{\pent{n^{1-\ee}}} \right)$. Therefore, using Lemma \ref{MTR},
\[
I_2 \leq c \bb E_x \left( z+M_{\pent{n^{1-\ee}}} + 1+N \left( X_{\pent{n^{1-\ee}}} \right) \,;\, \hat{T}_z > \pent{n^{1-\ee}} \,,\, \nu_n^\ee > \pent{n^{1-\ee}} \right).
\]
By Lemma \ref{Mnsubmartingale}, $\bb E_x \left( z+M_{\pent{n^{1-\ee}}} \,;\, \hat{T}_z > \pent{n^{1-\ee}} \,,\, \nu_n^\ee > \pent{n^{1-\ee}} \right) \leq J_1$, where $J_1$ is defined in \eqref{UBJ1ET2}. 
Using inequalities \eqref{MajJ1}, \eqref{decexpN} and Lemma \ref{concentnu}, with $m_{\ee} = \pent{n^{1-\ee}}-\pent{n^\ee}$, we obtain
\begin{align}
	I_2 &\leq \e^{-c_\ee n^\ee} \left( 1+N(x) \right) + c \bb E_x \left( 1+ \e^{-cn^{\ee}} N\left( X_{m_{\ee}} \right) \,;\, \hat{T}_z > m_{\ee} \,,\, \nu_n > m_{\ee} \right) \nonumber\\
	&\leq \e^{-c_\ee n^\ee} \left( 1+N(x) \right).
	\label{minorWntau002}
\end{align}
Putting together \eqref{minorWntau002}, \eqref{minorWntau001} and \eqref{minordec} and using \eqref{WLZX}, we have
\begin{align*}
	\bb E_x \left( \hat{W}_n(x,z) \,;\, \tau_y > n \right) \geq\;& \bb E_x \left( \hat{W}_{\pent{n^{1-\ee}}}(x,z) \,;\, \tau_y > \pent{n^{1-\ee}} \right) \\
	&- \frac{c_{\ee}}{n^\ee} \left( \max(z,0) + 1 + N(x) \right).
\end{align*}
By the point \ref{prophW003} of Lemma \ref{prophW}, $( \bb E_x ( \hat{W}_n(x,z) \,;\, \tau_y > n ) )_{n\in \bb N}$ is non-increasing. So using Lemma \ref{lemanalyse2} we conclude that, for any $n\geq 2$ and $n_f \in \{ 2, \dots, n \}$,
\[
\bb E_x \left( \hat{W}_n(x,z) \,;\, \tau_y > n \right) \geq \bb E_x \left( \hat{W}_{n_f}(x,z) \,;\, \tau_y > n_f \right) - \frac{c_{\ee}}{n_f^\ee} \left( \max(z,0) + 1 + N(x) \right).
\]
\end{proof}

\begin{proposition}\ 
\label{PosdeV}
\begin{enumerate}[ref=\arabic*, leftmargin=*, label=\arabic*.]
	\item \label{PosdeV001} For any $\delta \in (0,1)$, $x\in \bb X$ and $y >0$,
	\[
	V(x,y) \geq \left( 1- \delta \right) y - c_{\delta} \left( 1 + N(x) \right).
	\]
	\item \label{PosdeV002} For any $x\in \bb X$,
	\[
	\underset{y\to +\infty}{\lim} \frac{V(x,y)}{y} = 1.
	\]
\end{enumerate}
\end{proposition}

\begin{proof}
\textit{Claim \ref{PosdeV001}.} By Lemmas \ref{mWtau2} and \ref{minpos001}, we immediately have, with $z=y+r(x)$,
\[
V(x,y) \geq \bb E_x \left( \hat{W}_{n_f}(x,z) \,;\, \tau_y > n_f \right) - \frac{c_{\ee}}{n_f^\ee} \left( \max(z,0) + 1 + N(x) \right).
\]
Using the point \ref{prophW001} of Lemma \ref{prophW} and the point \ref{majmart002} of Lemma \ref{majmart},
\begin{align*}
	V(x,y) &\geq \bb E_x \left( z+M_{n_f} \,;\, \tau_y > n_f \right) - \frac{c_{\ee}}{n_f^\ee} \left( \max(z,0) + 1 + N(x) \right) \\
	&\geq z \bb P_x \left( \tau_y > n_f \right) - c \left( \sqrt{n_f} + N(x) \right) - \frac{c_{\ee}}{n_f^\ee} \left( \max(z,0) + 1 + N(x) \right).
\end{align*}
Since, by the Markov inequality,
\[
\bb P_x \left( \tau_y > n_f \right) \geq \bb P_x \left( \underset{1\leq k \leq n_f}{\max} \abs{f\left( X_k \right)} < \frac{y}{n_f} \right) \geq 1-\frac{ cn_f^2 \left( 1+N(x) \right)}{y},
\]
we obtain that, by the definition of $z$,
\begin{equation}
	\label{PosdeV001ter}
	V(x,y) \geq \left( 1-\frac{c_{\ee}}{n_f^\ee} \right) y - c_{\ee} n_f^2 \left( 1 + N(x) \right).
\end{equation}
Let $\delta \in(0,1)$. Taking $n_f$ large enough, we obtain the desired inequality.

\textit{Proof of the claim \ref{PosdeV002}.} By the claim \ref{PosdeV001}, for any $\delta \in (0,1)$ and $x\in \bb X$, we have that $\liminf_{y\to +\infty} V(x,y)/y \geq 1-\delta$. Taking the limit as $\delta \to 0$, we obtain the lower bound. Now by \eqref{VWppthW} and \eqref{WLZX}, for any integer $n_f \geq 2$, $y\in \bb R$ and $z=y+r(x)$,
\[
V(x,y) \leq \hat{W}(x,z) \leq \left( 1 + \frac{c_{\ee}}{n_f^\ee} \right) \left( \max(z,0) + cN(x) \right) + c_{\ee} n_f^{1/2} + \e^{-c_{\ee} n_f^{\ee}} N(x).
\]
Using the definition of $z$, we conclude that
\[
\underset{y\to +\infty}{\limsup} \frac{V(x,y)}{y} \leq \underset{n_f \to +\infty}{\lim} \left( 1 + \frac{c_{\ee}}{n_f^\ee} \right) = 1.
\]
\end{proof}

Now, for any $\gamma > 0$, consider the stopping time:
\[
\zeta_{\gamma} := \inf\left\{ k \geq 1, \, \abs{y+S_k} > \gamma \left( 1+N \left( X_k \right) \right) \right\}.
\]
The control on the tail of $\zeta_{\gamma}$ is given by the following Lemma.
\begin{lemma}
\label{concentmu}
For any $\gamma>0$, $x \in \bb X$, $y \in \bb R$ and $n \geq 1$,
\[
\bb P_x \left( \zeta_\gamma > n \right) \leq \e^{-c_{\gamma} n} \left( 1 + N \left( x \right) \right).
\]
\end{lemma}

\begin{proof}
The reasoning is very close to that of the proof of the Lemma \ref{concentnu}. Let $\gamma > 0$. Consider the integer $l\geq 1$ which will be chosen later. Define $K:= \pent{\frac{n}{2l}}$ and introduce the event $A_{k,y}^{\gamma} := \underset{k' \in \{ 1, \dots, k\}}{\bigcap}\left\{  \abs{y+S_{k'l}} \leq \gamma \left( 1+N \left( X_{k'l} \right) \right)  \right\}$. We have
\[
\bb P_x \left( \zeta_\gamma > n \right) \leq \bb P_x \left( A_{2K,y}^{\gamma} \right).
\]
By the Markov property,
\begin{align}
	\bb P_x \left( A_{2K,y}^{\gamma} \right) = \int_{\bb X \times \bb R} &\int_{\bb X \times \bb R} \bb P_{x''} \left( A_{1,y''}^{\gamma} \right) \bb P_{x'} \left( X_l \in \dd x'' \,,\, y'+S_l \in \dd y''  \,,\, A_{1,y'}^{\gamma} \right) \nonumber\\
	&\times \bb P_x \left( X_{2(K-1)l} \in \dd x' \,,\,  y+S_{2(K-1)l} \in \dd y' \,,\, A_{2(K-1),y}^{\gamma} \right).
	\label{PAg00}
\end{align}
We write
\begin{align*}
	\bb P_{x''} \left( A_{1,y''}^{\gamma} \right) \leq\;& \bb P_{x''} \left( \abs{y''+S_l} \leq 2\gamma \sqrt{l} \right) + \bb P_{x''} \left( N \left( X_l \right) > \sqrt{l} \right)\\
	\leq\;& \bb P_{x''} \left( \frac{-y''}{\sqrt{l}} - 2\gamma \leq \frac{S_l}{\sqrt{l}} \leq \frac{-y''}{\sqrt{l}} + 2\gamma \right) + \bb E_{x''} \left( \frac{N \left( X_l \right)}{\sqrt{l}} \right).
\end{align*}
By Corollary \ref{BerEss} and the point \ref{Momdec001} of Hypothesis \ref{Momdec}, there exists $\ee_0 \in (0,1/4)$ such that, for any $\ee \in (0,\ee_0)$,
\begin{align*}
\bb P_{x''} \left( A_{1,y''}^{\gamma} \right) &\leq \int_{\frac{-y''}{\sqrt{l}} - 2\gamma}^{\frac{-y''}{\sqrt{l}} + 2\gamma} \e^{-\frac{u^2}{2\sigma^2}} \frac{\dd u}{\sqrt{2\pi} \sigma} + \frac{2c_{\ee}}{l^\ee} \left(1+N(x'')\right) + \frac{c}{\sqrt{l}} \left( 1+N \left( x'' \right) \right).
\end{align*}
Set $q_{\gamma} := \int_{- 2\gamma}^{2\gamma} \e^{-\frac{u^2}{2\sigma^2}} \frac{\dd u}{\sqrt{2\pi} \sigma}<1$. From \eqref{PAg00}, we obtain
\begin{align*}
	\bb P_x \left( A_{2K,y}^{\gamma} \right) &\leq \int_{\bb X \times \bb R} \left( q_{\gamma} + \frac{c_{\ee}}{l^{\ee}} + \frac{c_{\ee}}{l^{\ee}} \bb E_{x'} \left( N \left( X_l \right) \right) \right) \nonumber\\
	&\qquad \times \bb P_x \left( X_{2(K-1)l} \in \dd x' \,,\,  y+S_{2(K-1)l} \in \dd y' \,,\, A_{2(K-1),y}^{\gamma} \right) \\
	&\leq \left( q_{\gamma} + \frac{c_{\ee}}{l^{\ee}} \right) \bb P_x \left( A_{2(K-1),y}^{\gamma} \right) + \e^{-c_{\ee}l} \bb E_x \left( N\left( X_{2(K-1)l} \right) \,;\, A_{2(K-1),y}^{\gamma} \right).
\end{align*}
For brevity, set  $p_K= \bb P_x \left( A_{2K,y}^{\gamma} \right)$ and $E_{K} = \bb E_x \left( N\left( X_{2Kl} \right) \,;\, A_{2K,y}^{\gamma} \right)$. Then, the previous inequality can be rewritten as
\begin{equation}
	\label{ineqpK}
	p_K \leq \left( q_{\gamma} + \frac{c_{\ee}}{l^{\ee}} \right) p_{K-1} + \e^{-c_{\ee}l} E_{K-1}.
\end{equation}
Moreover, from \eqref{decexpN}, we have
\begin{equation}
	\label{ineqpK002}
	E_{K} \leq c p_{K-1} + \e^{-c 2l} E_{K-1}.
\end{equation}
Using \eqref{ineqpK} and \eqref{ineqpK002}, we write that
\begin{equation}
	\label{recdouble}
	\begin{pmatrix} p_K \\ E_K \end{pmatrix} \leq A_l \begin{pmatrix} p_{K-1} \\ E_{K-1} \end{pmatrix}
\end{equation}
where
\[
A_l := \begin{pmatrix} q_{\gamma} + \frac{c_{\ee}}{l^{\ee}} & \e^{-c_{\ee}l} \\
c & \e^{-c l} \end{pmatrix} \underset{l\to+\infty}{\longrightarrow} A = \begin{pmatrix} q_{\gamma} & 0 \\
c & 0 \end{pmatrix}.
\]
Since the spectral radius $q_{\gamma}$ of $A$ is less than $1$, we can choose $l=l(\ee,\gamma)$ large enough such that the spectral radius $\rho_{\ee,\gamma}$ of $A_l$ is less than $1$.
Iterating \eqref{recdouble}, we get
\[
p_K \leq c \rho_{\ee,\gamma}^K \max \left( p_1, E_1 \right) \leq c \rho_{\ee,\gamma}^K \left( 1+N(x) \right).
\]
Taking into account that $K \geq c_{\ee,\gamma} n$, we obtain
\[
\bb P_x \left( A_{2K,y}^{\gamma} \right) \leq \e^{-c_{\gamma} n} \left( 1+N(x) \right).
\]
\end{proof}

Now we shall establish some properties of the set $\mathscr{D}_{\gamma}$ introduced in Section \ref{sec-not-res}. It is easy to see that, for any $\gamma > 0$,
\[
\mathscr{D}_{\gamma} = \left\{ (x,y) \in \bb X \times \bb R, \; \exists n_0 \geq 1, \bb P_x \left( \zeta_{\gamma} \leq n_0 \,,\, \tau_y > n_0 \right) > 0 \right\}.
\]
\begin{proposition}
\label{posdeVsurDgamma}\ 
\begin{enumerate}[ref=\arabic*, leftmargin=*, label=\arabic*.]
	\item \label{posdeVsurDgamma001} For any $\gamma_1 \leq \gamma_2$, it holds $\mathscr{D}_{\gamma_1} \supseteq \mathscr{D}_{\gamma_2}$.
	\item \label{posdeVsurDgamma002} For any $\gamma >0$, there exists $c_{\gamma}>0$ such that
	\[
	\mathscr{D}_{\gamma}^c \subseteq \left\{ (x,y) \in \bb X \times \bb R, \; \bb P_x \left( \tau_y > n \right) \leq \e^{-c_{\gamma} n} \left( 1+N(x) \right) \right\}.
	\]
	\item \label{posdeVsurDgamma003} For any $\gamma > 0$, the domain of positivity of the function $V$ is included in $\mathscr{D}_{\gamma}$:
	\[
	\mathscr{D}_+(V) = \left\{ (x,y),\; V(x,y) >0 \right\} \subseteq \mathscr{D}_{\gamma}.
	\]
	\item \label{posdeVsurDgamma004} There exists $\gamma_0 > 0$ such that for any $\gamma \geq \gamma_0$,
	\[
	\mathscr{D}_+(V) = \mathscr{D}_{\gamma}.
	\]
	Moreover,
	\[
	\left\{ (x,y) \in \bb X \times \bb R_+^*,\; y > \frac{\gamma_0}{2} \left( 1+N(x) \right) \right\} \subseteq \mathscr{D}_+(V).
	\]
\end{enumerate}
\end{proposition}

\begin{proof}
\textit{Claim \ref{posdeVsurDgamma001}.} For any $\gamma_1 \leq \gamma_2$, we have $\zeta_{\gamma_1} \leq \zeta_{\gamma_2}$ and the claim \ref{posdeVsurDgamma001} follows.

\textit{Claim \ref{posdeVsurDgamma002}.} Fix $\gamma > 0$. By the definition of $\mathscr{D}_{\gamma}$, for any $(x,y) \in \mathscr{D}_{\gamma}^c$ and $n\geq 1$,
\[
0 = \bb P_x \left( \zeta_{\gamma} \leq n \,,\, \tau_y > n \right) = \bb P_x \left( \tau_y > n \right) - \bb P_x \left( \zeta_{\gamma} > n \,,\, \tau_y > n \right).
\]
From this, using Lemma \ref{concentmu}, we obtain
\[
\bb P_x \left( \tau_y > n \right)= \bb P_x \left( \zeta_{\gamma} > n \,,\, \tau_y > n \right) \leq \bb P_x \left( \zeta_{\gamma} > n \right) \leq \e^{-c_{\gamma} n} \left( 1 + N \left( x \right) \right).
\]

\textit{Claim \ref{posdeVsurDgamma003}.} Fix $\gamma >0$. Using the claim \ref{posdeVsurDgamma002} and Lemma \ref{majmart}, we have, for any $(x,y) \in \mathscr{D}_{\gamma}^c$ and $z=y+r(x),$
\begin{align*}
	\bb E_x \left( z+M_n \,;\, \tau_y > n \right) &\leq \abs{z} \bb P_x \left( \tau_y > n \right) +  \bb E_x^{1/2} \left( \abs{M_n}^2 \right) \bb P_x^{1/2} \left( \tau_y > n \right)\\
	&\leq \abs{z} \left( 1 + N \left( x \right) \right) \e^{-c_{\gamma} n} +  c \sqrt{n} \left( 1+N(x) \right)^{3/2} \e^{-c_{\gamma} n}.
\end{align*}
Taking the limit when $n\to +\infty$, by the point \ref{IFP001} of Proposition \ref{IFP}, we get
\[
V(x,y) = 0,
\]
and we conclude that $\mathscr{D}_{\gamma}^c \subseteq \mathscr{D}_+(V)^c$.

\textit{Claim \ref{posdeVsurDgamma004}.} By the point \ref{PosdeV001} of Proposition \ref{PosdeV}, taking $\delta = 1/2$, there exists $\gamma_0 > 0$ such that, for any $x\in \bb X$ and $y>0,$
\begin{equation}
	\label{PosdeV001bis}
	V(x,y) \geq \frac{y}{2} - \frac{\gamma_0}{4} \left( 1 + N(x) \right).
\end{equation}
Now, fix $(x,y) \in \mathscr{D}_{\gamma_0}$ and let $n_0\geq 1$ be an integer such that $\bb P_x \left( \zeta_{\gamma_0} \leq n_0 \,,\, \tau_y > n_0 \right) > 0$. By the point \ref{IFP004} of Proposition \ref{IFP},
\begin{align*}
	V(x,y) &= \bb E_x \left( V\left( X_{n_0}, y+S_{n_0} \right) \,;\, \tau_y > n_0 \right) \\
	&\geq \bb E_x \left( V\left( X_{n_0}, y+S_{n_0} \right) \,;\, \tau_y > n_0 \,,\, \zeta_{\gamma_0} \leq n_0 \right).
\end{align*}
By the Doob optional stopping theorem, \eqref{PosdeV001bis} and the definition of $\zeta_{\gamma_0}$,
\begin{align*}
	V(x,y) &\geq \bb E_x \left( V\left( X_{\zeta_{\gamma_0}}, y+S_{\zeta_{\gamma_0}} \right) \,;\, \tau_y > \zeta_{\gamma_0} \,,\, \zeta_{\gamma_0} \leq n_0 \right) \\
	&\geq \frac{1}{2} \bb E_x \left( y+S_{\zeta_{\gamma_0}} - \frac{\gamma_0}{2} \left( 1+ N\left( X_{\zeta_{\gamma_0}} \right) \right) \,;\, \tau_y > \zeta_{\gamma_0} \,,\, \zeta_{\gamma_0} \leq n_0 \right)\\
	&\geq \frac{1}{2} \bb E_x \left( \frac{\gamma_0}{2} \left( 1+ N\left( X_{\zeta_{\gamma_0}} \right) \right) \,;\, \tau_y > \zeta_{\gamma_0} \,,\, \zeta_{\gamma_0} \leq n_0 \right)\\
	&\geq \frac{\gamma_0}{4} \bb P_x \left( \tau_y > n_0 \,,\, \zeta_{\gamma_0} \leq n_0 \right).
\end{align*}
Now, since $n_0$ has been chosen such that the last probability is strictly positive, we get that $V(x,y) > 0$. This proves that $\mathscr{D}_{\gamma_0} \subseteq \mathscr{D}_+(V)$. Using the claims \ref{posdeVsurDgamma001} and \ref{posdeVsurDgamma003}, for any $\gamma \geq \gamma_0$, we obtain that $\mathscr{D}_{\gamma} \subseteq \mathscr{D}_{\gamma_0} \subseteq \mathscr{D}_+(V) \subseteq \mathscr{D}_{\gamma}$ and so $\mathscr{D}_{\gamma} = \mathscr{D}_{\gamma_0} = \mathscr{D}_+(V)$. Using \eqref{PosdeV001bis} proves the second assertion of the claim \ref{posdeVsurDgamma004}.
\end{proof}

\textbf{Proof of Theorem \ref{thonV}.} The claim \ref{thonV001} is proved by the point \ref{IFP001} of Proposition \ref{IFP}~; the claim \ref{thonV002} is proved by the point \ref{IFP004} of Proposition \ref{IFP}~; the claim \ref{thonV003} is proved by the points \ref{IFP002} and \ref{IFP003} of Proposition \ref{IFP} and by Proposition \ref{PosdeV}~; the claim \ref{thonV004} is proved by the point \ref{posdeVsurDgamma004} of Proposition \ref{posdeVsurDgamma}.

\section{Asymptotic for the exit time}
\label{AsExTi}

\subsection{Preliminary results} 

\begin{lemma}
\label{SurE1etE2}
There exists $\ee_0 >0$ such that, for any $\ee\in (0,\ee_0)$, $x\in \bb X$, $y\in \bb R$ and $z=y+r(x),$
\begin{align*}
	E_1 &:= \bb E_x \left( z+M_{\nu_n} \,;\, \tau_y > \nu_n \,,\, \nu_n \leq \pent{n^{1-\ee}} \right) \leq c_{\ee} \left( 1+\max(y,0)+N(x) \right), \quad \forall n \geq 1, \\
	E_2 &:= \bb E_x \left( z+M_{\nu_n^{\ee^2}} \,;\, \tau_y > \nu_n^{\ee^2} \,,\, \nu_n^{\ee^2} \leq \pent{n^{1-\ee}} \right) \underset{n\to \infty}{\longrightarrow} V(x,y).
\end{align*}
Moreover, for any $n\geq 1$, $\ee\in (0,\ee_0)$, $x\in \bb X$ and $y\in \bb R,$
\[
\abs{E_2 - V(x,y)} \leq \frac{c_{\ee}}{n^{\ee/8}} \left( 1+\max(y,0)+N(x) \right).
\]
\end{lemma}

\begin{proof}
Using the fact $\{ \tau_y > \nu_n \} \subseteq \{ \hat{T}_z > \nu_n \}$ and Lemma \ref{Mnsubmartingale},
\[
E_1 \leq \bb E_x \left( z+M_{\pent{n^{1-\ee}}} \,;\, \hat{T}_z > \pent{n^{1-\ee}} \right) - J_{21}'',
\]
where $J_{21}''$ is defined in \eqref{decJ21'} and by \eqref{MajJ21''} the quantity $-J_{21}''$ does not exceed $\e^{-c_{\ee}n^{\ee}} ( 1+N(x) )$. Again, by Lemma \ref{Mnsubmartingale} and the point \ref{IFP001} of Proposition \ref{IFP}, we have that $( \bb E_x ( z+M_{n} \,;\, \hat{T}_z > n ) )_{n\geq 0}$ is a non-decreasing sequence which converges to $\hat{W}(x,z)$. So, using the point \ref{IFP003} of Proposition \ref{IFP} and the fact that $z=y+r(x)$,
\begin{equation}
	\label{MajdeE1}
	E_1 \leq \hat{W}(x,z) + \e^{-c_{\ee} n^{\ee}} \left( 1+N(x) \right) \leq c_{\ee} \left( 1+\max(y,0)+N(x) \right).
\end{equation}

By the point \ref{IFP004} of Proposition \ref{IFP}, we have
\begin{align*}
	V(x,y) =\;& \bb E_x \left( V \left( X_n, y+S_n \right) \,;\, \tau_y > n \,,\, \nu_n^{\ee^2} \leq \pent{n^{1-\ee}} \right) \\
	&+ \bb E_x \left( V \left( X_n, y+S_n \right) \,;\, \tau_y > n \,,\, \nu_n^{\ee^2} > \pent{n^{1-\ee}} \right).
\end{align*}
Using the point \ref{IFP003} of Proposition \ref{IFP}, for any $n_f \geq 2$,
\begin{align*}
	V(x,y) \leq\;& \bb E_x \left( V \left( X_{\nu_n^{\ee^2}}, y+S_{\nu_n^{\ee^2}} \right) \,;\, \tau_y > \nu_n^{\ee^2} \,,\, \nu_n^{\ee^2} \leq \pent{n^{1-\ee}} \right) \\
	&+ c\bb E_x \left( \max\left( z+M_n, 0 \right) + 1 + N\left( X_n \right) \,;\, \tau_y > n \,,\, \nu_n^{\ee^2} > \pent{n^{1-\ee}} \right) \\
	\leq\;& \left( 1+\frac{c_{\ee}}{n_f^{\ee}} \right) E_2 + c_{\ee} \bb E_x \left( \sqrt{n_f}+N \left( X_{\nu_n^{\ee^2}} \right) \,;\, \tau_y > \nu_n^{\ee^2} \,,\, \nu_n^{\ee^2} \leq \pent{n^{1-\ee}} \right) \\
	&\underbrace{-c_{\ee} \bb E_x \left( z+M_{\nu_n^{\ee^2}} \,;\, z+M_{\nu_n^{\ee^2}} < 0 \,,\, \tau_y > \nu_n^{\ee^2} \,,\, \nu_n^{\ee^2} \leq \pent{n^{1-\ee}} \right)}_{=J_{22}'(\ee^2)} \\
	&+ c\bb E_x \left( z+M_n + \abs{r\left( X_n \right)} +1+N\left( X_n \right) \,;\, \tau_y > n \,,\, \nu_n^{\ee^2} > \pent{n^{1-\ee}} \right).
\end{align*}
From the previous bound, using the Markov property, the bound \eqref{decexpN} and the approximation \eqref{decMSX}, we get 
\begin{align*}
	V(x,y) \leq\;& \left( 1+\frac{c_{\ee}}{n_f^{\ee}} \right) E_2 + J_{22}'(\ee^2) + c\underbrace{\bb E_x \left( z+M_n \,;\, \hat{T}_z > n \,,\, \nu_n^{\ee^2} > \pent{n^{1-\ee}} \right)}_{=J_1(\ee^2)} \\
	&+ c_{\ee} \bb E_x \left( \sqrt{n_f}+e^{-cn^{\ee^2}} N \left( X_{\nu_n} \right) \,;\, \tau_y > \nu_n \,,\, \nu_n \leq \pent{n^{1-\ee}} \right) \\
	&+ c\bb E_x \left( 1+e^{-c_{\ee}n} N\left( X_{\pent{n^{1-\ee}}} \right) \,;\, \tau_y > \pent{n^{1-\ee}} \,,\, \nu_n^{\ee^2} > \pent{n^{1-\ee}} \right).
\end{align*}
Proceeding in the same way as for the bound \eqref{MajJ22'},
\begin{align*}
	J_{22}'(\ee^2) &\leq c_{\ee} \bb E_x \left( 1+ \e^{-c n^{\ee^2}} N \left( X_{\nu_n} \right) \,;\, \tau_y > \nu_n \,,\, \nu_n \leq \pent{n^{1-\ee}} \right) \\
	&\leq \frac{c_{\ee}}{n^{1/2-\ee}} E_1 + \e^{-c_{\ee} n^{\ee^2}} \left( 1+N(x) \right).
\end{align*}
Moreover, similarly as for the bound \eqref{MajJ1}, we have
\[
J_1(\ee^2) \leq \e^{-c_{\ee}n^{\ee^2}} \left( 1+N(x) \right).
\]
Taking into account these bounds and using Lemma \ref{concentnu},
\begin{equation}
	\label{MajorantE2}
	V(x,y) \leq \left( 1+\frac{c_{\ee}}{n_f^{\ee}} \right) E_2 + \frac{c_{\ee} \sqrt{n_f}}{n^{1/2-\ee}} E_1 + \e^{-c_{\ee} n^{\ee^2}} \left( 1+N(x) \right).
\end{equation}
Analagously, by \eqref{PosdeV001ter} and \eqref{decMSX}, we have the lower bound
\begin{align}
	V(x,y) \geq\;& \bb E_x \left( V \left( X_{\nu_n^{\ee^2}}, y+S_{\nu_n^{\ee^2}} \right) \,;\, \tau_y > \nu_n^{\ee^2} \,,\, \nu_n^{\ee^2} \leq \pent{n^{1-\ee}} \right) \nonumber\\
	\geq\;& \left( 1-\frac{c_{\ee}}{n_f^{\ee}} \right) E_2 - c_{\ee} n_f^2 \bb E_x \left( 1+N \left( X_{\nu_n^{\ee^2}} \right) \,;\, \tau_y > \nu_n^{\ee^2} \,,\, \nu_n^{\ee^2} \leq \pent{n^{1-\ee}} \right) \nonumber\\
	\label{MinorantE2}
	\geq\;& \left( 1-\frac{c_{\ee}}{n_f^{\ee}} \right) E_2 - \frac{c_{\ee}n_f^2}{n^{1/2-\ee}} E_1 - n_f^2 \e^{-c_{\ee} n^{\ee^2}} \left( 1+N(x) \right).
\end{align}
Taking $n_f=n^{1/4-\ee}$ in \eqref{MinorantE2} and \eqref{MajorantE2}, we conclude that, for any $\ee \in (0,1/8)$,
\[ \abs{V(x,y) - E_2} \leq \frac{c_{\ee}}{n^{\ee/8}} E_2 + \frac{c_{\ee}}{n^{\ee}} \left( E_1+1+N(x) \right). \]
Again, using  \eqref{MinorantE2},
\[
\abs{V(x,y) - E_2} \leq \frac{c_{\ee}}{n^{\ee/8}} V(x,y) + \frac{c_{\ee}}{n^{\ee}} \left( E_1+1+N(x) \right).
\]
Finally, employing \eqref{MajdeE1} and \eqref{encadrV},
\[ \abs{V(x,y) - E_2} \leq \frac{c_{\ee}}{n^{\ee/8}} \left( 1+\max(y,0)+N(x) \right). \] 
\end{proof}

\begin{lemma}
\label{taupptrn}
There exists $\ee_0 > 0$ such that, for any $\ee \in (0,\ee_0)$, $x\in \bb X$, $y\in \bb R$ and $n\geq 1$,
\[
\bb P_x \left( \tau_y > n \right) \leq \frac{c_\ee}{n^{1/2-\ee}} \left( 1+\max(y,0)+N(x) \right).
\]
Moreover, summing this bound, for any $\ee \in (0,\ee_0)$, $x\in \bb X$, $y\in \bb R$ and $n\geq 1$, we have 
\[
\sum_{k=1}^{\pent{n^{1-\ee}}} \bb P_x \left( \tau_y > k \right) \leq c_\ee \left( 1+\max(y,0)+N(x) \right) n^{1/2+\ee}.
\]
\end{lemma}

\begin{proof}
Using Lemma \ref{concentnu} and Lemma \ref{SurE1etE2}, with $z=y+r(x)$,
\begin{align*}
	\bb P_x \left( \tau_y > n \right) &\leq \bb P_x \left( \tau_y > n \,,\, \nu_n \leq \pent{n^{1-\ee}} \right) + \bb P_x \left( \hat{T}_z > n \,,\, \nu_n > \pent{n^{1-\ee}} \right) \\
	&\leq \bb E_x \left( \frac{z+M_{\nu_n}}{n^{1/2-\ee}} \,;\, \tau_y > n \,,\, \nu_n \leq \pent{n^{1-\ee}} \right) + \e^{-c_\ee n^\ee} \left( 1+N(x) \right) \\
	&\leq \frac{c_{\ee}}{n^{1/2-\ee}} \left( 1+\max(y,0)+N(x) \right).
\end{align*}
\end{proof}

\begin{lemma}
\label{MajE3}
There exists $\ee_0 > 0$ such that, for any $\ee \in (0,\ee_0)$, $x\in \bb X$, $y\in \bb R$ and $z=y+r(x)$,
\[
E_3 := \bb E_x \left( z+M_{\nu_n} \,;\, z+M_{\nu_n} > n^{1/2-\ee/2} \,,\, \tau_y > \nu_n \,,\, \nu_n \leq \pent{n^{1-\ee}} \right) \underset{n\to +\infty}{\longrightarrow} 0.
\]
More precisely, for any $n\geq 1$, $\ee \in (0,\ee_0)$, $x\in \bb X$, $y\in \bb R$ and $z=y+r(x)$,
\[
E_3 \leq c_{\ee}\frac{\max(y,0) +  \left( 1+ y\mathbbm 1_{\{y> n^{1/2-2\ee}\}} +N(x) \right)^2}{n^{\ee}}.
\]
\end{lemma}

\begin{proof}
Notice that when $\nu_n \neq 1$ the following inclusion holds: $\{ z+M_{\nu_n} > n^{1/2-\ee/2} \} \subseteq \{ \xi_{\nu_n} > n^{1/2-\ee/2} - n^{1/2-\ee} \geq c_{\ee} n^{1/2-\ee/2} \}$. Therefore,
\begin{align}
	E_3 \leq\;& \underbrace{\bb E_x \left( z+M_{\nu_n} \,;\, \nu_n \leq 2\pent{n^{\ee}} \right)}_{=:E_{30}} \nonumber\\
	\label{decdeE3}
	&+ \underbrace{\sum_{k=2\pent{n^{\ee}}+1}^{\pent{n^{1-\ee}}} \bb E_x \left( z+M_k \,;\, \xi_k > c_{\ee} n^{1/2-\ee/2} \,,\, \tau_y > k \,,\, \nu_n=k \right)}_{=:E_{31}}.
\end{align}

\textit{Bound of $E_{30}$.}
For $y\leq n^{1/2-2\ee}$, by the Markov inequality and Lemma \ref{majmart},
\begin{align*}
	\bb P_x \left( \nu_n \leq 2 \pent{n^{\ee}} \right) &\leq \sum_{k=1}^{2\pent{n^{\ee}}} \bb P_x \left( r(x)+M_k > c_{\ee} n^{1/2-\ee} \right) \leq \frac{c_{\ee} \left( 1+N(x) \right)}{n^{1/2-3\ee}}.
\end{align*}
For $y > n^{1/2-2\ee}$, in the same way, we have $\bb P_x \left( \nu_n \leq 2 \pent{n^{\ee}} \right) \leq \frac{c_{\ee} \left( 1+y+N(x) \right)}{n^{1/2-3\ee}}$. Putting together these bounds, we get, for any $y \in \bb R$,
\begin{equation}
	\label{nupetit}
	\bb P_x \left( \nu_n \leq 2 \pent{n^{\ee}} \right) \leq \frac{c_{\ee} \left( 1+ y\mathbbm 1_{\{y> n^{1/2-2\ee}\}}+N(x) \right)}{n^{1/2-3\ee}}.
\end{equation}
Using Lemma \ref{majmart},
\begin{align}
	E_{30} &\leq z \bb P_x \left( \nu_n \leq 2 \pent{n^{\ee}} \right) + \sum_{k=1}^{2\pent{n^{\ee}}} \bb E_x^{1/2} \left( \abs{M_k}^2 \right) \bb P_x^{1/2} \left( \nu_n \leq 2\pent{n^\ee} \right) \nonumber\\
	\label{E30to0}
	&\leq \frac{c_{\ee} \left( 1+ y\mathbbm 1_{\{y> n^{1/2-2\ee}\}} +N(x) \right)^2}{n^{\ee}}.
\end{align}

\textit{Bound of $E_{31}$.} Changing the index of summation ($j=k-\pent{n^\ee}$) and using the Markov property,
\begin{align}
	E_{31} \leq\;& \sum_{j=\pent{n^{\ee}}+1}^{\pent{n^{1-\ee}}} \int_{\bb X \times \bb R} \max(z',0) \bb P_{x'} \left( \xi_{\pent{n^\ee}} > c_{\ee} n^{1/2-\ee/2} \right)  \nonumber\\
	&\underbrace{\hspace{3cm} \times \bb P_x \left( X_j \in \dd x' \,,\, z+M_j \in \dd z' \,,\, \tau_y > j \right)}_{=:E_{32}} \nonumber\\
	\label{decdeE31}
	&+\sum_{j=\pent{n^{\ee}}+1}^{\pent{n^{1-\ee}}} \int_{\bb X \times \bb R} \bb E_{x'}^{1/2} \left( \abs{M_{\pent{n^\ee}}}^2 \right) \bb P_{x'}^{1/2} \left( \xi_{\pent{n^\ee}} > c_{\ee} n^{1/2-\ee/2} \right) \\
	&\underbrace{\hspace{3cm} \times \bb P_x \left( X_j \in \dd x' \,,\, z+M_j \in \dd z' \,,\, \tau_y > j \right).}_{=:E_{33}} \nonumber
\end{align}

\textit{Bound of $E_{32}$.}
Using \eqref{majdesxi000}, the Markov inequality and \eqref{decexpNl} with $l=\pent{c_{\ee} n^{1/2-\ee/2}}$,
\begin{align*}
	\bb P_{x'} \left( \xi_{\pent{n^\ee}} > c_{\ee} n^{1/2-\ee/2} \right) \leq\;& \bb P_{x'} \left( N \left( X_{\pent{n^\ee}} \right) > c_{\ee} n^{1/2-\ee/2} \right) \\
	&+ \bb P_{x'} \left( N \left( X_{\pent{n^\ee}-1} \right) > c_{\ee} n^{1/2-\ee/2} \right) \\
	\leq\;& \frac{1}{l} \bb E_{x'} \left( N_l \left( X_{\pent{n^\ee}} \right) \right) + \frac{1}{l} \bb E_{x'} \left( N_l \left( X_{\pent{n^\ee}-1} \right) \right) \\
	\leq\;& \frac{c}{l^{2+\beta}} + \frac{1}{l} \e^{-c n^{\ee}} \left( 1+N(x') \right).
\end{align*}
Choosing $\ee>0$ small enough we find that
\begin{equation}
	\label{queuedexi}
	\bb P_{x'} \left( \xi_{\pent{n^\ee}} > c_{\ee} n^{1/2-\ee/2} \right) \leq \frac{c_{\ee}}{n^{1+\beta/4}} + \e^{-c_{\ee} n^{\ee}} N(x').
\end{equation}
By the definition of $E_{32}$ in \eqref{decdeE31},
\begin{align*}
	E_{32} \leq\;& \frac{c_{\ee}}{n^{1+\beta/4}} \sum_{j=\pent{n^{\ee}}+1}^{\pent{n^{1-\ee}}} \left[ \bb E_x \left( z+M_j \,;\, \tau_y > j \right) + \bb E_x \left( \abs{r\left( X_j \right)} \right) \right] \\
	&+ \e^{-c_{\ee} n^{\ee}} \sum_{j=\pent{n^{\ee}}+1}^{\pent{n^{1-\ee}}} \left[ \max(z,0) \bb E_x \left( N \left( X_j \right) \right) +  \bb E_x^{1/2} \left( \abs{M_j}^2 \right) \bb E_x^{1/2} \left( N \left( X_j \right)^2 \right) \right].
\end{align*}
Using \eqref{VnPPWn2}, Lemma \ref{majmart} and the point \ref{Momdec001} of Hypothesis \ref{Momdec}, we find that
\begin{equation}
	\label{E32to0}
	E_{32} \leq c_{\ee}\frac{\max(y,0) +  \left( 1+ y\mathbbm 1_{\{y> n^{1/2-2\ee}\}} +N(x) \right)  \left( 1+N(x) \right)}{n^{\beta/4}}.
\end{equation}

\textit{Bound of $E_{33}$.} Using \eqref{queuedexi} and Lemma \ref{majmart}, we have
\[
E_{33} \leq \sum_{j=\pent{n^{\ee}}+1}^{\pent{n^{1-\ee}}} \bb E_x \left( n^{\ee/2} \left( 1+N \left( X_j \right) \right) \left( \frac{c_{\ee}}{n^{1/2+\beta/8}} + \e^{-c_{\ee} n^{\ee}} N\left( X_j \right)^{1/2} \right) \,;\, \tau_y >j \right).
\]
By the Markov property,
\[
E_{33} \leq \e^{-c_{\ee} n^{\ee}} \left( 1+ N(x) \right)^{3/2} + \frac{c_{\ee}}{n^{1/2+\beta/8-\ee/2}} \sum_{j=1}^{\pent{n^{1-\ee}}} \bb E_x \left( 1+ \e^{-c n^\ee} N \left( X_j \right) \,;\, \tau_y > j \right).
\]
Using Lemma \ref{taupptrn},
\begin{equation}
	\label{E33to0}
	E_{33} \leq c_{\ee} \frac{\max(y,0)+ \left( 1+N(x) \right)^{3/2}}{n^{\beta/8-3\ee/2}}.
\end{equation}
With \eqref{E33to0}, \eqref{E32to0} and \eqref{decdeE31}, for $\ee > 0$ small enough, we find that
\[
E_{31} \leq c_{\ee} \frac{\max(y,0) +  \left( 1+ y\mathbbm 1_{\{y> n^{1/2-2\ee}\}} +N(x) \right)  \left( 1+N(x) \right)}{n^{\ee}}.
\]
This bound, together with \eqref{E30to0} and \eqref{decdeE3}, proves the lemma.
\end{proof}

\begin{lemma}
\label{SurE4}
There exists $\ee_0 > 0$ such that, for any $\ee \in (0,\ee_0)$, $x\in \bb X$, $y\in \bb R$ and $z=y+r(x)$,
\[
E_4 := \bb E_x \left( z+M_{\nu_n^{\ee^2}} \,;\, z+M_{\nu_n^{\ee^2}} > n^{1/2-\ee/4} \,,\, \tau_y > \nu_n^{\ee^2} \,,\, \nu_n^{\ee^2} \leq \pent{n^{1-\ee}} \right) \underset{n\to +\infty}{\longrightarrow} 0.
\]
More precisely, for any $n\geq 1$, $\ee \in (0,\ee_0)$, $x\in \bb X$, $y\in \bb R$ and $z=y+r(x)$,
\[
E_4 \leq c_{\ee}\frac{\max(y,0) +  \left( 1+ y\mathbbm 1_{\{y> n^{1/2-2\ee}\}} +N(x) \right)^2}{n^{\ee/2}}.
\]
\end{lemma}

\begin{proof}
We will apply Lemma \ref{MajE3}. For this we write
\begin{align}
	E_4 =\;& \bb E_x \left( z+M_{\nu_n^{\ee^2}} \,;\, z+M_{\nu_n^{\ee^2}} > n^{1/2-\ee/4} \,,\, z+M_{\nu_n} > n^{1/2-\ee/2} \,,\, \right. \nonumber\\
	 &\underbrace{\hspace{8cm} \left. \tau_y > \nu_n^{\ee^2} \,,\, \nu_n^{\ee^2} \leq \pent{n^{1-\ee}} \right)}_{=:E_{41}} \nonumber\\
	\label{decdeE4}
	+& \bb E_x \left( z+M_{\nu_n^{\ee^2}} \,;\, z+M_{\nu_n^{\ee^2}} > n^{1/2-\ee/4} \,,\, z+M_{\nu_n} \leq n^{1/2-\ee/2} \,,\, \right. \\
	&\underbrace{\hspace{8cm} \left. \tau_y > \nu_n^{\ee^2} \,,\, \nu_n^{\ee^2} \leq \pent{n^{1-\ee}} \right)}_{=:E_{42}}. \nonumber
\end{align}

\textit{Bound of $E_{41}$.} By the Markov property,
\begin{align*}
	E_{41} &= \sum_{k=1}^{\pent{n^{1-\ee}}-\pent{n^{\ee^2}}} \int_{\bb X \times \bb R} \bb E_{x'} \left( z' + M_{\pent{n^{\ee^2}}} \,;\, z' + M_{\pent{n^{\ee^2}}} > n^{1/2-\ee/4} \,,\, \tau_{y'} > \pent{n^{\ee^2}} \right) \\
	&\qquad \times \bb P_x \left( X_k \in \dd x' \,,\, z+M_k \in \dd z' \,,\, z+M_k > n^{1/2-\ee/2} \,,\, \tau_y > k \,,\, \nu_n = k \right),
\end{align*}
where $y'=z'-r(x')$. Moreover, for any $x' \in \bb X$, $z' \in \bb R$, using \eqref{VnPPWn2}, we have
\begin{align*}
	\bb E_{x'} &\left( z' + M_{\pent{n^{\ee^2}}} \,;\, z' + M_{\pent{n^{\ee^2}}} > n^{1/2-\ee/4} \,,\, \tau_{y'} > \pent{n^{\ee^2}} \right) \\
	&\qquad \leq \bb E_{x'} \left( z' + M_{\pent{n^{\ee^2}}} \,;\, z' + M_{\pent{n^{\ee^2}}} > 0 \,,\, \tau_{y'} > \pent{n^{\ee^2}} \right) \\
	&\qquad \leq \bb E_{x'} \left( z' + M_{\pent{n^{\ee^2}}} \,;\, \tau_{y'} > \pent{n^{\ee^2}} \right) + \bb E_{x'} \left( \abs{r\left( X_{n^{\ee^2}} \right)} \right) \\
	&\qquad \leq c_{\ee} \max(z',0) + c_{\ee} \left( 1+N(x') \right).
\end{align*}
Consequently,
\begin{align}
	E_{41} &\leq c_{\ee} E_3 + c_{\ee} \bb E_x \left( 1 + N \left( X_{\nu_n} \right) \,;\, z+M_{\nu_n} > n^{1/2-\ee/2} \,,\, \tau_y > \nu_n \,,\, \nu_n \leq \pent{n^{1-\ee}} \right) \nonumber\\
	&\leq 2c_{\ee} E_3 + c_{\ee} \bb E_x \left( N \left( X_{\nu_n} \right) \,;\, N \left( X_{\nu_n} \right) > n^{1/2-\ee} \,,\, \tau_y > \nu_n \,,\, \nu_n \leq \pent{n^{1-\ee}} \right)  \nonumber\\
	&\qquad +c_{\ee} \bb E_x \left( n^{1/2-\ee} \,;\, N \left( X_{\nu_n} \right) \leq n^{1/2-\ee} \,,\, z+M_{\nu_n} > n^{1/2-\ee/2} \,,\, \right. \nonumber\\
	&\hspace{8cm} \left.  \tau_y > \nu_n \,,\, \nu_n \leq \pent{n^{1-\ee}} \right) \nonumber\\
	\label{decdeE41}
	 &\leq 3c_{\ee} E_3 + c_{\ee} \underbrace{\bb E_x \left( N \left( X_{\nu_n} \right) \,;\, N \left( X_{\nu_n} \right) > n^{1/2-\ee} \,,\, \tau_y > \nu_n \,,\, \nu_n \leq \pent{n^{1-\ee}} \right)}_{=:E_{41}'}.
\end{align}
Denoting $l=\pent{n^{1/2-\ee}}$ and using the point \ref{Momdec001} of \ref{Momdec} and \eqref{decexpNl}, we have
\begin{align*}
	E_{41}' \leq\;& \bb E_x \left( \frac{N \left( X_{\nu_n} \right)^2}{n^{1/2-\ee}} \,;\, \nu_n \leq \pent{n^{\ee}} \right) + \sum_{k=\pent{n^\ee}+1}^{\pent{n^{1-\ee}}} \bb E_x \left( N_l \left( X_k \right) \,;\, \tau_y > k \,,\, \nu_n=k \right) \\
	\leq\;& \frac{c n^{\ee} \left( 1+N(x) \right)^2}{n^{1/2-\ee}} + \sum_{k=1}^{\pent{n^{1-\ee}}}  \left[\frac{c}{l^{1+\beta}}\bb P_x \left( \tau_y > k \right) + \e^{-cn^{\ee}} \bb E_x \left( 1+N \left( X_k \right) \right) \right].
\end{align*}
Using Lemma \ref{taupptrn} and taking $\ee>0$ small enough,
\begin{equation}
	\label{MajE41'}
	E_{41}' \leq c_{\ee}\frac{\max(y,0) + \left( 1+N(x) \right)^2}{n^{\min(1,\beta)/4}}.
\end{equation}
In conjunction with Lemma \ref{MajE3}, from \eqref{decdeE41} we obtain that, for some $\ee >0$,
\begin{equation}
	\label{E41to0}
	E_{41} \leq c_{\ee}\frac{\max(y,0) +  \left( 1+ y\mathbbm 1_{\{y> n^{1/2-2\ee}\}} +N(x) \right)^2}{n^{\ee}}.
\end{equation}

\textit{Bound of $E_{42}$.} For any $z' \in (0, n^{1/2-\ee/2}]$, we have 
\[
\left(z'+M_{\pent{n^{\ee^2}}} \right) \bb P_{x'} ( z'+M_{\pent{n^{\ee^2}}} > n^{1/2-\ee/4} ) \leq z' \bb P_{x'} ( M_{\pent{n^{\ee^2}}} > c_{\ee} n^{1/2-\ee/4} ) + \abs{M_{\pent{n^{\ee^2}}}}.
\]
Therefore, by the Markov property,
\begin{align}
	E_{42} \leq\;& \int_{\bb X \times \bb R} z' \bb P_{x'} \left( M_{\pent{n^{\ee^2}}} > c_{\ee} n^{1/2-\ee/4} \right) \bb P_x \left( X_{\nu_n} \in \dd x' \,,\, z+M_{\nu_n} \in \dd z' \,,\, \right. \nonumber\\
	&\underbrace{\hspace{4cm} \left.  z+M_{\nu_n} \leq n^{1/2-\ee/2} \,,\, \tau_y > \nu_n \,,\, \nu_n \leq \pent{n^{1-\ee}} \right)}_{=:E_{43}} \nonumber\\
	\label{decdeE42}
	&+ \int_{\bb X \times \bb R} \bb E_{x'} \left( \abs{M_{\pent{n^{\ee^2}}}} \right) \bb P_x \left( X_{\nu_n} \in \dd x' \,,\, z+M_{\nu_n} \in \dd z' \,,\, \right. \\
	&\underbrace{\hspace{4cm} \left. z+M_{\nu_n} \leq n^{1/2-\ee/2} \,,\, \tau_y > \nu_n \,,\, \nu_n \leq \pent{n^{1-\ee}} \right)}_{=:E_{44}}. \nonumber
\end{align}

\textit{Bound of $E_{43}$.} Using Lemma \ref{majmart},
\[
\bb P_{x'} \left( M_{\pent{n^{\ee^2}}} > c_{\ee} n^{1/2-\ee/4} \right) \leq \frac{c_{\ee} n^{\ee^2} \left( 1+N(x') \right)}{n^{1/2-\ee/4}}.
\]
Therefore, we have
\begin{align*}
	E_{43} \leq\;& \bb E_x \left(  \frac{c_{\ee}}{n^{3\ee/4-\ee^2}}\left( z+M_{\nu_n} \right) \mathbbm 1_{\left\{ N \left( X_{\nu_n} \right) \leq n^{1/2-\ee} \right\}} + \frac{c_{\ee}}{n^{\ee/4-\ee^2}}  N \left( X_{\nu_n} \right) \mathbbm 1_{\left\{ N \left( X_{\nu_n} \right) > n^{1/2-\ee} \right\}} \,;\,   \right. \\
	 &\hspace{6cm} \left. z+M_{\nu_n} \leq n^{1/2-\ee/2} \,,\, \tau_y > \nu_n \,,\, \nu_n \leq \pent{n^{1-\ee}} \right) \\
	 \leq\;& \frac{c_{\ee}}{n^{3\ee/4-\ee^2}} E_1 + \frac{c_{\ee}}{n^{\ee/4-\ee^2}} E_{41}'.
\end{align*}
By Lemma \ref{SurE1etE2} and \eqref{MajE41'}, we obtain for some small $\ee>0$,
\begin{equation}
	\label{E43to0}
	E_{43} \leq c_{\ee}\frac{ \max(y,0) + \left( 1+N(x) \right)^2}{n^{\ee/2}}.
\end{equation}

\textit{Bound of $E_{44}$.} Again by Lemma \ref{majmart}, $\bb E_{x'} \left( \abs{M_{\pent{n^{\ee^2}}}} \right) \leq n^{\ee^2} \left( 1+N(x') \right)$. Consequently,
\begin{align*}
	E_{44} \leq\;& \frac{c_{\ee}}{n^{\ee-\ee^2}} \bb E_x \left( z+M_{\nu_n} \,;\, N \left( X_{\nu_n} \right) \leq n^{1/2-2\ee} \,,\, \tau_y > \nu_n \,,\, \nu_n \leq \pent{n^{1-\ee}} \right) \\
	 &+ c_{\ee}n^{\ee^2} \bb E_x \left( N \left( X_{\nu_n} \right) \,;\, N \left( X_{\nu_n} \right) > n^{1/2-2\ee} \,,\, \tau_y > \nu_n \,,\, \nu_n \leq \pent{n^{1-\ee}} \right).
\end{align*}
Proceeding exactly as in the proof of the bound of $E_{41}'$ but with $l=\pent{n^{1/2-2\ee}}$, we obtain, by Lemma \ref{SurE1etE2},
\[
E_{44} \leq c_{\ee}\frac{ \max(y,0) + \left( 1+N(x) \right)^2}{n^{\ee/2}}.
\]
Putting together this bound with \eqref{E43to0} and \eqref{decdeE42}, we find that 
\[
E_{42} \leq c_{\ee}\frac{ \max(y,0) + \left( 1+N(x) \right)^2}{n^{\ee/2}}.
\]
So, using \eqref{decdeE4} and \eqref{E41to0}, we obtain the second assertion. The first one is an easy consequence of the second one.
\end{proof}

The following results are similar to that provided by Lemmas \ref{SurE1etE2} and \ref{SurE4} (see $E_2$ and $E_4$ respectively).

\begin{lemma}
\label{SurF2etF4}
There exists $\ee_0 >0$ such that, for any $\ee\in (0,\ee_0)$, $x\in \bb X$ and $y\in \bb R$,
\begin{align*}
	F_2 &:= \bb E_x \left( y+S_{\nu_n^{\ee^2}} \,;\, \tau_y > \nu_n^{\ee^2} \,,\, \nu_n^{\ee^2} \leq \pent{n^{1-\ee}} \right) \underset{n\to \infty}{\longrightarrow} V(x,y), \\
	F_4 &:= \bb E_x \left( y+S_{\nu_n^{\ee^2}} \,;\, y+S_{\nu_n^{\ee^2}} > n^{1/2-\ee/8} \,,\, \tau_y > \nu_n^{\ee^2} \,,\, \nu_n^{\ee^2} \leq \pent{n^{1-\ee}} \right) \underset{n\to +\infty}{\longrightarrow} 0.
\end{align*}
More precisely, for any $n\geq 1$, $\ee \in (0,\ee_0)$, $x\in \bb X$ and $y\in \bb R$,
\[
\abs{F_2 - V(x,y)} \leq \frac{c_{\ee}}{n^{\ee/8}} \left( 1+\max(y,0)+N(x) \right)
\]
and
\[
F_4 \leq c_{\ee}\frac{\max(y,0) +  \left( 1+ y\mathbbm 1_{\{y> n^{1/2-2\ee}\}} +N(x) \right)^2}{n^{\ee/2}}.
\]
\end{lemma}

\begin{proof}
By \eqref{decMSX},
\[
\abs{F_2 - E_2} \leq \underbrace{\bb E_x \left( \abs{r\left( X_{\nu_n^{\ee^2}} \right)} \,;\, \tau_y > \nu_n^{\ee^2} \,,\, \nu_n^{\ee^2} \leq \pent{n^{1-\ee}} \right)}_{=:F_2'}.
\]
Using the Markov property, the definition of $\nu_n$ and Lemma \ref{SurE1etE2},
\begin{align}
	F_2' &\leq c \bb E_x \left( 1 + \e^{-cn^{\ee^2}} N\left( X_{\nu_n} \right) \,;\, \tau_y > \nu_n \,,\, \nu_n \leq \pent{n^{1-\ee}} \right) \nonumber\\
	&\leq \frac{c}{n^{1/2-\ee}} E_1 + \e^{-cn^{\ee^2}} \left( 1+N(x) \right) \nonumber\\
	\label{MajF2'}
	&\leq \frac{c_{\ee}}{n^{1/2-\ee}} \left( 1+\max(y,0)+N(x) \right).
\end{align}
Therefore, by Lemma \ref{SurE1etE2}, 
\[
\abs{F_2 - V(x,y)} \leq \abs{E_2 - V(x,y)} + F_2' \leq \frac{c_{\ee}}{n^{\ee/8}} \left( 1+\max(y,0)+N(x) \right).
\]

To bound $F_4$, set $z=y+r(x)$. By \eqref{decMSX}, on the event 
$\left\{ z+M_{\nu_n^{\ee^2}} \leq n^{1/2-\ee/4} \right\}  \cap \left\{ y+S_{\nu_n^{\ee^2}} > n^{1/2-\ee/8} \right\}$ 
we have $\abs{r\left( X_{\nu_n^{\ee^2}} \right)} > c_{\ee} n^{1/2-\ee/8}.$ 
Therefore, $$y+S_{\nu_n^{\ee^2}} \leq n^{1/2-\ee/4} - r\left( X_{\nu_n^{\ee^2}} \right) \leq \left( \frac{c_{\ee}}{n^{\ee/8}}+1 \right) \abs{r\left( X_{\nu_n^{\ee^2}} \right)},$$ which implies that
\[
F_4 \leq \bb E_x \left( y+S_{\nu_n^{\ee^2}} \,;\, z+M_{\nu_n^{\ee^2}} > n^{1/2-\ee/4} \,,\, \tau_y > \nu_n^{\ee^2} \,,\, \nu_n^{\ee^2} \leq \pent{n^{1-\ee}} \right) + c_{\ee} F_2'.
\]
By \eqref{decMSX}, Lemma \ref{SurE4} and \eqref{MajF2'}, we conclude that
\[
F_4 \leq E_4 + F_2' + c_{\ee} F_2' \leq c_{\ee}\frac{\max(y,0) +  \left( 1+ y\mathbbm 1_{\{y> n^{1/2-2\ee}\}} +N(x) \right)^2}{n^{\ee/2}}.
\]
\end{proof}

\subsection{Proof of Theorem \ref{thontau}}

Assume that $(x,y) \in \bb X \times \bb R$. Let $\left( B_t \right)_{t \geq 0}$ be the Brownian motion defined by Proposition \ref{majdeA_k}. Consider the event 
\begin{equation}
	\label{defdeAk}
	A_k = \{ \sup_{0\leq t \leq 1} \abs{S_{\pent{tk}} - \sigma B_{tk}} \leq k^{1/2-2\ee} \}
\end{equation}
and denote by $\overline{A}_k$ its complement. Using these notations, we write
\begin{align}
	\bb P_x \left( \tau_y >n \right) =\;& \bb P_x \left( \tau_y >n \,,\, \nu_n^{\ee^2} > \pent{n^{1-\ee}} \right) \nonumber\\
	+&\sum_{k=\pent{n^{\ee^2}}+1}^{\pent{n^{1-\ee}}} \int_{\bb X \times \bb R} \bb P_{x'} \left( \tau_{y'} > n-k \,,\, \overline{A}_{n-k} \right) \bb P_x \left( X_k \in \dd x' \,,\, y+S_k \in \dd y' \,,\, \right. \nonumber\\
	&\underbrace{\hspace{8.5cm} \left.  \tau_y > k \,,\, \nu_n^{\ee^2} = k \right)}_{=:J_1} \nonumber\\
	\label{taudecJ0}
	+&\sum_{k=\pent{n^{\ee^2}}+1}^{\pent{n^{1-\ee}}} \int_{\bb X \times \bb R} \bb P_{x'} \left( \tau_{y'} > n-k \,,\, A_{n-k} \right) \bb P_x \left( X_k \in \dd x' \,,\, y+S_k \in \dd y' \,,\, \right. \\
	&\underbrace{\hspace{8.5cm} \left. \tau_y > k \,,\, \nu_n^{\ee^2} = k \right)}_{=:J_2} \nonumber.
\end{align}

\textit{Bound of $J_1$.} Since $n-k \geq c_{\ee} n$, for any $k\leq \pent{n^{1-\ee}}$, by Proposition \ref{majdeA_k}, we have $\bb P_{x'} \left( \tau_{y'} > n-k \,,\, \overline{A}_{n-k} \right) \leq \frac{c_{\ee} \left( 1+N(x') \right)}{n^{2\ee}}$. So, using the fact that $n^{1/2-\ee} \leq z+M_{\nu_n}$ and Lemma \ref{SurE1etE2},
\begin{align}
	J_1 &\leq \frac{c_{\ee}}{n^{2\ee}} \bb E_x \left( 1+\e^{-c n^{\ee^2}} N\left( X_{\nu_n} \right) \,;\, \tau_y > \nu_n \,,\, \nu_n \leq \pent{n^{1-\ee}} \right) \nonumber\\
	&\leq \frac{c_{\ee}}{n^{1/2+\ee}} E_1 +  \e^{-c_{\ee} n^{\ee^2}} \left( 1+N(x) \right) \nonumber\\
	&\leq \frac{c_{\ee} \left( 1+\max(y,0)+N(x) \right)}{n^{1/2+\ee}}.
	\label{taumajJ1}
\end{align}

\textit{Bound of $J_2$.} We split $J_2$ into two terms:
\begin{align}
	J_2 =\;& \sum_{k=\pent{n^{\ee^2}}+1}^{\pent{n^{1-\ee}}} \int_{\bb X \times \bb R} \bb P_{x'} \left( \tau_{y'} > n-k \,,\, A_{n-k} \right) \nonumber\\
	&\underbrace{\quad \times \bb P_x \left( X_k \in \dd x' \,,\, y+S_k \in \dd y' \,,\, y+S_k > n^{1/2-\ee/8} \,,\, \tau_y > k \,,\, \nu_n^{\ee^2} = k \right)}_{=:J_3} \nonumber\\
	\label{taudecJ2}
	&+\sum_{k=\pent{n^{\ee^2}}+1}^{\pent{n^{1-\ee}}} \int_{\bb X \times \bb R} \bb P_{x'} \left( \tau_{y'} > n-k \,,\, A_{n-k} \right) \\
	&\underbrace{\quad \times \bb P_x \left( X_k \in \dd x' \,,\, y+S_k \in \dd y' \,,\, y+S_k \leq n^{1/2-\ee/8} \,,\, \tau_y > k \,,\, \nu_n^{\ee^2} = k \right)}_{=:J_4}. \nonumber
\end{align}

\textit{Bound of $J_3$.} With $y'_+ = y'+(n-k)^{1/2-2\ee}$, we have
\begin{equation}
	\label{tautotaubm}
	\bb P_{x'} \left( \tau_{y'} > n-k \,,\, A_{n-k} \right) \leq \bb P_{x'} \left( \tau_{y_+'}^{bm} > n-k \right),
\end{equation}
where $\tau_y^{bm}$ is defined in \eqref{defdetaubm}. By the point \ref{exittimeforB001} of Proposition \ref{exittimeforB} and Lemma \ref{SurF2etF4},
\begin{align}
	J_3 &\leq \frac{c_{\ee}}{\sqrt{n}} \bb E_x \left( y+S_{\nu_n^{\ee^2}} +n^{1/2-2\ee}  \,;\, y+S_{\nu_n^{\ee^2}} > n^{1/2-\ee/8} \,,\, \tau_y > \nu_n^{\ee^2} \,,\, \nu_n^{\ee^2} \leq \pent{n^{1-\ee}} \right) \nonumber\\
	&\leq \frac{2c_{\ee}}{\sqrt{n}} F_4 \nonumber\\
	\label{taumajJ3}
	&\leq c_{\ee}\frac{\max(y,0) +  \left( 1+ y\mathbbm 1_{\{y> n^{1/2-2\ee}\}} +N(x) \right)^2}{n^{1/2+\ee/2}}.
\end{align}

\textit{Upper bound of $J_4$.} For $y' \leq n^{1/2-\ee/8}$ and any $k\leq \pent{n^{1-\ee}}$, it holds $y_+' \leq 2n^{1/2-\ee/8} \leq c_{\ee} (n-k)^{1/2-\ee/8}$. Therefore, by \eqref{tautotaubm} and the point \ref{exittimeforB002} of Proposition \ref{exittimeforB} with $\theta_m = c_{\ee} m^{-\ee/8}$ and $m=n-k$, we have
\begin{align*}
	J_4 &\leq \sum_{k=\pent{n^{\ee^2}}+1}^{\pent{n^{1-\ee}}} \int_{\bb X \times \bb R} \frac{2\left( 1+ \theta_{n-k}^2 \right)}{\sqrt{2\pi (n-k)}\sigma} \bb E_x \left( y + S_k +(n-k)^{1/2-2\ee} \,;\, \right. \\
	&\hspace{7cm} \left. y+S_k \leq n^{1/2-\ee/8} \,,\, \tau_y > k \,,\, \nu_n^{\ee^2} = k \right).
\end{align*}
Since $\frac{2\left( 1+ \theta_{n-k}^2 \right)}{\sqrt{2\pi (n-k)}\sigma} \leq \frac{2}{\sqrt{2\pi n}\sigma} \left( 1+\frac{c_{\ee}}{n^{\ee/4}} \right)$ and $n^{1/2-\ee} \leq z+M_{\nu_n}$, we get
\begin{align*}
	J_4 &\leq \frac{2}{\sqrt{2\pi n}\sigma} \left( 1+\frac{c_{\ee}}{n^{\ee/4}} \right) \bb E_x \left( y+S_{\nu_n^{\ee^2}} +n^{1/2-2\ee}  \,;\, y+S_{\nu_n^{\ee^2}} \leq n^{1/2-\ee/8} \,,\, \right.\\
	&\hspace{9cm} \left. \tau_y > \nu_n^{\ee^2} \,,\, \nu_n^{\ee^2} \leq \pent{n^{1-\ee}} \right)\\
	&\leq \frac{2}{\sqrt{2\pi n}\sigma} \left( 1+\frac{c_{\ee}}{n^{\ee/4}} \right) F_2 + \frac{c_{\ee}}{n^{1/2+\ee}} E_1.
\end{align*}
By Lemmas \ref{SurE1etE2}, \ref{SurF2etF4} and \eqref{encadrV},
\begin{equation}
	\label{taumajJ4001}
	J_4 \leq \frac{2V(x,y)}{\sqrt{2\pi n}\sigma} + \frac{c_{\ee} \left( 1+\max(y,0)+N(x) \right) }{n^{1/2+\ee/8}}.
\end{equation}

\textit{Lower bound of $J_4$.} With $y_-'=y'-(n-k)^{1/2-2\ee}$, we have $\bb P_{x'} \left( \tau_{y'} > n-k \,,\, A_{n-k} \right) \geq \bb P_{x'} \left( \tau_{y_-'}^{bm} > n-k \right) - \bb P_{x'} \left( \overline{A}_{n-k} \right)$.
Considering the event $\{ y+S_k > (n-k)^{1/2-2\ee} \}$ and repeating the arguments used to bound $J_1$ (see \eqref{taumajJ1}), we obtain
\begin{align*}
	J_4 &\geq \sum_{k=\pent{n^{\ee^2}}+1}^{\pent{n^{1-\ee}}} \int_{\bb X \times \bb R} \bb P_{x'} \left( \tau_{y_-'}^{bm} > n-k \right) \bb P_x \left( X_k \in \dd x' \,,\, y+S_k \in \dd y' \,,\, \right. \\
	&\hspace{3cm} \left. y+S_k \leq n^{1/2-\ee/8} \,,\, y+S_k > (n-k)^{1/2-2\ee} \,,\, \tau_y > k \,,\, \nu_n^{\ee^2} = k \right) \\
	&\quad-\frac{c_{\ee}\left(1+\max(y,0)+N(x)\right)}{n^{1/2+\ee}}.
\end{align*}
Using the point \ref{exittimeforB002} of Proposition \ref{exittimeforB} and Proposition \ref{majdeA_k},
\begin{align*}
	J_4 \geq\;& \frac{2}{\sqrt{2\pi n}\sigma} \left( 1-\frac{c_{\ee}}{n^{\ee/4}} \right) \bb E_x \left( y+S_{\nu_n^{\ee^2}} - (n-\nu_n^{\ee^2})^{1/2-2\ee}  \,;\, \right.\\
	&\qquad \left. y+S_{\nu_n^{\ee^2}} > (n-\nu_n^{\ee^2})^{1/2-2\ee} \,,\, y+S_{\nu_n^{\ee^2}} \leq n^{1/2-\ee/8} \,,\, \tau_y > \nu_n^{\ee^2} \,,\, \nu_n^{\ee^2} \leq \pent{n^{1-\ee}} \right)\\
	&-\frac{c_{\ee}\left(1+\max(y,0)+N(x)\right)}{n^{1/2+\ee}} \\
	\geq\;& \frac{2}{\sqrt{2\pi n}\sigma} \left( 1-\frac{c_{\ee}}{n^{\ee/4}} \right) F_2 - \frac{c_{\ee}}{\sqrt{n}} F_4 - \frac{c_{\ee}}{n^{1/2+\ee}} E_1 -\frac{c_{\ee}\left(1+\max(y,0)+N(x)\right)}{n^{1/2+\ee}}.
\end{align*}
By Lemmas \ref{SurE1etE2}, \ref{SurF2etF4} and \eqref{encadrV},
\begin{equation}
	\label{taumajJ4002}
	J_4 \geq \frac{2V(x,y)}{\sqrt{2\pi n}\sigma} -c_{\ee}\frac{\max(y,0) +  \left( 1+ y\mathbbm 1_{\{y> n^{1/2-2\ee}\}} +N(x) \right)^2}{n^{1/2+\ee/8}}.
\end{equation}

Putting together \eqref{taumajJ4002}, \eqref{taumajJ4001}, \eqref{taumajJ3} and \eqref{taudecJ2},
\[
\abs{J_2 - \frac{2V(x,y)}{\sqrt{2\pi n}\sigma}} \leq  c_{\ee}\frac{\max(y,0) +  \left( 1+ y\mathbbm 1_{\{y> n^{1/2-2\ee}\}} +N(x) \right)^2}{n^{1/2+\ee/8}}.
\]
Taking into account \eqref{taumajJ1}, \eqref{taudecJ0} and Lemma \ref{concentnu}, we conclude that, for any  $(x,y) \in \bb X \times \bb R$,
\begin{equation}
\abs{\bb P_x \left( \tau_y > n \right) - \frac{2V(x,y)}{\sqrt{2\pi n}\sigma}} \leq  
c_{\ee}\frac{\max(y,0) +  \left( 1+ y\mathbbm 1_{\{y> n^{1/2-2\ee}\}} +N(x) \right)^2}{n^{1/2+\ee/8}}.
\label{taumajJ4003}
\end{equation}

Taking the limit as $n\to +\infty$ in \eqref{taumajJ4003}, we obtain the point \ref{thontau001} of Theorem \ref{thontau}. 
The point \ref{thontau002} of Theorem \ref{thontau} is an immediate consequence of the points \ref{posdeVsurDgamma002} and \ref{posdeVsurDgamma004} of Proposition \ref{posdeVsurDgamma}.

\subsection{Proof of Theorem \ref{thontau2}}

The point \ref{thontau001bis} of Theorem \ref{thontau2} is exactly \eqref{taumajJ4003}. In order to prove the point \ref{thontau003} of Theorem \ref{thontau2}, we will first establish a bound for $\bb P_x \left( \tau_y > n \right)$ when $z=y+r(x)\geq n^{1/2-\ee}$. Set $m_{\ee} = n -\pent{n^{\ee}}$. By the Markov property,
\begin{align}
	\bb P_x \left( \tau_y > n \right) = \int_{\bb X \times \bb R} &\bb P_{x'} \left( \tau_{y'} > m_{\ee} \right) \nonumber\\
	\label{majtauygrand}
	&\qquad \times \bb P_x \left( X_{\pent{n^{\ee}}} \in \dd x' \,,\, y+S_{\pent{n^{\ee}}} \in \dd y' \,,\, \tau_y > \pent{n^{\ee}} \right).
\end{align}
For any $x' \in \bb X$ and $y' >0$, using $A_{m_{\ee}}$ defined by \eqref{defdeAk}, we have
\[
\bb P_{x'} \left( \tau_{y'} > m_{\ee} \right) \leq \bb P_{x'} \left( \tau_{y_+'}^{bm} > m_{\ee} \right) + \bb P_{x'} \left( \overline{A}_{m_{\ee}} \right),
\]
where $\tau_{y_+'}^{bm}$ is defined by \eqref{defdetaubm} and $y_+' = y'+ m_{\ee}^{1/2-2\ee}$. By the point \ref{exittimeforB001} of Proposition \ref{exittimeforB} and Proposition \ref{majdeA_k},
\[
\bb P_{x'} \left( \tau_{y'} > m_{\ee} \right) \leq \frac{c y_+'}{\sqrt{m_{\ee}}} + \frac{c_{\ee}}{m_{\ee}^{2\ee}} \left( 1+N(x') \right) \leq  \frac{c_{\ee} y'}{\sqrt{n}} + \frac{c_{\ee}}{n^{2\ee}} + \frac{c_{\ee}}{n^{2\ee}} N(x').
\]
Introducing this bound in \eqref{majtauygrand}, we get
\[
\bb P_x \left( \tau_y > n \right) \leq \frac{c_{\ee}}{\sqrt{n}} \bb E_x \left( y+S_{\pent{n^{\ee}}} \,,\, \tau_y > \pent{n^{\ee}} \right) + \frac{c_{\ee}}{n^{2\ee}} + \frac{c_{\ee}}{n^{2\ee}} \bb E_x \left( N \left( X_{\pent{n^{\ee}}} \right) \right).
\]
Using Corollary \ref{unifmajdelafoncinv}, the inequality \eqref{decexpN} and the fact that $n^{1/2-\ee} \leq z$, we find
\begin{equation}
	\label{point2th22}
	\bb P_x \left( \tau_y > n \right) \leq \frac{c_{\ee} \left( z+N(x) \right)}{\sqrt{n}}.
\end{equation}

Now, for any $x\in \bb X$, $z \in \bb R$ and $y=z-r(x)$, using the Markov property, \eqref{point2th22} and the fact that $\sqrt{n-\nu_n} \geq c_{\ee} \sqrt{n}$ on the event $\{ \nu_n \leq \pent{n^{1-\ee}} \}$, we have
\begin{align*}
	\bb P_x \left( \tau_y > n \right) \leq\;& \frac{c_{\ee}}{\sqrt{n}} \bb E_x \left( z+M_{\nu_n} +N \left( X_{\nu_n} \right) \,;\, \tau_y >\nu_n \,,\, \nu_n \leq \pent{n^{1-\ee}} \right) \\
	&+ \bb P_x \left( \tau_y > n \,,\, \nu_n > \pent{n^{1-\ee}} \right).
\end{align*}
Using Lemma \ref{concentnu} and the fact that $N \left( X_{\nu_n} \right) \leq z+M_{\nu_n}$ on the event $\{ N \left( X_{\nu_n} \right) \leq n^{1/2-\ee} \}$, with $l=\pent{n^{1/2-\ee}}$, it holds
\begin{align*}
	\bb P_x \left( \tau_y > n \right) \leq\;& \frac{c_{\ee}}{\sqrt{n}} \bb E_x \left( \left( z+M_{\nu_n} \right) \left( 1 + \mathbbm 1_{\left\{ N \left( X_{\nu_n} \right) \leq n^{1/2-\ee} \right\}} \right) \,;\, \tau_y >\nu_n \,,\, \nu_n \leq \pent{n^{1-\ee}} \right) \\
	&+ \frac{c_{\ee}}{\sqrt{n}} \bb E_x \left( N_l \left( X_{\nu_n} \right) \,;\, \tau_y >\nu_n \,,\, \nu_n \leq \pent{n^{1-\ee}} \right) + \e^{-c_{\ee} n^{\ee}} \left( 1+N(x) \right) \\
	\leq\;& \frac{2c_{\ee}}{\sqrt{n}} E_1 + \frac{c_{\ee}}{\sqrt{n}} \sum_{k=1}^{\pent{n^{\ee}}} \bb E_x \left( N_l \left( X_k \right) \right) \\
	&+ \frac{c_{\ee}}{\sqrt{n}} \sum_{k=\pent{n^{\ee}}+1}^{\pent{n^{1-\ee}}} \bb E_x \left( N_l \left( X_k \right) \,;\, \tau_y > k \right) + \e^{-c_{\ee} n^{\ee}} \left( 1+N(x) \right).
\end{align*}
By \eqref{decexpNl} and the Markov property,
\begin{align*}
	\bb P_x \left( \tau_y > n \right) \leq\;& \frac{c_{\ee}}{\sqrt{n}} E_1 + \frac{c_{\ee}}{\sqrt{n}} \left(  \frac{c n^{\ee}}{l^{1+\beta}} + \left( 1+N(x) \right) \right) + \e^{-c_{\ee} n^{\ee}} \left( 1+N(x) \right) \\
	&+ \frac{c_{\ee}}{\sqrt{n}} \sum_{j=1}^{\pent{n^{1-\ee}}-\pent{n^{\ee}}} \left[ \frac{c}{l^{1+\beta}} \bb P_x \left( \tau_y > j \right) + \e^{-cn^{\ee}} \bb E_x \left( \left( 1+N\left( X_j \right) \right) \right) \right] \\
	\leq\;& \frac{c_{\ee}}{\sqrt{n}} E_1 + \frac{c_{\ee} \left( 1+N(x) \right)}{\sqrt{n}} + \frac{c_{\ee}}{\sqrt{n}} \frac{c}{l^{1+\beta}} \sum_{j=1}^{\pent{n^{1-\ee}}} \bb P_x \left( \tau_y > j \right).
\end{align*}
Using Lemmas \ref{SurE1etE2} and \ref{taupptrn}, we deduce the point \ref{thontau003} of Theorem \ref{thontau2}.

\section{Asymptotic for the conditioned Markov walk}
\label{AsCondMarkWalk}

In this section, we prove Theorem \ref{loideRayleigh}. The arguments are similar to those given in Section \ref{AsExTi}. We also keep the same notations. 
Assume that $(x,y) \in \bb X \times \bb R$ and let $t_0> 0$ be a positive real. For any $t\in \left[ 0 , t_0 \right]$, we write
\begin{align}
	\bb P_x &\left( y+S_n \leq t\sqrt{n} \,,\, \tau_y > n \right) \nonumber\\
	=\;& \bb P_x \left( y+S_n \leq t\sqrt{n} \,,\, \tau_y > n \,,\, \nu_n^{\ee^2} > \pent{n^{1-\ee}} \right) \nonumber\\
	&+ \sum_{k=\pent{n^{\ee^2}}+1}^{\pent{n^{1-\ee}}} \int_{\bb X \times \bb R} \bb P_{x'} \left( y'+S_{n-k} \leq t\sqrt{n} \,,\, \tau_{y'} > n-k \,,\, \overline{A}_{n-k} \right) \nonumber\\
	&\underbrace{\hspace{3cm} \times \bb P_x \left( X_k \in \dd x' \,,\, y+S_k \in \dd y' \,,\, \tau_y > k \,,\, \nu_n^{\ee^2} = k \right)}_{=:L_1} \nonumber\\
	\label{decL0}
	&+ \sum_{k=\pent{n^{\ee^2}}+1}^{\pent{n^{1-\ee}}} \int_{\bb X \times \bb R} \bb P_{x'} \left( y'+S_{n-k} \leq t\sqrt{n} \,,\, \tau_{y'} > n-k \,,\, A_{n-k} \right).\\
&\underbrace{\hspace{3cm} \times \bb P_x \left( X_k \in \dd x' \,,\, y+S_k \in \dd y' \,,\, \tau_y > k \,,\, \nu_n^{\ee^2} = k \right)}_{=:L_2}. \nonumber
\end{align}

\textit{Bound of $L_1$.} With $J_1$ defined in \eqref{taudecJ0} and with the bound \eqref{taumajJ1}, we have,
\begin{equation}
	\label{MajL1}
	L_1 \leq J_1 \leq \frac{c_{\ee} \left( 1+\max(y,0)+N(x) \right)}{n^{1/2+\ee}}.
\end{equation}

\textit{Bound of $L_2$.} According to whether $y+S_k \leq n^{1/2-\ee/8}$ or not, we write
\begin{align}
	L_2 =\;& \sum_{k=\pent{n^{\ee^2}} +1}^{\pent{n^{1-\ee}}} \int_{\bb X \times \bb R} \bb P_{x'} \left( y'+S_{n-k} \leq t \sqrt{n} \,,\, \tau_{y'} > n-k \,,\, A_{n-k} \right) \nonumber\\
	&\underbrace{\quad \times \bb P_x \left( X_k \in \dd x' \,,\, y+S_k \in \dd y' \,,\, y+S_k > n^{1/2-\ee/8} \,,\, \tau_y > k \,,\, \nu_n^{\ee^2} = k \right)}_{=:L_3} \nonumber\\
	\label{decL2}
	&+\sum_{k=\pent{n^{\ee^2}}+1}^{\pent{n^{1-\ee}}} \int_{\bb X \times \bb R} \bb P_{x'} \left( y'+S_{n-k} \leq t \sqrt{n} \,,\, \tau_{y'} > n-k \,,\, A_{n-k} \right) \\
	&\underbrace{\quad \times \bb P_x \left( X_k \in \dd x' \,,\, y+S_k \in \dd y' \,,\, y+S_k \leq n^{1/2-\ee/8} \,,\, \tau_y > k \,,\, \nu_n^{\ee^2} = k \right)}_{=:L_4}. \nonumber
\end{align}

\textit{Bound of $L_3$.} With $J_3$ defined in \eqref{taudecJ2} and with the bound \eqref{taumajJ3}, we have
\begin{equation}
	\label{MajL3}
	L_3 \leq J_3 \leq  c_{\ee}\frac{\max(y,0) +  \left( 1+ y\mathbbm 1_{\{y> n^{1/2-2\ee}\}} +N(x) \right)^2}{n^{1/2+\ee/2}}.
\end{equation}

\textit{Bound of $L_4$.} We start with the upper bound. Set $y_+' = y' + (n-k)^{1/2-2\ee}$ and $t_+ = t+\frac{2}{n^{2\ee}}$. Note that on the event $\{ y'+S_{n-k} \leq t\sqrt{n} \,,\, \tau_{y'} > n-k \,,\, A_{n-k} \}$ we have $y_+' + \sigma B_{n-k} \leq t_+\sqrt{n}$ and $\tau_{y_+'}^{bm} > n-k$.
Therefore, by Proposition \ref{intBro},
\begin{align*}
	&\bb P_{x'} \left( y'+S_{n-k} \leq t\sqrt{n} \,,\, \tau_{y'} > n-k \,,\, A_{n-k} \right) \\
	&\hspace{4cm} \leq \frac{2}{\sqrt{2\pi}} \int_0^{\frac{t_+ \sqrt{n}}{\sigma \sqrt{n-k}}} \e^{-s^2/2} \sh \left( s \frac{y'_+}{\sqrt{n-k}\sigma} \right) \dd s.
\end{align*}
We shall use the following bounds:
\begin{align*}
	&\sh(u) \leq u \left( 1+\frac{u^2}{6} \ch(u) \right), \qquad &&\text{for } u\geq 0, \\
	&\frac{y_+'}{\sigma\sqrt{n-k}} \leq \frac{y_+'}{\sigma\sqrt{n}} \left( 1+ \frac{c_{\ee}}{n^{\ee}} \right) \leq \frac{c_{\ee}}{n^{\ee/8}}, \qquad &&\text{for } y' \leq n^{1/2-\ee/8} \text{ and } k\leq \pent{n^{1-\ee}},\\
	&\frac{t_+ \sqrt{n}}{\sigma \sqrt{n-k}} \leq \frac{t}{\sigma} + \frac{c_{\ee,t_0}}{n^{\ee}} \leq c_{\ee,t_0}, \qquad &&\text{for } k\leq \pent{n^{1-\ee}}.
\end{align*}
Consequently,
\begin{align*}
	&\bb P_{x'} \left( y'+S_{n-k} \leq t\sqrt{n} \,,\, \tau_{y'} > n-k \,,\, A_{n-k} \right) \\
	&\hspace{2cm} \leq \frac{2y_+'}{\sqrt{2\pi n}\sigma} \left( 1+ \frac{c_{\ee}}{n^{\ee}} \right) \int_0^{\frac{t_+ \sqrt{n}}{\sigma \sqrt{n-k}}} s \e^{-s^2/2} \left( 1 + \frac{c_{\ee} s^2}{n^{\ee/4}}\ch \left( c_{\ee} s \right) \right) \dd s \\
	&\hspace{2cm} \leq \frac{2y_+'}{\sqrt{2\pi n}\sigma} \left( 1+ \frac{c_{\ee}}{n^{\ee}} \right) \left( 1 + \frac{c_{\ee,t_0}}{n^{\ee/4}} \right) \left( \int_0^{\frac{t}{\sigma}} s \e^{-s^2/2} \dd s + \int_{\frac{t}{\sigma}}^{\frac{t_+ \sqrt{n}}{\sigma \sqrt{n-k}}} s \e^{-s^2/2} \dd s \right) \\
	&\hspace{2cm} \leq \frac{2y_+'}{\sqrt{2\pi n}\sigma} \left( 1+\frac{c_{\ee,t_0}}{n^{\ee/4}} \right) \left( 1-\e^{-\frac{t^2}{2\sigma^2}} + \frac{c_{\ee,t_0}}{n^{\ee}} \right).
\end{align*}
This implies the upper bound (with $F_2$ and $E_1$ from Lemmas \ref{SurF2etF4}  and \ref{SurE1etE2}, respectively)
\begin{align*}
	L_4 &\leq \frac{2}{\sqrt{2\pi n}\sigma} \left( 1+\frac{c_{\ee,t_0}}{n^{\ee/4}} \right) \left( 1-\e^{-\frac{t^2}{2\sigma^2}} + \frac{c_{\ee,t_0}}{n^{\ee}} \right) F_2 + \frac{c_{\ee,t_0}}{n^{1/2+\ee}} E_1 \\
	&\leq \frac{2V(x,y)}{\sqrt{2\pi n}\sigma} \left( 1-\e^{-\frac{t^2}{2\sigma^2}} \right) + \frac{c_{\ee,t_0} \left( 1+\max(y,0)+N(x) \right)}{n^{1/2+\ee/8}}.
\end{align*}
The proof of the following lower bound of $L_4$, being similar, is left to the reader: 
\[
L_4 \geq \frac{2V(x,y)}{\sqrt{2\pi n}\sigma} \left( 1-\e^{-\frac{t^2}{2\sigma^2}} \right) - c_{\ee,t_0}\frac{\max(y,0) +  \left( 1+ y\mathbbm 1_{\{y> n^{1/2-2\ee}\}} +N(x) \right)^2}{n^{1/2+\ee/8}}.
\]
Combining the upper and the lower bounds of $L_4$ and \eqref{MajL3} with \eqref{decL2} 
we obtain an asymptotic developpement of $L_2.$ 
Implementing this developpement and the bound \eqref{MajL1} into \eqref{decL0} and using Lemma \ref{concentnu}, we conclude that
\begin{align*}
&\abs{\bb P_x \left( y+S_n \leq t \sqrt{n} \,,\, \tau_y > n \right) - \frac{2V(x,y)}{\sqrt{2\pi n}\sigma} \left( 1-\e^{-\frac{t^2}{2\sigma^2}} \right)} \\
&\hspace{6cm} \leq c_{\ee,t_0}\frac{\max(y,0) +  \left( 1+ y\mathbbm 1_{\{y> n^{1/2-2\ee}\}} +N(x) \right)^2}{n^{1/2+\ee/8}}.
\end{align*}
Using the asymptotic of $\bb P_x (\tau_y>n)$ provided by Theorem \ref{thontau} finishes the proof of Theorem \ref{loideRayleigh}.

\section{Appendix}
\label{Appendix}

\subsection{Convergence of recursively bounded monotonic sequences}

We recall two lemmas from \cite{GLLP_affine_2016} which give sufficient conditions for a monotonic sequence to be bounded.

\begin{lemma}
\label{lemanalyse}
Let $(u_n)_{n\geq 1}$ be a non-decreasing sequence of reals such that there exist $\ee \in (0,1)$ and $\alpha, \beta, \gamma, \delta \geq 0$ such that, for any $n\geq 2$,
\[
u_n \leq \left( 1+ \frac{\alpha}{n^\ee} \right) u_{\pent{n^{1-\ee}}} + \frac{\beta}{n^\ee} + \gamma \e^{-\delta n^\ee}.
\]
Then, for any $n\geq 2$ and any integer $n_f \in \{2, \dots, n \}$,
\[
u_n \leq \left( 1+ \frac{c_{\alpha,\ee}}{n_f^\ee} \right) u_{n_f} + \beta \frac{c_{\alpha,\ee}}{n_f^\ee} + \gamma \e^{-c_{\alpha,\delta,\ee} n_f^\ee}.
\]
In particular, choosing $n_f$ constant, it follows that $(u_n)_{n\geq 1}$ is bounded. 
\end{lemma}

\begin{lemma}
\label{lemanalyse2}
Let $(u_n)_{n\geq 1}$ be a non-increasing sequence of reals such that there exist $\ee \in (0,1)$ and $\beta \geq 0$ such that, for any $n\geq 2$,
\[
u_n \geq u_{\pent{n^{1-\ee}}} - \frac{\beta}{n^\ee}.
\]
Then, for any $n\geq 2$ and any integer $n_f \in \{2, \dots, n \}$,
\[
u_n \geq u_{n_f} - c_\ee \frac{\beta}{n_f^\ee}.
\]
In particular, choosing $n_f$ constant, it follows that  $(u_n)_{n\geq 1}$ is bounded.
\end{lemma}

\subsection{Brownian motion and strong approximation}
\label{Strong Approx}

We consider the standard Brownian motion $\left( B_t \right)_{t\geq 0}$ with values in $\bb R$ living on the probability space $\left( \Omega, \mathscr F, \bb P \right)$. Define the exit time
\begin{equation}
	\label{defdetaubm}
	\tau_y^{bm} = \inf \{ t\geq 0, \, y+\sigma B_t \leq 0 \},
\end{equation}
where $\sigma>0$. 

The following affirmations are due to L\'evy \cite{levy_theorie_1937}.

\begin{proposition}
\label{intBro}
For any $y>0$, $0\leq a \leq b$ and $n \geq 1$,
	\[
	\bb P \left( \tau_y^{bm} > n \,,\, y+\sigma B_n \in [a,b] \right) = \frac{1}{\sqrt{2\pi n} \sigma} \int_a^b \left( \e^{-\frac{(s-y)^2}{2n\sigma^2}} - \e^{-\frac{(s+y)^2}{2n\sigma^2}} \right) \dd s.
\]
\end{proposition}

\begin{proposition}\ 
\label{exittimeforB}
\begin{enumerate}[ref=\arabic*, leftmargin=*, label=\arabic*.]
	\item
	\label{exittimeforB001}
	For any $y>0$,
	\[
	\bb P \left( \tau_y^{bm}>n \right) \leq c\frac{y}{\sqrt{n}}.
	\]
	\item
	\label{exittimeforB002}
	For any sequence of real numbers $(\theta_n)_{n\geq 0}$ such that $\theta_n \underset{n\to +\infty}{\longrightarrow} 0$,
	\[
	\underset{y\in [0; \theta_n \sqrt{n}]}{\sup} \left( \frac{\bb P \left( \tau_y^{bm}>n \right)}{\frac{2y}{\sqrt{2\pi n}\sigma}} - 1 \right) = O(\theta_n^2).
	\]
\end{enumerate}
\end{proposition}

Moreover, under hypotheses \ref{BASP}-\ref{CECO} it is proved in \cite{ion_grama_rate_2014} that there is a version of the Markov walk $(S_n)_{n\geq 0}$ and of the standard Brownian motion $(B_t)_{t\geq 0}$ living on the same 
probability space which are close enough in the following sense: 

\begin{proposition}
\label{majdeA_k}
Assume that the Markov chain $\left(X_n\right)_{n\geq 0}$ and the function $f$ satisfy Hypotheses \ref{BASP}-\ref{CECO}. There exists $\ee_0 >0$ such that, for any $\ee \in (0,\ee_0]$, $x\in \bb X$ and $n\geq 1$, without loss of generality (on an extension of the initial probability space) one can reconstruct the sequence $(S_n)_{n\geq 0}$ with a continuous time Brownian motion $(B_t)_{t\in \bb R_{+} }$, such that  
\begin{equation}
\bb P_x \left( \underset{0 \leq t \leq 1}{\sup} \abs{S_{\pent{tn}}-\sigma B_{tn}} > n^{1/2-\ee} \right) \leq \frac{c_{\ee}}{n^{\ee}} ( 1+N(x) ),
\label{KMT001}
\end{equation}
where $\sigma$ is defined in the point \ref{MomAs002} of Proposition \ref{MomAs}.
\end{proposition}

In the original result the right-hand side in \eqref{KMT001}  is $c_\ee n^{-\ee} (1+N(x))^{\alpha}$ with $\alpha>2.$ To obtain the result of Proposition \ref{majdeA_k} it suffices to take the power $1/\alpha$ on the both sides and to use the obvious 
inequality $p<p^{1/\alpha},$ for $p \in [0,1]$.

Using this proposition, we deduce the following result.

\begin{corollary}
\label{BerEss}
There exists $\ee_0 >0$ such that, for any $\ee \in (0,\ee_0)$, $x\in \bb R$ and $n\geq 1$,
\[
\underset{t\in \bb R}{\sup} \, \abs{ \bb P_x \left( \frac{S_n}{\sqrt{n}} \leq t \right) - \int_{-\infty}^{t} \e^{-\frac{u^2}{2\sigma^2}} \frac{\dd u}{\sqrt{2\pi} \sigma} } \leq \frac{c_{\ee}}{n^\ee} \left(1+N(x)\right).
\]
\end{corollary}

\begin{proof} Using Proposition \ref{majdeA_k},
\begin{align*}
	\bb P_x \left( \frac{S_n}{\sqrt{n}} \leq t \right)
	\leq \; &\bb P_x \left( \abs{S_n-\sigma B_n} > n^{1/2-\ee} \right) + \bb P_x \left( \frac{\sigma B_n}{\sqrt{n}} \leq t + \frac{1}{n^\ee} \right) \\
	\leq \; & \frac{c_{\ee}}{n^\ee} \left(1+N(x)\right) + \int_{-\infty}^{t+\frac{1}{n^\ee}} \e^{-\frac{u^2}{2\sigma^2}} \frac{\dd u}{\sqrt{2\pi} \sigma}.
\end{align*}
Therefore,
\[
\bb P_x \left( \frac{S_n}{\sqrt{n}} \leq t \right) - \int_{-\infty}^{t} \e^{-\frac{u^2}{2\sigma^2}} \frac{\dd u}{\sqrt{2\pi} \sigma} \leq \frac{c_{\ee}}{n^\ee} \left(1+N(x)\right).
\]
In the same way way,
\[
\bb P_x \left( \frac{S_n}{\sqrt{n}} \leq t \right) \geq \int_{-\infty}^{t-\frac{1}{n^\ee}} \e^{-\frac{u^2}{2\sigma^2}} \frac{\dd u}{\sqrt{2\pi} \sigma} - \bb P_x \left( \frac{\abs{S_n-\sigma B_n}}{\sqrt{n}} > \frac{1}{n^\ee} \right)
\]
and the result follows.
\end{proof}

\subsection{Finiteness of the exit times \texorpdfstring{$\tau_y$}{}  and \texorpdfstring{$T_z$}{}}
\label{secExitfinit}

\begin{lemma}
\label{Exitfinit}
For any $x\in \bb X$ and $y \in \bb R$,
	\[
	\tau_y < +\infty \quad \bb P_x\text{-a.s.}
\]
\end{lemma}

\begin{proof} Let $x \in \bb X.$ Assume first that $y > 0$. Since $\{ \tau_y > n \}$ is a non-increasing sequence of events,
	\[
	\bb P_x \left( \tau_y = + \infty \right) = \underset{n\to +\infty}{\lim} \bb P_x \left( \tau_y > n \right) 
	= \underset{n\to +\infty}{\lim} \bb P_x \left(  y+S_k > 0, \, \forall k \leq n \right).
\]
Using Proposition \ref{majdeA_k},
\[
\bb P_x \left(  y+S_k > 0, \, \forall k \leq n \right) \leq \frac{c_{\ee}}{n^\ee} \left(1+N(x)\right) + \bb P \left( \tau_{y+n^{1/2-\ee}}^{bm} > n \right).
\]
Thus, by the point \ref{exittimeforB001} of Proposition \ref{exittimeforB},
\begin{equation}
\bb P_x \left( \tau_y > n \right) \leq \frac{c_{\ee}}{n^\ee} \left(1+N(x)\right) + c\frac{y+n^{1/2-\ee}}{\sqrt{n}} \leq \frac{c_{\ee}}{n^\ee} \left(1+y+N(x)\right).
\label{Pourypos}
\end{equation}
When $y \leq 0$, we have, for any $y' > 0$, $\bb P_x \left( \tau_y > n \right) \leq \bb P_x \left( \tau_{y'} > n \right)$. Taking the limit when $y'\to 0,$ we obtain that
\begin{equation}
\bb P_x \left( \tau_y > n \right) \leq \frac{c_{\ee}}{n^\ee} \left(1+N(x)\right).
\label{Pouryneg}
\end{equation}
From \eqref{Pourypos} and \eqref{Pouryneg} it follows that, for any $y \in \bb R$,
\begin{equation}
\label{tauyto0}
	\bb P_x \left( \tau_y > n \right) \leq \frac{c_{\ee}}{n^{\ee}} \left(1+\max(y,0)+N(x)\right).
\end{equation}
Taking the limit as $n\to +\infty$, we conclude that $\tau_y < +\infty$ $\bb P_x$-a.s.
\end{proof}

The same result can be obtained for the exit time  $T_z$ of the martingale $(z+M_n)_{n\geq 0}$.
\begin{lemma}
\label{ExitfinitTz}
For any $x \in \bb X$ and $z \in \bb R$,
	\[
	T_z < + \infty \quad \bb P_x\text{-a.s.}
\]
\end{lemma}

\begin{proof} Let $x \in \bb X,$ $z\in \bb R$ and $y=z-r(x)$.
Assume first that $y=z-r(x)>0.$ 
Following the proof of Lemma \ref{Exitfinit}, 
	\[
	\bb P_x \left( T_z = + \infty \right) = \underset{n\to +\infty}{\lim} \bb P_x \left(  z+M_k > 0 , \, \forall k \leq n \right).
\]
By \eqref{decMSX} the martingale $(z+M_n)_{n\geq 0}$ is relied to the Markov walk $(y+S_n)_{n\geq 0}$, which gives
\begin{align}
	\bb P_x \left( z+M_k > 0,\, \forall k \leq n \right) \leq\;& \bb P_x \left( y+S_k > -n^{1/2-\ee},\,    \forall k \leq n \right) \nonumber\\
	\label{decpourTzfinit}
	&\hspace{2cm} + \bb P_x \left( \underset{1 \leq k \leq n}{\max} \abs{r\left( X_k \right)} > n^{1/2-\ee} \right).
\end{align}
On the one hand, in the same way as in the proof of Lemma \ref{Exitfinit},
\begin{equation}
	\label{decpourTzfinit001}
	\bb P_x \left(  y+S_k > -n^{1/2-\ee},\, \forall k \leq n \right) \leq \frac{c_{\ee}}{n^\ee} \left(1+N(x)\right) + \bb P_x \left( \tau_{y+2n^{1/2-\ee}}^{bm} > n \right).
\end{equation}
On the other hand, using Lemma \ref{MTR}, for $n$ large enough,
\[
	\bb P_x \left( \underset{1 \leq k \leq n}{\max} \abs{r\left( X_k \right)} > n^{1/2-\ee} \right) \leq \sum_{k=1}^{\pent{n^{\ee}}} \bb E_x \left( \frac{c N\left( X_k \right)}{n^{1/2-\ee}} \right) + \sum_{k=\pent{n^{\ee}}+1}^{n} \bb E_x \left( \frac{c N_l \left( X_k \right)}{n^{1/2-\ee}} \right),
\]
where $l = c n^{1/2-\ee}$. 
So, using \eqref{decexpNl} and taking $\ee \leq \min \left( \frac{1}{6}, \frac{\beta}{2(3+\beta)} \right)$, we obtain
\begin{equation}
	\label{decpourTzfinit002}
	\bb P_x \left( \underset{1 \leq k \leq n}{\max} \abs{r\left( X_k \right)} > n^{1/2-\ee} \right) \leq \frac{c_{\ee}}{n^{\ee}} \left( 1+N(x) \right).
\end{equation}
Putting together \eqref{decpourTzfinit}, \eqref{decpourTzfinit001} and  \eqref{decpourTzfinit002} and using the point \ref{exittimeforB001} of Proposition \ref{exittimeforB}, we have, for $z>r(x),$ 
\[
\bb P_x \left( T_z > n \right) \leq \frac{c_{\ee}}{n^\ee} \left(1+N(x)\right) + c\frac{y+2n^{1/2-\ee}}{\sqrt{n}} \leq \frac{c_{\ee}}{n^\ee} \left(1+\max(z,0)+N(x)\right).
\]
Since $z \mapsto T_z$ is non-decreasing, we obtain the same bound for any $z \in \bb R$,
\begin{equation}
	\label{Tzto0}
	\bb P_x \left( T_z > n \right) \leq \frac{c_{\ee}}{n^{\ee}} \left(1+\max(z,0)+N(x)\right).
\end{equation}
Taking the limit as $n\to +\infty$ we conclude that $T_z < +\infty$ $\bb P_x$-a.s.
\end{proof}

\subsection{Proof of Proposition \ref{PP001}}
\label{proof-rec-sto-Mat}

In this section, for the affine random walk in $\bb R^d$ conditioned to stay in a half-space, we verify that 
Hypotheses \ref{BASP}-\ref{CECO} hold true on an appropriate Banach space 
which we proceed to introduce. 
Let $\delta>0$ be the constant from Hypothesis \ref{hypoH}.
Denote by $\mathscr{C}( \bb R^d )$ the space of continuous complex valued functions on $\bb R^d$. Let $\ee$ and $\theta$ be two positive numbers satisfying
\[
1+\ee < \theta < 2 < 2+2\ee < 2+2\delta.
\]
For any function $h\in \mathscr{C}( \bb R^d )$ introduce the norm $\norm{h}_{\theta,\ee} = \abs{h}_{\theta} + \left[h\right]_{\ee}$, where
\[
\abs{h}_{\theta} = \underset{x \in \bb R^d}{\sup} \frac{\abs{h(x)}}{\left( 1+\abs{x} \right)^{\theta}}, \quad \left[h\right]_{\ee} = \underset{x\neq y}{\sup} \frac{\abs{h(x)-h(y)}}{\abs{x-y}^{\ee}\left( 1+\abs{x} \right)\left( 1+\abs{y} \right)}
\]
and consider the Banach space
\[
\mathscr{B} = \mathscr{L}_{\theta,\ee} = \left\{ h \in \mathscr{C} \left( \bb R^d \right),\; \norm{h}_{\theta,\ee} < +\infty \right\}.
\]

\textit{Proof of \ref{BASP}.} Conditions \ref{BASP001}, \ref{BASP002} and \ref{BASP003} of \ref{BASP} 
can be easily verified under the point \ref{H1} of Hypothesis \ref{hypoH} and the fact that $\theta < 2+2\delta$ and $\norm{\bs \delta_x}_{\mathscr{B}'} \leq \left( 1+\abs{x} \right)^{\theta}$, for any $x \in \bb R^d$. 

We verify the point \ref{BASP004} of Hypothesis \ref{BASP}. For any  $(x,y) \in \bb R^d \times \bb R^d$ and $t \in \bb R$, we have $\abs{\e^{itf(x)} - \e^{itf(y)}} \leq \abs{t} \abs{f(x) - f(y)} \leq \abs{t} \abs{u} \abs{x-y}$ and $\abs{\e^{itf(x)} - \e^{itf(y)}} \leq 2$. Therefore, we write
\[
\abs{\e^{itf(x)} - \e^{itf(y)}} \leq 2^{1-\ee} \abs{t}^{\ee} \abs{u}^{\ee} \abs{x-y}^{\ee}.
\]
Supposing that $\abs{x} \leq \abs{y}$, we obtain, for any $h \in \mathscr{L}_{\theta,\ee}$,
\[
\abs{\e^{itf(x)}h(x) - \e^{itf(y)}h(y)} \leq \abs{\e^{itf(x)} - \e^{itf(y)}} \abs{h}_{\theta} \left( 1+\abs{x} \right)^{\theta} +\abs{h(x) - h(y)}.
\]
Since $\theta < 2$, we have $\left[ \e^{itf}h - \e^{itf}h \right]_{\ee} \leq 2^{1-\ee} \abs{t}^{\ee} \abs{u}^{\ee} \abs{h}_{\theta} + \left[h\right]_{\ee}$. 
Consequently, $\norm{\e^{itf}h}_{\theta,\ee} \leq \left( 1+2^{1-\ee}\abs{t}^{\ee}\abs{u}^{\ee} \right) \norm{h}_{\theta,\ee}$ and the point \ref{BASP004} 
is verified.

\textit{Proof of \ref{SPGA} and \ref{PETO}.} We shall verify that the conditions of the theorem of Ionescu-Tulcea and Marinescu are satisfied (see \cite{norman1972markov} and \cite{tulcea_theorie_1950}). We start by establishing two lemmas.

\begin{samepage}
\begin{lemma} Assume Hypothesis \ref{hypoH}.
\label{opPt}
\begin{enumerate}[ref=\arabic*, leftmargin=*, label=\arabic*.]
	\item \label{opPt001} There exists a constant $c>0$ such that, for any $t \in \bb R$, $n\geq 1$, and $h \in \mathscr{L}_{\theta,\ee}$,
	\[
	\abs{\mathbf{P}_t^n h}_{\theta} \leq c \abs{h}_{\theta}.
	\]
	\item \label{opPt002} There exist constants $c_1$, $c_2$ and $\rho<1$ such that, for any $n \geq 1$, $h \in \mathscr{L}_{\theta,\ee}$ and $t\in \bb R$,
	\[
	 \left[ \mathbf{P}_t^n h \right]_{\ee} \leq c_1 \rho^n \left[h\right]_{\ee} + c_2 \abs{t}^{\ee} \abs{h}_{\theta}.
	\]
	\item \label{opPt003} For any $t\in \bb R,$ the operator $\mathbf{P}_t$ is compact from $(\mathscr{B},\norm{\cdot}_{\theta,\ee})$ to $(\mathscr{C}\left( \bb R^d \right),\abs{\cdot}_{\theta})$.
\end{enumerate}
\end{lemma}
\end{samepage}

\begin{proof}
\textit{Claim \ref{opPt001}.} For any $x \in \bb R^d$,
\[
\abs{\mathbf{P}_t^n h(x)} = \abs{\bb E_x \left( \e^{itS_n} h \left( X_n \right) \right)} \leq 3^{\theta} \abs{h}_{\theta} \left( 1+\bb E \left( \abs{\Pi_n}^{\theta} \right) \abs{x}^{\theta}+\bb E \left( \abs{X_n^0}^{\theta} \right) \right),
\]
with $\Pi_n = A_nA_{n-1} \dots A_1$ and $X_n^0 = g_n \dots g_1 \cdot 0 = \sum_{k=1}^n A_n \dots A_{k+1}B_k$. By the point \ref{H1} of Hypothesis \ref{hypoH}, there exist $c(\delta)> 0$ and $0<\rho(\delta)<1$ such that, for any $n\geq 1$,
\[
\bb E^{\frac{2+2\delta}{\theta}} \left( \abs{\Pi_n}^{\theta} \right) \leq \bb E \left( \abs{\Pi_n}^{2+2\delta} \right) \leq c(\delta) \rho(\delta)^{n} \underset{n\to +\infty}{\longrightarrow} 0,
\]
from which it follows that
\[
\bb E \left( \abs{X_n^0}^{\theta} \right) \leq \left( \sum_{k=1}^n \bb E^{1/\theta} \left( \abs{\Pi_n}^{\theta} \right) \bb E^{1/\theta} \left( \abs{B_1}^{\theta} \right) \right)^{\theta}  < +\infty.
\]
This proves the claim \ref{opPt001}.

\textit{Proof of the claim \ref{opPt002}.} For any $x\neq y \in \bb R^d$, with $\abs{x} \leq \abs{y}$, we have
\begin{align*}
	&\abs{\mathbf{P}_t^n h(x) - \mathbf{P}_t^n h(y)} \\
	&\quad \leq \bb E \left( 2^{1-\ee} \abs{t}^{\ee} \abs{u}^{\ee} \left( \sum_{k=1}^n \abs{\Pi_k} \right)^{\ee} \abs{x-y}^{\ee} \abs{h}_{\theta} \left( 1+ \abs{\Pi_n} \abs{x} + \abs{X_n^0} \right)^{\theta} \right)\\
	&\quad \qquad+ \bb E \left( \left[h\right]_{\ee} \abs{\Pi_n}^{\ee} \abs{x-y}^{\ee} \left( 1+\abs{\Pi_n}\abs{x}+\abs{X_n^0} \right)\left( 1+\abs{\Pi_n}\abs{y}+\abs{X_n^0} \right) \right).
\end{align*}
Since $\theta < 2$, we obtain that
\[
\left[\mathbf{P}_t^n h\right]_{\ee} \leq 2^{1-\ee} \abs{t}^{\ee} \abs{u}^{\ee} C_2(n) \abs{h}_{\theta} + C_1(n) \left[h\right]_{\ee},
\]
where 
\[
C_1(n) = \bb E \left( \abs{\Pi_n}^{\ee} \left( 1+\abs{\Pi_n}+\abs{X_n^0} \right)^2 \right)
\]
and
\[
C_2(n) = \bb E \left( \left( \sum_{k=1}^n \abs{\Pi_k} \right)^{\ee}  \left( 1+ \abs{\Pi_n}  + \abs{X_n^0} \right)^{\theta} \right).
\]
Since $2+2\ee < 2+2\delta=p$, by the H\"older inequality,
\begin{align*}
C_1(n) &\leq \bb E^{\frac{\ee}{1+\ee}} \left( \abs{\Pi_n}^{1+\ee} \right) \bb E^{\frac{1}{1+\ee}} \left( \left( 1+ \abs{\Pi_n}  + \abs{X_n^0} \right)^{2+2\ee} \right) \\
&\leq c(\delta)^{\frac{\ee}{p}} \rho(\delta)^{\frac{n\ee}{p}} 3^2 \left( 1+c(\delta)^{\frac{2}{p}}+\left( \frac{c(\delta)^{\frac{1}{p}} \bb E^{\frac{1}{p}} \left( \abs{B_1}^{p} \right)}{1-\rho(\delta)^{\frac{1}{p}}} \right)^2 \right),
\end{align*}
which shows that $C_1(n) $ converges exponentially fast to $0.$
In the same way, taking into account that $\theta < 2$ we show that $C_2(n)$ is bounded:
\begin{align*}
C_2(n) &\leq \left( \sum_{k=1}^n \bb E^{\frac{1}{1+\ee}} \left( \abs{\Pi_k}^{1+\ee} \right) \right)^{\ee} \bb E^{\frac{1}{1+\ee}} \left( \left( 1+ \abs{\Pi_n}  + \abs{X_n^0} \right)^{2+2\ee} \right) \\
&\leq \left( \frac{c(\delta)^{\frac{1}{p}}}{1-\rho(\delta)^{\frac{1}{p}}} \right)^{\ee} 3^2 \left( 1+c(\delta)^{\frac{2}{p}}+\left( \frac{c(\delta)^{\frac{1}{p}} \bb E^{\frac{1}{p}} \left( \abs{B_1}^{p} \right)}{1-\rho(\delta)^{\frac{1}{p}}} \right)^2 \right).
\end{align*}

\textit{Proof of the claim \ref{opPt003}.} Let $B$ be a bounded subset of $\mathscr{B}$, $(h_n)_{n\geq 0}$ be a sequence in $B$ and $K$ be a compact of $\bb R^d$. Using the claim \ref{opPt001}, it follows that, for any $x\in K$ and $n\geq 0$,
\[
\abs{\mathbf{P}_th_n(x)} \leq c\abs{h_n}_{\theta} \left( 1+\abs{x} \right)^{\theta} \leq c_{K},
\]
which implies that the set $\mathscr A=\{ \mathbf{P}_th_n,\, n\geq 0\}$ is uniformly bounded in $(\mathscr{C}\left( K  \right),\abs{\cdot}_{\infty})$, where $\abs{\cdot}_{\infty}$ is the supremum norm.
By the claims \ref{opPt001} and \ref{opPt002}, we have that, for any $x,y \in K$ and $n\geq 0$,
\[
\abs{\mathbf{P}_t h_n (x) - \mathbf{P}_t h_n (y)} \leq \left[ \mathbf{P}_t h_n \right]_{\ee} \abs{x-y}^{\ee} \left(1+\abs{x} \right)^{\theta} \left(1+\abs{y} \right)^{\theta} \leq c_{K} \norm{h_n}_{\mathscr{B}} \abs{x-y}^{\ee}
\]
and, thereby, the set $\mathscr A$ is uniformly equicontinuous. By the theorem of Arzel\` a-Ascoli, we conclude that $\mathscr  A$ is relatively compact in $(\mathscr{C}\left( K  \right),\abs{\cdot}_{\infty})$.
Using a diagonal extraction, we deduce that there exist a subsequence $(n_k)_{k\geq 1}$ and a function $\varphi \in \mathscr{C} ( \bb R^d )$ such that, for any compact $K \subset \bb R^d$,
\[
\sup_{x\in K} \abs{P_t h_{n_k}(x) - \varphi(x)} \underset{n\to+\infty}{\longrightarrow} 0.
\]
Moreover, by the claims \ref{opPt001} and \ref{opPt002}, for any $n \geq 1$ and $x\in \bb R^d$,
\[
\abs{P_t h_n (x)} \leq \abs{P_t h_n (0)} + \left[P_t h_n \right]_{\ee} \abs{x}^{\ee} \left( 1+\abs{x} \right) \leq c \abs{h_n}_{\theta} + c \norm{h_n}_{\mathscr{B}} \abs{x}^{\ee} \left( 1+\abs{x} \right).
\]
Since $B$ is bounded, we have $\abs{P_t h_n (x)} \leq c ( 1+\abs{x} )^{1+\ee}$, for any $x\in \bb R^d,$ as well as $\varphi(x) \leq c ( 1+\abs{x} )^{1+\ee}$, for any $x\in \bb R^d$. Consequently, for any $k \geq 1$ and $A > 0$,
\[
\sup_{x\in \bb R^d} \frac{\abs{P_t h_{n_k} (x) - \varphi(x)}}{\left(1+\abs{x}\right)^{\theta}} \leq \sup_{\abs{x} \leq A} \abs{P_t h_{n_k} (x) - \varphi(x)} + 2c \sup_{\abs{x} > A} \frac{\left(1+\abs{x}\right)^{1+\ee}}{\left(1+\abs{x}\right)^{\theta}}.
\]
Taking the limit as $k\to +\infty$ and then the limit as $A\to +\infty$, we can conclude that $\lim_{k\to+\infty} \abs{P_t h_{n_k} - \varphi}_{\theta} = 0$.

\end{proof}

\begin{samepage}
\begin{lemma} Assume Hypothesis \ref{hypoH}. 
\nobreak
\label{uneseulevp}
\begin{enumerate}[ref=\arabic*, leftmargin=*, label=\arabic*.]
	\item \label{uneseulevp001} The operator $\mathbf P$ has a unique invariant probability $\bs \nu$ which coincides with the distribution of the $\bb P$-a.s.\ convergent series $Z:= \sum_{k=1}^{+\infty} A_1 \dots A_{k-1} B_k$. Moreover, the unique eigenvalue of modulus $1$ of the operator $\mathbf{P}$ on $\mathscr{B}$ is $1$ and the associated eigenspace is generated by the function $e$: $x \mapsto 1$.
	\item \label{uneseulevp002} Let $t \in \bb R^*$. If $h \in \mathscr{B}$ and $z\in \bb C$ of modulus $1$ are such that
	\[
	\mathbf P_t h(x) = z h(x), \qquad x \in \supp(\bs \nu),
	\]
	then $h=0$ on $\supp(\bs \nu)$.
	\end{enumerate}
\end{lemma}
\end{samepage}

\begin{proof} We proceed as in Guivarc'h and Le Page \cite{guivarch_spectral_2008} and Buraczewski, Damek and Guivarc'h \cite{buraczewski_convergence_2009}. For any $g=(A,B) \in \GL\left( d, \bb R \right) \times \bb R^d$ and $x\in \bb R^d$, we set $g \cdot x = Ax+B$.

\textit{Proof of claim \ref{uneseulevp001}.} Since $k(\delta) < 1$, the series $\sum_{k} \bb E^{\frac{1}{2+2\delta}} ( \abs{A_1 \dots A_{k-1} B_k}^{2+2\delta} )$ converges and so the sequence $g_1 \dots g_n \cdot x = A_1 \dots A_n x + \sum_{k=1}^n A_1 \dots A_{k-1} B_k$ converges almost surely to $Z=\sum_{k=1}^{+\infty} A_1 \dots A_{k-1} B_k$ as $n\to+\infty$. 
Therefore, for any $\varphi \in \mathscr{B}$, the sequence $\varphi(g_1 \dots g_n \cdot x)$ converges to $\varphi(Z)$ almost surely 
as $n\to+\infty$. Moreover, since  $\abs{\varphi(x)} \leq \abs{\varphi}_{\theta} \left( 1+\abs{x} \right)^{\theta}$ and $\theta < 2 +2\delta$, the sequence $(\varphi(g_1 \dots g_n \cdot x))_{n\geq1}$ is uniformly integrable. So $\mathbf{P}^n\varphi (x)$ converges to $\bb E(\varphi(Z))$ as $n\to+\infty$. 
This proves that the distribution $\bs \nu$ of $Z$ is the only invariant probability of $\mathbf P$. 

Fix $z \in \bb C$ such that $\abs{z}=1$ and let $h \neq 0$ belonging to $\mathscr{B}$ be an eigenfunction of $\mathbf{P}$, so that $\mathbf{P}h = zh$. From the previous argument, it follows that, for any $x \in \bb R^d$,
\[
z^n h(x) = \mathbf{P}^nh(x) \underset{n\to +\infty}{\longrightarrow} \bs \nu(h).
\]
Since there exists $x \in \bb R^d$ such that $h(x) \neq 0$, the sequence $(z^n)_{n\geq 1}$ should be convergent 
which is possible only if $z=1$. 
From this, we deduce that for any $x\in \bb R^d$, $h(x) = \bb E (h(Z))$ which implies that $h$ is constant.

\textit{Proof of the claim \ref{uneseulevp002}.} Our argument is by contradiction. 
Let $t \in \bb R^* $, $h \in \mathscr B$ and $z\in \bb C$ of modulus $1$ be such that $\mathbf P_t h (x) = z h(x),$ for any $x \in \supp (\bs\nu)$
and suppose that there exists $x_0\in \supp(\bs \nu)$ such that $h(x_0)\not=0$.

First we establish that $\abs{h}$ is constant on the support of the distribution $\bs \nu$. 
Since $\bs \nu$ is $\bs \mu$-invariant, for any $(g,x) \in \supp(\bs \mu) \times \supp(\bs \nu)$ we have $g\cdot x \in \supp(\bs \nu).$  
From this fact it follows that $\mathbf{P}_t^n h(x)=z^n h(x)$, for any $n\geq 1$ and $x\in \supp(\bs \nu)$. This implies that $\abs{h}(x) \leq \mathbf{P}^n \abs{h}(x),$ for any $x\in \supp(\bs \nu)$.  Note also that $\abs{h}$ belongs to $\mathscr{B}$.  Therefore, 
as we have seen in the proof of the first claim,
we have, $\lim_{n\to +\infty} \mathbf{P}^n \abs{h}(x) = \bs \nu( \abs{h} ) = \bb E( \abs{h} (Z) ) <+\infty$,  for any $x \in \supp (\bs \nu)$.  
So  $\abs{h}(x) \leq \int_{x' \in \bb R^d} \abs{h}(x') \bs \nu( \dd x')$, 
for any $x\in \supp (\bs \nu)$.
Since $\abs{h}$ is continuous, this implies that $\abs{h}$ is constant on the support of $\bs \nu$. In particular, this means that $h(x)\not=0$ for any $x\in \supp (\bs \nu)$.

Since the support of $\bs \nu$ is stable by all the elements of the support of $\bs \mu$, we deduce that the random variable $\xi_n(x) =\exp( it \scal{u}{\sum_{k=1}^n g_k \dots g_1 \cdot x}) h( g_n \dots g_1 \cdot x)$ takes values on the sphere $\bb S_{\nu(\abs{h})} = \{ a \in \bb C: \abs{a} = \bs \nu(\abs{h}) \}$, for all $x$ in the support of $\bs \nu$.  Moreover, the mean $z^n h(x)$ of $\xi_n(x)$ is also on $\bb S_{\nu(\abs{h})}$, which is possible only if $\xi_n(x)$ is a constant, for any $x\in \supp (\bs \nu)$. Consequently, for any pair $x,y \in \supp (\bs \nu)$, there exists an event $\Omega_{x,y}$ of $\bb P$-probability one such that on $\Omega_{x,y}$ it holds, for any $n\geq 1$,
\[
\exp \left( it \scal{u}{\sum_{k=1}^n g_k \dots g_1 \cdot v} \right) h\left( g_n \dots g_1 \cdot v \right) = z^n h(v),
\]
with $v \in \{x,y\}$, from which we get
\begin{equation}
\label{hpresqueconst002}
\frac{h\left( g_n \dots g_1 \cdot y \right)}{h\left( g_n \dots g_1 \cdot x \right)}
= \frac{h(y)}{h(x)} \exp\left(it \scal{\sum_{k=1}^n {}^t\!A_1 \dots {}^t\!A_k u}{ x-y } \right).
\end{equation}
In addition, for any $n \geq 1$,
\[
\bb E \left( \abs{ \frac{h\left( g_n \dots g_1 \cdot y \right)}{h\left( g_n \dots g_1 \cdot x \right)} -1 } \right) = \bb E \left( \abs{ \frac{h\left( g_1 \dots g_n \cdot y \right)}{h\left( g_1 \dots g_n \cdot x \right)} -1 } \right).
\]
Since, for $v\in\{x,y\}$, the sequence $h(g_1 \dots g_n \cdot v)$ converges a.s.\ to $h(Z)$ and since $h$ is bounded with a constant modulus, we have by \eqref{hpresqueconst002},
\begin{align*}
	0 &= \underset{n\to+\infty}{\lim} \bb E \left( \abs{ \frac{h\left( g_n \dots g_1 \cdot y \right)}{h\left( g_n \dots g_1 \cdot x \right)} -1 } \right) \\
	&= \underset{n\to+\infty}{\lim} \bb E \left( \abs{ \frac{h(y)}{h(x)} \exp\left(it \scal{\sum_{k=1}^n {}^t\!A_1 \dots {}^t\!A_k u}{ x-y } \right) -1 } \right).
\end{align*}
Taking into account that the series $\sum_{k=1}^n {}^t\!A_1 \dots {}^t\!A_k$ converges a.s.\ to a random variable $Z'$, we have for any $x,y \in \supp(\bs \nu)$,
\begin{equation}
\label{hpresqueconst003}
\bb E \left( \abs{ \frac{h(y)}{h(x)} \e^{it \scal{Z' u}{ x-y }} -1 } \right) = 0.
\end{equation}

Since the support of $\bs \nu$ is invariant by all the elements of the support of $\bs \mu$, by the point \ref{H2} of Hypothesis \ref{hypoH}, we deduce that the support of $\bs \nu$ is not contained in an affine subspace of $\bb R^d$, \textit{i.e.}\ for any $1 \leq j \leq d$, there exist $x_j,y_j \in \supp(\bs \nu)$, such that the family $(v_j)_{1\leq j \leq d} = (x_j-y_j)_{1\leq j \leq d}$ generates $\bb R^d$. From \eqref{hpresqueconst003}, we conclude that for any $1\leq j \leq d$,
\[
\frac{h(y_j)}{h(x_j)} \e^{it \scal{Z' u}{ v_j }} = 1, \qquad \bb P\text{-a.s.}
\]

Let $\theta_j$ be such that $\frac{h(x_j)}{h(y_j)} = \e^{i\theta_j}$. Denoting by $\bs \eta_u$ the distribution of $Z'u$, we obtain that $\scal{Z'u}{v_j} \in \frac{\theta_j +2\pi \bb Z}{t}$ $\bb P$-a.s.\ and so the support of $\bs \eta_u$ is discrete. Moreover, the measure $\bs \eta_u$ is invariant for the Markov chain $X_{n+1}' = {}^t\!A_{n+1} ( X_n' + u )$ and so, for any Borel set $B$ of $\bb R^d$,
\begin{equation}
	\label{etaB}
	\bs \eta_u \left( B \right) = \bb E \left( \int_{v \in \bb R^d} \mathbbm 1_{B} \left( {}^t\!A_1 \left( v +u \right) \right) \bs \eta_u(\dd v) \right).
\end{equation}
Since $\bs \eta_u$ is discrete, the set $E_{max} = \{ x\in \bb R^d: \bs \eta_u \left( \{x\} \right) = \max_{y\in \bb R^d} \bs \eta_u \left( \{y\} \right) \}$ is non-empty and finite. Moreover, using \eqref{etaB} with $B=\left\{ x\right\}$ and $x\in E_{max},$ we can see that the image ${}^tA_1^{-1} x-u$ belongs to $E_{max}$ $\bb P$-a.s. Denoting by $v_0$ the barycentre of $E_{max}$, we find that
\[
\bb P \left( {}^tA_1^{-1}  v_0-u  = v_0 \right) = 1.
\]
The fact that $u \neq 0$ implies that $v_0\neq 0$. 
The latter implies that ${}^tA_1^{-1} v_0 = v_0 +u = {}^tA_2^{-1} v_0$ almost surely, 
which contradicts the point \ref{H3} of Hypothesis \ref{hypoH}.
\end{proof}

The conditions (b), (c) and (d) of the theorem of Ionescu-Tulcea and Marinescu as stated in Chapter 3 of Norman \cite{norman1972markov}
follow from points \ref{opPt001}-\ref{opPt003} of Lemma \ref{opPt} repectively.
It remains to show the condition (a).
Let $\left( h_n \right)_{n\geq 0}$ be a sequence in $\mathscr{L}_{\theta,\ee}$ satisfying $\norm{h_n}_{\theta,\ee} \leq K$,  for any $n\geq 0$
and some constant $K$
and suppose that there exists $h\in \mathscr{C}( \bb R^d )$ such that $\lim_{n\to+\infty} \abs{h_n-h}_{\theta} = 0$.
For any $x,y,z \in \bb R^{d} $ and $n\geq 0,$
\begin{eqnarray*}
&& \frac{\abs{h(x)-h(y)}}{\abs{x-y}^{\ee} (1+\abs{x})(1+\abs{y})} + \frac{\abs{h(z)}}{(1+\abs{z})^{\theta}}\\
&&\leq 
 \abs{h_n-h}_{\theta} \left(      \frac{(1+\abs{x})^{\theta} + (1+\abs{y})^{\theta}}{\abs{x-y}^{\ee} (1+\abs{x})(1+\abs{y})} +1    \right)
+
\left[ h_n \right]_{\ee} + \abs{ h_n }_{\theta}.
\label{IT001}
\end{eqnarray*}
Taking the limit as $n \to +\infty,$ shows that
 $h \in \mathscr{L}_{\theta,\ee}$ and $\norm{h}_{\theta,\ee} \leq K$.

The theorem of Ionescu-Tulcea and Marinescu and the unicity of the one-dimen\-sional projector proved in the point \ref{uneseulevp001} of  Lemma \ref{uneseulevp}  imply Hypothesis \ref{SPGA}.
Hypothesis \ref{PETO} is obtained easily from  Lemma \ref{opPt}.

The point \ref{uneseulevp002} of  Lemma \ref{uneseulevp} will be used latter to prove that $\sigma^2 > 0.$

\textit{Proof of \ref{Momdec}.} By the hypothesis $\alpha=\frac{2+2\delta}{1+\ee}>2$. 
Consider the function $N$: $\bb R^d \to \bb R_{+}$ defined by $N(x) = \abs{x}^{1+\ee}$. 
For any $x,y \in \bb R^d$ satisfying $\abs{x} \leq \abs{y}$,
\[
\abs{N(x) - N(y)} \leq (1+\ee) \abs{y}^{\ee} \abs{x-y}.
\]
Using the fact that $\abs{N(x)-N(y)} \leq 2 \abs{y}^{1+\ee}$, we have
\[
\abs{N(x) - N(y)} \leq (1+\ee)^{\ee} 2^{1-\ee} \abs{y}^{\ee^2+(1+\ee)(1-\ee)} \abs{x-y}^{\ee} = c_{\ee} \abs{y} \abs{x-y}^{\ee}.
\]
Together with $\abs{N}_{\theta} < +\infty$, this proves that the function $N$ is in $\mathscr{B}=\mathscr{L}_{\theta,\ee}$.

Obviously $\abs{f(x)}^{1+\ee} = \abs{ \scal{u}{x} }^{1+\ee} \leq \abs{u}^{1+\ee} \left( 1+N(x) \right)$. Moreover, for any $h\in \mathscr{L}_{\theta,\ee}$,
\[
\abs{h(x)} \leq \left[ h \right]_{\ee} \abs{x}^{\ee} \left( 1+\abs{x} \right) + \abs{h(0)} \leq 2 \norm{h}_{\theta,\ee} \left( 1+N(x) \right)
\]
and so $\norm{\bs \delta_x}_{\mathscr B'} \leq 2 \left( 1+N(x) \right)$. Note that for any $p\in [1,\alpha]$,
\[
\bb E^{1/p} \left( N\left( g_n \dots g_1 \cdot x \right)^{p} \right) \leq 2^{1+\ee} \left( \bb E^{1/p} \left( \Pi_n^{p(1+\ee)} \right)N(x) + \bb E^{1/p} \left( \abs{g_n \dots g_1 \cdot 0}^{p(1+\ee)} \right) \right).
\]
Since $p(1+\ee) \leq 2+2\delta$, the previous inequality proves that $\bb E_x^{1/p} \left( N\left( X_n \right)^p \right) \leq c \left( 1+N(x) \right)$.
Thus, we proved the first inequality of the point \ref{Momdec001} of \ref{Momdec}.
  
For any $l \geq 1$, we consider the function $\phi_l$ on $\bb R_+$ defined by:
\[
\phi_l(t) = \left\{
\begin{array}{ll}
	0 & \text{ if } t\leq l^{\frac{1}{1+\ee}} - 1, \\
	t- \left( l^{\frac{1}{1+\ee}} - 1 \right) & \text{ if } t\in \left[ l^{\frac{1}{1+\ee}} - 1, l^{\frac{1}{1+\ee}} \right], \\
	1 & \text{ if } t\geq l^{\frac{1}{1+\ee}}.
\end{array}
\right.
\]
Define $N_l$ on $\bb R^d$ by $N_l(x) = \phi_l(\abs{x})N(x)$. For any $x\in \bb R^d$, we have $N(x) \mathbbm 1_{\left\{ N(x) > l \right\}} \leq N_l(x) \leq N(x)$ which implies that $\abs{N_l}_{\theta} \leq \abs{N}_{\theta} < +\infty$. Moreover, for any $x,y \in \bb R^d$ satisfying $\abs{x} \leq \abs{y}$, we have
\[
\abs{\phi_l(\abs{y}) - \phi_l(\abs{x})} \leq \min\left( \abs{y} - \abs{x}, 1 \right).
\]
So
\[
\abs{N_l(y) - N_l(x)} \leq \left[ N \right]_{\ee} \abs{x-y}^{\ee} \left( 1+\abs{x} \right) \left( 1+\abs{y} \right) + \abs{x}^{1+\ee} \abs{y-x}^{\ee}.
\]
Since $\abs{x} \leq \abs{y}$, we obtain that $\left[ N_l \right]_{\ee} \leq \left[ N \right]_{\ee} + 1 <+\infty$. Therefore, the function $N_l$ belongs to $\mathscr{B} = \mathscr{L}_{\theta,\ee}$, which finishes the proof of the point \ref{Momdec001} of \ref{Momdec}. 

Moreover, $\norm{N_l}_{\theta,\ee} \leq \norm{N}_{\theta,\ee}+1$ and, so the point \ref{Momdec002} of \ref{Momdec} is also established. 

Since $\int_{\bb X} \abs{x}^p \bs \nu(\dd x) < +\infty$, for any $p\leq 2+2\delta$, we find that
\[
\bs \nu \left( N_l \right) \leq \int_{\bb X} \abs{x}^{1+\ee} \mathbbm 1_{\left\{\abs{x} \geq l^{\frac{1}{1+\ee}}-1 \right\}} \bs \nu(\dd x) \leq  \frac{\int_{\bb X} \abs{x}^{2+2\delta} \bs \nu(\dd x)}{\left( l^{\frac{1}{1+\ee}}-1 \right)^{2+2\delta-(1+\ee)}}.
\]
Choosing $\beta = \alpha - 2 > 0$, we obtain the point \ref{Momdec003} of \ref{Momdec}.

\textit{Proof of \ref{CECO}.} Using \eqref{mu-sigma001} and the point \ref{H4} of Hypothesis \ref{hypoH},
\begin{equation}
	\label{muestnul}
	\mu = \int_{\bb R^d} \scal{u}{x} \bs \nu(\dd x) = \scal{u}{\bb E\left( \sum_{k=1}^{+\infty} A_1 \dots A_{k-1} B_k \right)} = 0.
\end{equation}
Now we prove that $\sigma^2 > 0$. For this, suppose the contrary: $\sigma^2 = 0$. One can easily check that the function $f$ belongs to $\mathscr{B}$. Using \ref{SPGA} and the fact that $\nu (f) = \mu = 0$, we deduce that $\sum_{n\geq 0} \norm{\mathbf{P}^n f}_{\theta,\ee}$ = $\sum_{n\geq 0} \norm{Q^n f}_{\theta,\ee} < +\infty$ and therefore the series $\sum_{n\geq 0} \mathbf{P}^n f$ converges in $\left(\mathscr{B},\norm{\cdot}_{\theta,\ee}\right)$. We denote by $\Theta \in \mathscr{B}$ its limit and notice that the function $\Theta$ satisfies the Poisson equation: $\Theta - \mathbf{P}\Theta = f$.

Using the bound \eqref{bound_EfXn}, we have $\abs{\sum_{n=1}^N f(x) \mathbf{P}^n f(x)} \leq c \left( 1+ N(x) \right)^2$. By the Lebesgue dominated convergence theorem, from  \eqref{mu-sigma001}, we obtain
\begin{align*}
	\sigma^2 &= \int_{\bb R^d} f(x) \left( 2\Theta(x) - f(x) \right) \bs \nu(\dd x) \\
	&= \int_{\bb R^d} \left( \Theta^2(x) - \left( \mathbf{P}\Theta \right)^2(x) \right) \bs \nu(\dd x)	\\
	&= \int_{\GL\left( d, \bb R \right) \times \bb R^d \times \bb R^d} \left( \Theta(g_1 \cdot x) - \mathbf{P}\Theta(x) \right)^2 \bs \mu(\dd g_1) \bs \nu(\dd x).
\end{align*}
As $\sigma^2=0$, we have $\Theta(g_1 \cdot x) = \mathbf{P}\Theta(x)$, \textit{i.e.}\ $f(g_1 \cdot x) = \mathbf{P}\Theta(x) - \mathbf{P}\Theta(g_1 \cdot x)$,
$\bs \mu \times \bs \nu$-a.s. Consequently, there exists a Borel subset $B_0$ of $\bb R^d$ such that $\bs \nu(B_0)=1$ and for any $t\in \bb R$ and $x\in B_0$,
\[
\int_{\GL\left( d, \bb R \right) \times \bb R^d} \e^{it\scal{u}{g_1 \cdot x}} \e^{it \mathbf{P}\Theta(g_1 \cdot x)} \bs \mu (\dd g_1) = \e^{it \mathbf{P}\Theta(x)}.
\]
Since the functions in the both sides are continuous, this equality holds for every $x\in \supp(\bs \nu)$. Since $\Theta \in \mathscr{L}_{\theta,\ee}$, the function $x \mapsto \e^{it \mathbf{P}\Theta(x)}$ belongs to $\mathscr{L}_{\theta,\ee} \smallsetminus \{0\}$. This contradicts the point \ref{uneseulevp002} of Lemma \ref{uneseulevp} and we conclude that $\sigma^2>0$ and so \ref{CECO} holds true.

\subsection{Proof of Proposition \ref{PP002}}
\label{proof-cas-compact}

We show that \ref{BASP}-\ref{CECO} hold true for the Markov chain $(X_n)_{n\geq 1}$, the function $f$ and the Banach space $\mathscr{L}(\bb X)$ given in Section \ref{Compact1}.

\textit{Proof of \ref{BASP}.} Obviously the Dirac measure belongs to $\mathscr{L}(\bb X)'$ and $\norm{\bs \delta_x}_{\mathscr{L}(\bb X)'} \leq 1$ for any $x\in \bb X$. For any $h \in \mathscr{L}(\bb X)$ and $t\in \bb R$ the function $\e^{itf}h$ belongs to $\mathscr{L}(\bb X)$ and 
\begin{equation}
	\label{opxhitborne}
	\norm{\e^{itf}h}_{\mathscr{L}} \leq \abs{t} \left[ f \right]_{\bb X} \norm{h}_{\infty} + \norm{h}_{\mathscr{L}} \leq \left( \abs{t} \left[ f \right]_{\bb X} + 1 \right) \norm{h}_{\mathscr{L}}.
\end{equation}

\textit{Proof of \ref{SPGA}.}
Let $(x_1,x_2)$ and $(y_1,y_2)$ be two elements of $\bb X$ and $h \in \mathscr{L}(\bb X)$. Since
\[
\mathbf Ph(x_1,x_2) = \int_{X} h(x_2,x') P(x_2, \dd x'),
\] 
we have $\norm{\mathbf Ph}_{\infty} \leq \norm{h}_{\infty}$. 
Denote by $h_{x_2}$ the function $z \mapsto h(x_2,z)$, which is an element of $\mathscr{L}(X).$ 
Since $\left[ h_{x_2} \right]_{X} \leq \left[ h \right]_{\bb X}$ and $\abs{h_{x_2}}_{\infty} \leq \norm{h}_{\infty}$, we obtain also that
\begin{align*}
	\abs{\mathbf Ph(x_1,x_2) - \mathbf Ph(y_1,y_2)} &= \abs{Ph_{x_2}(x_2) - Ph_{y_2}(y_2)} \\
	&\leq \left[ Ph_{x_2} \right]_{X}d_{X}(x_2,y_2) + \left[ h \right]_{\bb X} d_{X}(x_2,y_2) \\
	&\leq \left( \abs{P}_{\mathscr{L} \to \mathscr{L}}\norm{h}_{\bb X} + \left[ h \right]_{\bb X} \right) d_{X}(x_2,y_2),
\end{align*}
where $\abs{P}_{\mathscr{L} \to \mathscr{L}}$ is the norm of the operator $P$: $\mathscr{L}(X) \to \mathscr{L}(X)$. 
Therefore $\mathbf P$ is a bounded operator on $\mathscr{L}(\bb X)$ and
$\norm{\mathbf P}_{\mathscr{L} \to \mathscr{L}} \leq \left( 1+\abs{P}_{\mathscr{L} \to \mathscr{L}} \right).$
Now, for any $h \in \mathscr{L}(\bb X)$, we define the function $F_h$ by
\[
F_h (x_2) := \int_{X} h(x_2,x') P(x_2,\dd x') = \mathbf Ph(x_1,x_2).
\]
Notice that $F_h$ belongs to $\mathscr{L}(X)$ and $\abs{F_h}_{\mathscr{L}} \leq \norm{\mathbf Ph}_{\mathscr{L}}$. So by Proposition \ref{TroupourPpasgras}, for any $n\geq 2$, $(x_1,x_2) \in \bb X$ and $h \in \mathscr{L}(\bb X)$,
\[
\mathbf P^n h(x_1,x_2) = P^{n-1}F_h (x_2) = \nu(F_h) + R^{n-1}F_h (x_2) = \bs \nu(h) e(x_1,x_2) + Q^{n}h(x_1,x_2),
\]
where the probability $\bs \nu$ is defined on $\bb X$ by
\[
\bs \nu(h) = \nu( F_h )  = \int_{X\times X} h(x',x'') P(x',\dd x'') \nu(\dd x'), 
\]
the function $e$ is the unit function on $\bb X$, $e(x_1,x_2) = 1$, $\forall (x_1,x_2) \in \bb X$ and $Q$ is the linear operator on $\mathscr{L}(\bb X)$ defined by $Qh=R(F_h)=\mathbf P h - \bs \nu(h)$. By Proposition \ref{TroupourPpasgras}, the operator $Q$ is bounded and for any $n\geq 1$, $\norm{Q^n}_{\mathscr{L} \to \mathscr{L}} \leq \abs{R^{n-1}}_{\mathscr{L} \to \mathscr{L}} \norm{\mathbf P}_{\mathscr{L} \to \mathscr{L}} \leq C_{Q} \kappa^n$. Since $\nu$ is invariant by $P$, one can easily verify that $\Pi Q = Q \Pi = 0$, where $\Pi$ is the one-dimensional projector defined on $\mathscr{L}(\bb X)$ by $\Pi h = \nu(h) e$.

\textit{Proof of \ref{PETO}.} For any $t\in \bb R$, $h\in \mathscr{L}(\bb X)$ and $(x_1,x_2) \in \bb X$,
\[
\mathbf P_th(x_1,x_2) = \int_X \e^{itf(x_2,x')} h(x_2,x') P(x_2,\dd x') = \sum_{n=0}^{+\infty} \frac{i^n t^n}{n!} L_n(h)(x_1,x_2),
\]
where $L_n(h) = \mathbf P (f^n h)$. Since $\left( \mathscr{L}(\bb X), \norm{\cdot}_{\mathscr{L}} \right)$ is a Banach algebra, it follows that $L_n$ is a bounded operator on $\mathscr{L}(\bb X)$ and $\norm{L_n}_{\mathscr{L} \to \mathscr{L}} \leq \norm{\mathbf P}_{\mathscr{L} \to \mathscr{L}} \norm{f}_{\mathscr{L}}^n$. Consequently, the application $t \mapsto \mathbf P_t$ is analytic on $\bb R$ and so, by the analytic perturbation theory of linear operators (see \cite{kato_perturbation_1976}), there exists $\ee_0 > 0$ such that, for any $\abs{t} \leq \ee_0$,
\[
\mathbf P_t^n = \lambda_t^n \Pi_t + Q_t^n,
\] 
where $\lambda_t$ is an eigenvalue of $\mathbf P_t$, $\Pi_t$ is the projector on the one-dimensional eigenspace of $\lambda_t$ and $Q_t$ is an operator of spectral radius $r(Q_t) < \abs{\lambda_t}$ such that $\Pi_t Q_t = Q_t \Pi_t = 0$. The functions $t \mapsto \lambda_t$, $t\mapsto \Pi_t$ and $t \mapsto Q_t$ are analytic on $[-\ee_0,\ee_0]$. Furthermore, for any $h\in \mathscr{L}(\bb X)$ and $(x_1,x_2) \in \bb X$,
\[
\abs{\mathbf P_t h}(x_1,x_2) = \abs{\int_X \e^{itf(x_2,x')} h(x_2,x') P(x_2,\dd x')} \leq \norm{h}_{\infty}
\]
and necessarily $\abs{\lambda_t} \leq 1$, for any $\abs{t} \leq \ee_0$. Consequently
\[
\sup_{\abs{t} \leq \ee_0, n \geq 1} \norm{\mathbf P_t^n}_{\mathscr{L} \to \mathscr{L}} \leq c.
\]

\textit{Proof of \ref{Momdec} and \ref{CECO}.} Since for any $x\in \bb X$, $\abs{f(x)} \leq \abs{f}_{\infty}$ and $\norm{\bs \delta_x}_{\mathscr{L}(\bb X)'} \leq 1$, we can choose $N=0$ and $N_l = 0$ for any $l\geq 1$ and Hypothesis \ref{Momdec} is obviously satisfied.

Finally, Hypothesis \ref{hypoHcompactII} ensures that \ref{CECO} holds true.

\bibliographystyle{plain}
\bibliography{biblio5}

\end{document}